\documentclass[reqno,11pt]{amsart}
\usepackage{amssymb}
\usepackage{}
\usepackage[latin1]{inputenc} 
\usepackage{listings, tcolorbox}

\usepackage{graphicx}  
\usepackage{float}  
\usepackage{subfigure}  
\usepackage{amssymb, amsmath, amsthm}
\usepackage[colorlinks=true,linkcolor=blue,citecolor=red]{hyperref}
\usepackage{color}
\usepackage{float}
\usepackage{enumerate}
\usepackage{graphics}
\usepackage{epstopdf}
\usepackage{tikz}
\usepackage{cite}
\usepackage{epsfig}
\usepackage{amsmath}
\usepackage{amssymb}
\usepackage{amscd}
\usepackage{graphicx}
\usepackage{mathrsfs}
\usepackage[T1]{fontenc}

\numberwithin{equation}{section}
\numberwithin{equation}{section}
\newtheorem{theorem}{Theorem}[section]
\newtheorem{corollary}[theorem]{Corollary}
\newtheorem{lemma}[theorem]{Lemma}
\newtheorem{proposition}[theorem]{Proposition}

\newtheorem{Dbarproblem}[theorem]{$\bar{\partial}$-Problem}
\newtheorem{dbar-RHP}[theorem]{$\bar{\partial}$-RH problem}
\theoremstyle{definition}

\newtheorem{remark}[theorem]{Remark}
\newtheorem{RHP}[theorem]{RH problem}
\usepackage{epsfig}

\numberwithin{equation}{section}

\newcommand{\e}{\epsilon}

\newcommand{\R}{\mathbb{R}}
\newcommand{\C}{\mathbb{C}}

\allowdisplaybreaks[4]

\begin{document}
	
	\title[On Cauchy problem to the modified Camassa-Holm equation]{On Cauchy problem to the modified Camassa-Holm equation: Painlev\'{e} asymptotics}
	

	\author[J. F. Tong]{Jia-Fu Tong$ $}
	
	\author[S. F. Tian]{Shou-Fu Tian$^{*}$}

\address{Jia-Fu Tong \newline
		School of Mathematics, China University of Mining and Technology, Xuzhou 221116, China}
\email{jftong@cumt.edu.cn}	
\address{Shou-Fu Tian (Corresponding author) \newline
		School of Mathematics, China University of Mining and Technology, Xuzhou 221116, China}
\email{sftian@cumt.edu.cn, shoufu2006@126.com }
	
	\thanks{$^{*}$Corresponding author(sftian@cumt.edu.cn, shoufu2006@126.com). This author is contributed equally as the first author.}
	
\begin{abstract}
		{We investigate the Painlev\'{e} asymptotics for the Cauchy problem of the modified Camassa-Holm (mCH) equation with decaying initial data
\begin{align*}\nonumber
&m_t+\left((u^2-u_x^2)m\right)_x+\kappa u_{x}=0, \ (x,t)\in\mathbb{R}\times\mathbb{R}^+,\\
&u(x,0)=u_0(x),
\end{align*}
where $u_0(x)\in H^{4,2}(\mathbb{R})$ and $\kappa$ is a constant. Recently, Yang and Fan (Adv. Math. 402 (2022) 108340) reported the long-time asymptotic results for the mCH equation in the different solitonic regions. The main purpose of our work is to study the asymptotic behavior of the mCH equation in the transition regions, which are the critical regions between the different solitonic regions. The key is to establish a connection between the solution for the Cauchy problem of the mCH equation in the transition region and the Painlev\'{e} II equation. With the $\bar{\partial}$-generalization of the Deift-Zhou nonlinear steepest descent method and double scaling
limit technique, in two transition regions defined by
\begin{align}\nonumber
\mathcal{P}_{I}:=\{(x,t):0\leqslant \left|\frac{x}{t}-2\right|t^{2/3}\leqslant C\},~~~~\mathcal{P}_{II}:=\{(x,t):0\leqslant \left|\frac{x}{t}+1/4\right|t^{2/3}\leqslant C\},
\end{align}
where $C>0$ is a constant, we obtain that the leading order approximation to the solution
of the mCH equation can be expressed in terms of the Painlev\'{e} II equation.}
\end{abstract}

	\maketitle
	\tableofcontents
	\section{Introduction}
We consider the Cauchy problem for the modified Camassa-Holm (mCH) equation with decaying initial data
\begin{subequations}\label{mch}
\begin{align}
	&m_{t}+\left(m\left(u^{2}-u_{x}^{2}\right)\right)_{x}+\kappa u_{x}=0, \quad m=u-u_{xx}, &x\in\mathbb{R}, \ t>0, \label{def:mCHeq}\\
	&u(x,0)=u_0(x),
\end{align}
\end{subequations}
for the function $u(x,t)$ of time $t$ and a single spatial variable $x$, in which $\kappa$ is a constant.
In an equivalent form, the mCH equation was given by Fokas \cite{boo-5}, Fuchssteiner \cite{boo-3}, Olver
and Rosenau \cite{boo-6} and Qiao \cite{boo-7}.
So the mCH equation \eqref{def:mCHeq} is also known as the FORQ equation, which was first identified as an integrable system by Fuchssteiner\cite{boo-3}. It is nowadays usually recognized as an integrable modification of the celebrated Camassa-Holm(CH) equation \cite{boo-1,boo-2}
\begin{equation}
m_t+\left(um\right)_x+u_xm=0,\quad m:=u-u_{xx},
\end{equation}
which was proposed by Fuchsstiener and Fokas in \cite{boo-4} as an integrable equation. This equation was first introduced by Camassa and Holm in \cite{boo-1} as a model for shallow water waves. The CH equation has undergone extensive research throughout the past two decades \cite{boo-8,boo-9,boo-10,boo-11,boo-12,boo-13,boo-14,boo-15,boo-16,boo-17,boo-18}. This is primarily attributed to its elaborate mathematical structure and its utility in modeling the unidirectional movement of shallow water waves across a flat seabed.

Subsequently, numerous modified or generalized versions of the CH equation were proposed, e.g., \cite{boo-19} and references therein. Novikov \cite{boo-20} employed the perturbative symmetry approach to classify integrable equations of the form
\begin{equation}\nonumber
\begin{pmatrix}1-\partial_x^2\end{pmatrix}u_t=F(u,u_x,u_{xx},u_{xxx},\ldots),\quad u=u(x,t),\quad\partial_x=\partial/\partial x,
\end{equation}
which $F$ stands for a homogeneous differential polynomial in $\mathbb{C}$. In terms of $u$ and its $x$-derivatives, this polynomial manifests as either quadratic or cubic \cite{boo-21}. In the set of equations documented in \cite{boo-20}, equation (32) stood out as the second equation exhibiting cubic nonlinearity, and it was in the form of \eqref{def:mCHeq}. Fokas \cite{boo-5} provided this equation in an equivalent manifestation (and additional details can be found in \cite{boo-6} and \cite{boo-3} as well). Shiff perceived equation \eqref{def:mCHeq} to be the dual of the modified Korteweg-de Vries (mKdV) equation in \cite{boo-22}. Subsequently, Qiao \cite{boo-7} introduced an alternative Lax pair for \eqref{def:mCHeq} , which was later shown to be gauge equivalent to the one proposed by Shiff. Due to these contributions and connections, the mCH equation is also known by the name of the Fokas-Olver-Rosenau-Qiao (FORQ) equation, as indicated in \cite{boo-23}. The mCH equations admit non-smooth solitons as solutions \cite{boo-24}. Qu and Liu \cite{boo-25,boo-26} demonstrated the stability and orbital stability of peakons in the context of the mCH equation. Algebro-geometric quasiperiodic solutions were successfully formulated through the application of the algebro-geometric method, as documented in \cite{boo-23}. By making use of the reciprocal transformation, both the B{\"a}cklund transformation and the nonlinear superposition formula for the mCH equation were presented in \cite{boo-27}. The well-posedness of the Cauchy problem associated with the mCH equation was studied \cite{boo-28,boo-29,boo-30}. Additionally, the local well-posedness and the exact details of the blow-up phenomena for the Cauchy problem of the mCH equation were explored and deliberated upon in \cite{boo-31,boo-32}. Boutet de Monvel, along with Karpenko and Shepelsky, first developed  a Riemann-Hilbert(RH) approach for dealing with the mCH equation \eqref{def:mCHeq} which has non-zero boundary conditions, as described in \cite{boo-37}, and studied the long-time asymptotic behavior of the solution in different regions \cite{boo-38}. Xu, Yang and Zhang investigate long time asymptotics of the mCH equation in three transition zones
under a nonzero background \cite{boo-33}. Recently, the long time asymptotic behavior in space-time solitonic regions for the initial value problem of the mCH equation \eqref{mch} was studies by Yang and Fan \cite{boo-36}. Besides, there are many results about mCH equation \cite{boo-34,boo-35}.

In investigation of the Painlev\'{e}-type equation, Flaschka and Newell solved the initial value problem of Painlev\'{e} II equation by solving an inverse problem of
the corresponding ordinary differential equation \cite{boo-39}. Additionally, the asymptotic behavior of the Painlev\'{e}-II equation has been explored in a series of literatures \cite{boo-40,boo-41}. The appearance of transition asymptotic regions for integrable systems, where the asymptotics is described in terms of Painlev\'{e} transcendents, as well as the connection between different regions was first understood in the case of the Korteweg-de Vries equation by Segur and Ablowitz \cite{boo-42}. Deift and Zhou discovered the relationship between the mKdV equation and the Painlev\'{e} equation \cite{boo-43}. Boutet de Monvel, Its, and Shepelsky employed the nonlinear steepest descent method to determine the Painlev\'{e}-type asymptotics of the CH equation \cite{boo-10}. Charlier and Lenells conducted a meticulous examination of the Airy and higher order Painlev\'{e} asymptotics of the mKdV equation \cite{boo-45}. Meanwhile, Huang and Zhang managed to derive the Painlev\'{e} asymptotics for the entire modified mKdV hierarchy, as reported in \cite{boo-46}. Recently, Wang and Fan discovered the Painlev\'{e}-type asymptotics in two transition regions for the defocusing nonlinear Schr{\"o}dinger equation with non-zero boundary conditions \cite{boo-47}.

We re-examine the Cauchy problem \eqref{mch} considered by Yang and Fan in \cite{boo-36}. They use the $\bar{\partial}$-generalization of the Deift-Zhou nonlinear steepest descent method to derive the leading order approximation to the solution of the Cauchy problem \eqref{mch} in the solitonic regions.
\begin{align}
u(x,t)&=u^r(x,t;\tilde{\mathcal{D}}_\Lambda)+f_{11}t^{-1/2}+\mathcal{O}(t^{-3/4}),\quad\mathrm{for}~~\xi\in(-1/4,2),\\
u(x,t)&=u^r(x,t;\tilde{\mathcal{D}}_\Lambda)+\mathcal{O}(t^{-1+2\rho}),\quad\mathrm{for}~~\xi\in(-\infty,-1/4)\cup(2,+\infty),
\end{align}
where $\xi$, $u^r(x,t;\tilde{\mathcal{D}}_\Lambda)$, $f_{11}$ and $\rho$ have appropriate definitions, please refer to reference \cite{boo-36} for details. They obtain the long-time asymptotic results of the Cauchy problem \eqref{mch} in different solitonic regions, and their results also confirm the soliton resolution conjecture for the mCH equation \eqref{mch}.

In fact, exactly as Wang and Fan \cite{boo-47} solved the asymptotic behavior of the transition regions for the defocusing nonlinear Schr{\"o}dinger equation with finite
density initial data proposed by Cuccagna and Jenkins \cite{boo-57}, the asymptotic behavior of the mCH equation \eqref{mch} in the transition regions is worth considering. However, how to describe the asymptotic behavior of the Cauchy problem \eqref{mch} in the transition regions between different solitonic regions, seems to remain
unknown to the best of our knowledge. Hence, the main purpose of our paper is to derive
the Painlev\'{e} asymptotics of the solution $u(x,t)$ to the Cauchy problem \eqref{mch} in
two transition regions (see Fig.\ref{fig1}) defined by
\begin{align}\label{P}
\mathcal{P}_{I}:=\{(x,t):0\leqslant\left|\frac{x}{t}-2\right|t^{2/3}\leqslant C\},~~~~\mathcal{P}_{II}:=\{(x,t):0\leqslant\left|\frac{x}{t}+1/4\right|t^{2/3}\leqslant C\},
\end{align}
where $C>0$ is a constant.
\begin{figure}[h]
	\begin{center}
		\begin{tikzpicture}
            \draw[yellow!20, fill=yellow!20] (0,0)--(4.5,0)--(4.5,1.7)--(0, 0);
            \draw[yellow!20, fill=purple!20] (0,0)--(4.5,1.7)--(4.5,2)--(0, 0);
            \draw[yellow!20, fill=purple!20] (0,0)--(4,2)--(4.5,2)--(0, 0);
            \draw[blue!20, fill=green!20] (0,0)--(4,2)--(0,2)--(0,0);
            \draw[blue!20, fill=green!20] (0,0)--(-4.5,0)--(-4.5,2)--(0, 2)--(0,0);
           \draw[yellow!20, fill=yellow!20] (0,0)--(-4.5,0)--(-4.5,1)--(0,0);
           \draw[yellow!20, fill=purple!20] (0,0)--(-4.5,0.7)--(-4.5,1.3)--(0,0);

		\draw [-> ](0,0)--(0,2.8);
       \draw [-> ](-5,0)--(5,0);
		\node    at (0.1,-0.3)  {$0$};
		\node    at (5.26,0)  { $x$};
		\node    at (0,3.2)  { $t$};

		 \node  [below]  at (4.4,2.7) {\small$\tau=2$};
		 \node  [below]  at (-5.4,1.1) {\small$\tau=-1/4$};

		\draw [red](0,0)--(-4.5,1);
		\draw [red](0,0)--(4.5,2);

     \node [] at (0,0.85) {Solitonic region};

        \node [] at (-2.8,0.2) {Solitonless region};

    \node [] at (2.8,0.2) {Solitonless region};
    \node [] at (5.4,1.7) {Transition region $\mathcal{P}_{I}$};
    \node [] at (-5.4,1.3) {Transition region $\mathcal{P}_{II}$};
		\end{tikzpicture}
	\end{center}
	\caption{ \small The different asymptotic space-time regions of $x$ and $t$, where $\tau\triangleq\frac{x}{t}$.}
	\label{fig1}
\end{figure}

To prove this result, we represent the solution $u(x,t)$ based on the solution of a RH problem and employ the $\bar{\partial}$-generalization of the Deift-Zhou nonlinear steepest descent method to treat it, which has been made in proving the conjecture of soliton resolution of many nonlinear integrable model \cite{boo-52,boo-53,boo-54,boo-55,boo-56,boo-57,boo-58}.\\

\textbf{Our paper is arranged as follows:}\\

In Section \ref{s:2}, we quickly reviewed some basic results, especially the construction of a basic
RH formalism $M(z)$ related to the Cauchy problem \eqref{mch} for the mCH equation. For more details, see \cite{boo-36}. After a brief review, the two main results of Theorem \ref{the1.3} are presented in Sections \ref{s:3} and \ref{s:4}, respectively.

In Section \ref{s:3}, we focus on the long-time asymptotic analysis for the mCH equation in the transition region $\mathcal{P}_{I}$ with the following steps. First of all, we obtain a basic RH problem for $M^{(2)}(z)$ by removing the poles far away from the critical line of the RH
problem for $M(z)$ in Subsection \ref{s:3.1}. We deforme the RH problem for $M^{(2)}(z)$ into a hybrid $\bar{\partial}$-RH problem for $M^{(3)}(z)$ by continuously extending the jump matrix, which can be solved by decomposing
it into a pure RH problem for $M^{rhp}(z)$ and a pure $\bar{\partial}$-problem $M^{(4)}$. In Subsection \ref{s:3.2}, the pure RH problem
for $M^{rhp}(z)$ can be constructed by a solvable Painlev\'e model via the local paramatrix
near the critical points $z=-1$ and $z=-1$, and an modified reflectionless RH problem
$M^{out}(z)$ for the soliton components. In Subsection \ref{s:3.3}, we prove the existence of the solution $M^{(4)}$ and estimate its size. In Subsection \ref{s:3.4}, we then obtain the Painlev\'{e} asymptotics of the mCH equation in the transition region $\mathcal{P}_{I}$.

In Section \ref{s:4}, we investigate the asymptotics of the solution in the transition region $\mathcal{P}_{II}$
using a similar way as Section \ref{s:3}. In Section \ref{s:5}, we provide the asymptotic behavior of the Cauchy problem \eqref{mch} for the mCH equation in two transition regions.

	\section{Preliminary knowledge: Inverse scattering transform}\label{s:2}
    This section presents key results concerning the inverse scattering transform related to the Cauchy problem \eqref{mch}. For in-depth details, readers may refer to \cite{boo-36}.

We fix some notations used this paper. $\sigma_1$, $\sigma_2$ and $\sigma_3$ denote the Pauli matrices
\begin{equation}\nonumber
\sigma_1=\begin{pmatrix}0&1\\1&0\end{pmatrix},\quad\sigma_2=\begin{pmatrix}0&-i\\i&0\end{pmatrix},\quad\sigma_3=\begin{pmatrix}1&0\\0&-1\end{pmatrix}.
\end{equation}
The weighted Sobolev space $H^{k,s}(\mathbb{R})$ is defined by
\begin{equation}\nonumber
H^{k,s}(\mathbb{R})=\left\{f(x)\in L^2(\mathbb{R})|(1+|x|^s)\partial^jf(x)\in L^2(\mathbb{R}),\mathrm{for}~~j=1,...,k\right\}.
\end{equation}
For a $2\times2$ matrix $A$, we define
\begin{equation}\nonumber
e^{\hat{\sigma}_3}A:=e^{\sigma_3}Ae^{-\sigma_3}.
\end{equation}
We write $f\lesssim g$ to denote the inequality $f\leqslant Cg$ for some constant $C>0$.
    \subsection{Spectral analysis on the Lax pair}\label{s:2.1}
    The mCH equation \eqref{mch} admits the Lax pair
    \begin{equation}\label{Lax}
    \Phi_x=X\Phi,\quad\Phi_t=T\Phi,
\end{equation}
where
\begin{equation}\nonumber
    X=-\frac{k}{2}\sigma_3+\frac{i\lambda m(x,t)}{2}\sigma_2,
\end{equation}
\begin{equation}\nonumber
    T=\frac{k}{\lambda^{2}}\sigma_{3}+\frac{k}{2}\left(u^{2}-u_{x}^{2}\right)\sigma_{3}-i\left(\frac{u-ku_{x}}{\lambda}+\frac{\lambda}{2}\left(u^{2}-u_{x}^{2}\right)m\right)\sigma_{2},
\end{equation}
with
\begin{equation}\nonumber
    k=k(z)=\frac{i}{2}(z-\frac{1}{z}),\quad\lambda=\lambda(z)=\frac{1}{2}(z+\frac{1}{z}).
\end{equation}

As outlined in \cite{boo-36}, the Lax pair \eqref{Lax} for the mCH equation exhibits singularities at $z=0$, $z=\infty$. In the extended complex $z$-plane there are also branch cut points at $z=\pm i$. Therefore, it is necessary to regulate the asymptotic behavior of its eigenfunctions.

$\textbf{Case I: }z=\infty.$\\
We define
\begin{equation}\nonumber
    F(x,t)=\sqrt{\frac{q+1}{2q}}\begin{pmatrix}1&\frac{-im}{q+1}\\\frac{-im}{q+1}&1\end{pmatrix},
\end{equation}
and
\begin{equation}\nonumber
    p(x,t,z)=x-\int_x^\infty(q-1)dy-\frac{2t}{\lambda^2},\quad q=\sqrt{m^2+1}.
\end{equation}
By carrying out a transformation
\begin{equation}\label{2.5}
    \Phi_\pm=F\mu_\pm e^{-\frac{i}{4}(z-\frac{1}{z})p\sigma_3},
\end{equation}
we obtain a new Lax pair
\begin{equation}\label{2.6}
    (\mu_\pm)_x=-\frac{i}{4}(z-\frac{1}{z})p_x[\sigma_3,\mu_\pm]+P\mu_\pm,
\end{equation}
\begin{equation}\label{2.7}
    (\mu_\pm)_t=-\frac{i}{4}(z-\frac{1}{z})p_t[\sigma_3,\mu_\pm]+L\mu_\pm,
\end{equation}
where
\begin{equation}\nonumber
    P=\frac{im_x}{2q^2}\sigma_1+\frac{m}{2zq}\begin{pmatrix}-im&1\\-1&im\end{pmatrix},
\end{equation}
\begin{align}
    L=&\frac{im_t}{2q^2}\sigma_1-\frac{m\left(u^2-u_x^2\right)}{2zq}\begin{pmatrix}-im&1\\-1&im\end{pmatrix}+\frac{\left(z^2-1\right)u_x}{z^2+1}\sigma_1 \nonumber\\
    &-\frac{2zu}{\left(z^2+1\right)q}\begin{pmatrix}-im&1\\-1&im\end{pmatrix}+\frac{2iz\left(z^2-1\right)}{\left(z^2+1\right)^2}
    \begin{pmatrix}\frac{1}{q}-1&\frac{-im}{q}\\\frac{im}{q}&1-\frac{1}{q}\nonumber\end{pmatrix}.
\end{align}
Moreover, the transformation \eqref{2.5} indicates that
\begin{equation}\nonumber
    \mu_\pm(z)\to I,\quad x\to\pm\infty.
\end{equation}
Then we have the Volterra integral equation about the matrix Jost solution $\mu_\pm(z)$
\begin{equation}\nonumber
    \mu_\pm(z)=I+\int_{\pm\infty}^xe^{-\frac{i}{4}(z-\frac{1}{z})(p(x)-p(y))\hat{\sigma}_3}P(y)\mu_\pm(y)dy.
\end{equation}
Denote
\begin{equation}\nonumber
    \mu_\pm(z)=(\mu_{\pm,1}(z),\mu_{\pm,2}(z)),
\end{equation}
where $\mu_{\pm,1}(z)$ and $\mu_{\pm,2}(z)$ are the first and second columns of $\mu_\pm(z)$, From \cite{boo-36}, we can see that $\mu_{\pm}(z)$ has the following properties:
\begin{enumerate}[($i$)]
    \item $\mu_{-,1}(z)$ and $\mu_{+,2}(z)$ is analytic in $\mathbb{C}_{+}$;

    \item $\mu_{-,2}(z)$ and $\mu_{+,1}(z)$ is analytic in $\mathbb{C}_{-}$;

    \item $\mu_\pm(z)=\sigma_2\overline{\mu_\pm(\bar{z})}\sigma_2=\sigma_1\overline{\mu_\pm(-z^{-1})}\sigma_1$;

    \item $\mu_\pm(z)=F^{-2}\sigma_2\mu_\pm(-z^{-1})\sigma_2.$
        \end{enumerate}

Considering $\Phi_{\pm}(z;x,t)$ are two fundamental matrix solutions of the Lax pair \eqref{Lax}, there exists a linear relation between $\Phi_+(z;x,t)$ and $\Phi_-(z;x,t)$
\begin{equation}\nonumber
    \Phi_-(z;x,t)=\Phi_+(z;x,t)S(z),
\end{equation}
where $S(z)$ is the scattering matrix
\begin{equation}\nonumber
    S(z)=\begin{pmatrix}a(z)&-\overline{b(\bar{z})}\\b(z)&\overline{a(\bar{z})}\end{pmatrix},\quad\det[S(z)]=1,
\end{equation}
and satisfies
\begin{equation}\label{2.16}
S(z)=\overline{S(\bar{z}^{-1})}=\sigma_3S\begin{pmatrix}-z^{-1}\end{pmatrix}\sigma_3.
\end{equation}
We suppose that $a(z)$ has $N_1$ simple
zeros $z_1,\cdots,z_{N_1}$ on $\{z\in\mathbb{C}^+:0<\mathrm{arg}z\leq\pi/2,|z|>1\}$, and $N_2$ simple zeros $w_1,\cdots,w_{N_2}$
on the circle $\{z=e^{i\varphi}:0<\varphi\leq\pi/2\}$. The symmetries \eqref{2.16} imply that
\begin{equation}\nonumber
a(z_n)=0\Leftrightarrow a(-\bar{z}_n)=0\Leftrightarrow a\begin{pmatrix}-z_n^{-1}\end{pmatrix}=0\Leftrightarrow a\begin{pmatrix}\bar{z}_n^{-1}\end{pmatrix}=0,\quad n=1,\cdots,N_1,
\end{equation}
and on the circle
\begin{equation}\nonumber
a(w_m)=0\Leftrightarrow a(-\bar{w}_m)=0,\quad m=1,\cdots,N_2.
\end{equation}
For convenient, we define zeros of $a(z)$ as $\zeta_n=z_n$,
$\zeta_{n+N_1}=-\bar{z}_{n}$, $\zeta_{n+2N_1}=\bar{z}_{n}^{-1}$ and $\zeta_{n+3N_1}=-z_n^{-1}$ for $n=1,\cdots,N_{1}$; $\zeta_{m+4N_1}=w_m$
and $\zeta_{m+4N_1+N_2}=-\bar{w}_m$ for $m=1,\cdots,N_{2}$. Therefore, the discrete spectrum is
\begin{equation}\nonumber
    \mathcal{Z}=\left\{\zeta_n\right\}_{n=1}^{4N_1+2N_2},
\end{equation}
with $\zeta_n\in\mathbb{C}^+$.

\begin{figure}[H]
	\centering
	\begin{tikzpicture}[node distance=2cm]
		\filldraw[pink!40,line width=3] (4,0) rectangle (0,4);
		\filldraw[pink!40,line width=3] (-4,0) rectangle (-0,4);
		\draw[->](-4,0)--(4,0)node[right]{Re$z$};
		\draw[->](0,-4)--(0,4)node[above]{Im$z$};
		\draw[blue] (2,0) arc (0:360:2);

	\coordinate (A) at (2,3);
\coordinate (B) at (2,-3);
\coordinate (C) at (-0.6,0.9);
\coordinate (D) at (-0.6,-0.9);
\coordinate (E) at (0.6,0.9);
\coordinate (F) at (0.6,-0.9);
\coordinate (G) at (-2,2.985);
\coordinate (H) at (-2,-2.985);
		\coordinate (J) at (1.414,1.414);
		\coordinate (K) at (1.414,-1.414);
		\coordinate (L) at (-1.414,1.414);
		\coordinate (M) at (-1.414,-1.414);
		\fill (A) circle (1pt) node[right] {$z_n$};
		\fill (B) circle (1pt) node[right] {$\bar{z}_n$};
		\fill (C) circle (1pt) node[left] {$-\frac{1}{z_n}$};
		\fill (D) circle (1pt) node[left] {$-\frac{1}{\bar{z}_n}$};
		\fill (E) circle (1pt) node[right] {$\frac{1}{\bar{z}_n}$};
		\fill (F) circle (1pt) node[right] {$\frac{1}{z_n}$};
		\fill (G) circle (1pt) node[left] {$-\bar{z}_n$};
		\fill (H) circle (1pt) node[left] {$-z_n$};
		\fill (J) circle (1pt) node[right] {$w_m$};
		\fill (K) circle (1pt) node[right] {$\bar{w}_m$};
		\fill (L) circle (1pt) node[left] {$-\bar{w}_m$};
		\fill (M) circle (1pt) node[left] {$-w_m$};
	\end{tikzpicture}
	\caption{Analytical domains and distribution of the discrete spectrum $\mathcal{Z}$. The blue one is  unit circle.}
	\label{fig2}
\end{figure}

We define the reflection coefficients
\begin{equation}\nonumber
    r(z)=\frac{b(z)}{a(z)},
\end{equation}
which admits symmetry reductions
\begin{equation}\nonumber
r(z)=\overline{r(\bar{z}^{-1})}=r(-z^{-1})=-\overline{r(-\bar{z})}.\end{equation}
Additionally, by using trace formulae we have
\begin{equation}\nonumber
    a(z)=\prod_{j=1}^{4N_1+2N_2}\frac{z-\zeta_j}{z-\bar{\zeta}_j}\exp\left\{-\frac{1}{2\pi i}\int_{\mathbb{R}}\frac{\log(1+|r(s)|^2)}{s-z}ds\right\}.
\end{equation}
$\textbf{Case II: }z=0.$\\
As described in \cite{boo-36}, $r(z)\to0$ as $z\to0.$\\
$\textbf{Case III: }z=\pm i\text{ (corresponding to }\lambda=0).$
Take into account the Jost solutions corresponding to the Lax pair \eqref{Lax}, which are restricted by the boundary conditions
\begin{equation}\nonumber
\begin{array}{cc}\Phi_\pm(z)\sim e^{(-\frac{k}{2}x+\frac{k}{\lambda^2}t)\sigma_3},&x\to\pm\infty.\end{array}
\end{equation}
Define a new transformation
\begin{equation}\label{2.23}
\mu_\pm^0(z)=\Phi_\pm(z)e^{(\frac{k}{2}x-\frac{k}{\lambda^2}t)\sigma_3},
\end{equation}
then
\begin{equation}\nonumber
\mu_\pm^0(z)\to I,\quad x\to\pm\infty.\end{equation}
We obtain equivalent of Lax pairs
\begin{equation}\nonumber
(\mu_\pm^0)_x=-\frac{k}{2}[\sigma_3,\mu_\pm^0]+L_0\mu_\pm^0,
\end{equation}
\begin{equation}\nonumber
(\mu_\pm^0)_t=\frac{k}{\lambda^2}[\sigma_3,\mu_\pm^0]+M_0\mu_\pm^0,
\end{equation}
with
\begin{equation}\nonumber
\begin{aligned}&L_{0}=\frac{\lambda mi}{2}\sigma_{2},\\&M_{0}=\frac{\left(u^{2}-u_{x}^{2}\right)}{2}\begin{pmatrix}{k}&{-\lambda m}\\{\lambda m}&{-k}\end{pmatrix}+\frac{u}{\lambda}\begin{pmatrix}{0}&{-1}\\{1}&{0}\end{pmatrix}+\frac{k}{\lambda}u_{x}\sigma_{1}.\end{aligned}
\end{equation}
In order to recover the potential $u(x,t)$, we analyze the asymptotic expansion when $z\to i$.
\begin{equation}\nonumber
\mu^0=I+\mu_1^0(z-i)+\mathcal{O}\left((z-i)^2\right),
\end{equation}
where
\begin{equation}\nonumber
\mu_1^0=\begin{pmatrix}0&-\frac{1}{2}(u+u_x)\\-\frac{1}{2}(u-u_x)&0\end{pmatrix}.
\end{equation}
The relations \eqref{2.5} and \eqref{2.23} lead to
\begin{equation}\label{2.28}
\mu_{\pm}(z)=F^{-1}(x,t)\mu_{\pm}^{0}(z)e^{\frac{i}{4}(z-\frac{1}{z})c_{\pm}(x,t)\sigma_{3}},
\end{equation}
where
\begin{equation}\nonumber
c_\pm(x,t)=\int_{\pm\infty}^x{(q-1)dy}.\end{equation}
Further, taking $z\to i$ in \eqref{2.28}, we get the asymptotic of $a(z)$
\begin{equation}\nonumber
    a(z)=e^{\frac{1}{2}\int_{\mathbb{R}}(q-1)dx}\left(1+\mathcal{O}\left((z-i)^2\right)\right),\quad\mathrm{as}~~z\to i.
\end{equation}
Define a sectionally meromorphic matrix
\begin{equation}\nonumber
N(z)_1\triangleq N_1(z;x,t)=\left\{\begin{array}{ll}\left(\frac{\mu_{-,1}(z)}{a(z)},\mu_{+,2}(z)\right),&\mathrm{as}~~z\in\mathbb{C}^+,\\\\
\left(\mu_{+,1}(z),\frac{\mu_{-,2}(z)}{\overline{a(\bar{z})}}\right),&\mathrm{as}~~z\in\mathbb{C}^-.\end{array}\right.
\end{equation}
 The RH problem related to the Cauchy problem of the mCH equation reads as follows.
\begin{RHP}\label{RHP1}
    		Find a matrix valued function $N_1(z)$ admits:
    		\begin{enumerate}[($i$)]
    			\item Analyticity:~$N_1(z;x,t)$ is analytic in $k\in\mathbb{C}\setminus(\mathbb{R}\cup\mathcal{Z}\cup\bar{\mathcal{Z}})$;
                \item Symmetry:$N_1(z)=\sigma_{3}\overline{N_1(-\bar{z})}\sigma_{3}=\sigma_{2}\overline{N_1(\bar{z})}\sigma_{2}=F^{-2}\overline{N_1(-\bar{z}^{-1})};$
    			\item Jump condition:
    			$N_{1+}(z)=N_{1-}(z)\tilde{V}(z),\quad z\in\mathbb{R},$
    where
    		\begin{equation}\nonumber
    \tilde{V}(z)=\begin{pmatrix}1+|r(z)|^2&\overline{r(z)}e^{-kp}\\r(z)e^{kp}&1\end{pmatrix};
\end{equation}

                \item Residue conditions: $N_1(z)$ has simple poles at each point in $\mathcal{Z}\cup\bar{\mathcal{Z}}$ with:
                 \begin{equation}\nonumber
    \operatorname*{\mathrm{Res}}_{z=\zeta_n}N_1(z)=\lim_{z\to\zeta_n}N(z)\begin{pmatrix}0&0\\c_ne^{-k(\zeta_n)p(\zeta_n)}&0\end{pmatrix},
\end{equation}
                \begin{equation}\nonumber
    \operatorname*{\mathrm{Res}}_{z=\bar{\zeta}_n}N_1(z)=\lim_{z\to\bar{\zeta}_n}N(z)\begin{pmatrix}0&-\bar{c}_ne^{k(\bar{\zeta}_n)p(\bar{\zeta}_n)}\\0&0\end{pmatrix}.
\end{equation}

    			\item Asymptotic behavior:
                 \begin{equation}\nonumber
    N_1(z)=I+\mathcal{O}(z^{-1}),\quad z\to\infty,
\end{equation}
    			\begin{equation}\nonumber
    N_1(z)=F^{-1}\left[I+\mu_1^0(z-i)\right]e^{\frac{1}{2}c_+\sigma_3}+\mathcal{O}\left((z-i)^2\right).
\end{equation}
    		\end{enumerate}
    	\end{RHP}

The reconstruction of the solution to the mCH equation \eqref{mch} is impeded by the lack of knowledge about $p(x,t,z)$. To overcome this obstacle, Boutet de Monvel and Shepelsky proposed a concept of changing the spatial variable \cite{boo-8,boo-9}. Following this idea, we introduce a new scale
    \begin{equation}\nonumber
    y(x,t)=x-\int_{x}^{+\infty}(q(s)-1)ds.
\end{equation}
By the definition of the new scale $y(x,t)$, we define
\begin{equation}\nonumber
    M(z)=M(z;y,t)\triangleq N_1(z;x(y,t),t).
\end{equation}
Denote the phase function
\begin{equation}\nonumber
    \theta(z)=\frac{i}{2}k(z)\begin{bmatrix}\frac{y}{t}-2\lambda^{-2}(z)\end{bmatrix},
\end{equation}
then we can get the RH problem for the new variable $(y,t)$.

\begin{RHP}\label{RHP2}
    		Find a matrix valued function $M(z)$ admits:
    		\begin{enumerate}[($i$)]
    			\item Analyticity:~$M(z)$ is meromorphic in $\mathbb{C}\setminus\mathbb{R}$ and has single poles;
                \item Symmetry:$M(z)=\sigma_{3}\overline{M(-\bar{z})}\sigma_{3}=\sigma_{2}\overline{M(\bar{z})}\sigma_{2}=F^{-2}\overline{M(-\bar{z}^{-1})};$
    			\item Jump condition:
    			$M_+(z)=M_-(z)V(z),\quad z\in\mathbb{R},$
    where
    		\begin{equation}\nonumber
    V(z)=\begin{pmatrix}1+|r(z)|^2&e^{2it\theta(z)}\overline{r(z)}\\e^{-2it\theta(z)}r(z)&1\end{pmatrix};
\end{equation}

                \item Residue conditions: $M(z)$ has simple poles at each point in $\mathcal{Z}\cup\bar{\mathcal{Z}}$ with:
                 \begin{equation}\nonumber
    \operatorname*{\mathrm{Res}}_{z=\zeta_n}M(z)=\lim_{z\to\zeta_n}M(z)\begin{pmatrix}0&0\\c_ne^{-2it\theta_n}&0\end{pmatrix},
\end{equation}
                \begin{equation}\nonumber
    \operatorname*{\mathrm{Res}}_{z=\bar{\zeta}_n}M(z)=\lim_{z\to\bar{\zeta}_n}M(z)\begin{pmatrix}0&-\bar{c}_ne^{2it\theta_n}\\0&0\end{pmatrix}.
\end{equation}

    			\item Asymptotic behavior:
                 \begin{equation}\nonumber
    M(z)=I+\mathcal{O}(z^{-1}),\quad z\to\infty,
\end{equation}
    			\begin{equation}\label{2.29}
    M(z)=F^{-1}\left[I+\mu_1^0(z-i)\right]e^{\frac{1}{2}c_+\sigma_3}+\mathcal{O}\left((z-i)^2\right).
\end{equation}
    		\end{enumerate}
    	\end{RHP}
    Considering the asymptotic trends of the functions $\mu_\pm(z)$ along with equation \eqref{2.29}, we obtain the reconstruction formula for $u(x,t)=u(y(x,t),t)$:

    \begin{equation}\label{2.35}
    u(y,t)=\lim_{z\to i}\frac{1}{z-i}\left(1-\frac{(M_{11}(z)+M_{21}(z))(M_{12}(z)+M_{22}(z))}{(M_{11}(i)+M_{21}(i))(M_{12}(i)+M_{22}(i))}\right),
\end{equation}
\begin{equation}\nonumber
    x(y,t)=y+c_+(x,t)=y-\ln\left(\frac{M_{12}(i)+M_{22}(i)}{M_{11}(i)+M_{21}(i)}\right).
\end{equation}

\subsection{Classification of asymptotic regions}\label{s:2.2}
The jump matrix $V(z)$ admits the following two factorizations
\begin{align}\nonumber
    \begin{aligned}V(z)&=\begin{pmatrix}1&\bar{r}(z)e^{2it\theta}\\0&1\end{pmatrix}\begin{pmatrix}1&0\\r(z)e^{-2it\theta}&1\end{pmatrix}\\&=\begin{pmatrix}1&\\
    \frac{r(z)e^{-2it\theta}}{1+|r(z)|^2}&1\end{pmatrix}(1+|r(z)|^2)^{\sigma_3}\begin{pmatrix}1&\frac{\bar{r}(z)e^{2it\theta}}{1+|r(z)|^2}\\0&1\end{pmatrix}.\end{aligned}
\end{align}
The long-time asymptotic of RH problem \ref{RHP2} is influenced by the growth and decay of the oscillatory terms $e^{\pm2it\theta(z)}$ in the jump matrix $V(z)$. Direct calculations show that
\begin{align}\label{2.38}
    \theta(z)=-\frac{1}{4}(z-z^{-1})\left[\frac{y}{t}-\frac{8}{(z+z^{-1})^2}\right],
\end{align}
\begin{align}\nonumber
    \begin{aligned}\mathrm{Re}(2it\theta)&=-2t\mathrm{Im}\theta\\
    &=-2t\mathrm{Im}z\left[-\frac{\xi}{4}\left(1+|z|^{-2}\right)+2\frac{-|z|^6+2|z|^4
    +(3\mathrm{Re}^2z-\mathrm{Im}^2z)(1+|z|^2)+2|z|^2-1}{\left((\mathrm{Re}^2z-\mathrm{Im}^2z+1)^2+4\mathrm{Re}^2z\mathrm{Im}^2z\right)^2}\right],\end{aligned}
\end{align}
where $\xi\triangleq\frac{y}{t}$. The signature of $\mathrm{Im}\theta$ is shown in Fig.\ref{fig3}. The count of phase points located on the jump contour $\mathbb{R}$ gives us the means to split the $(y,t)$ half-plane into six separate kinds of asymptotic regions.
\begin{enumerate}[($i$)]
    			\item $\xi>2$ and $\xi<-1/4$ in  Fig.\ref{fig3}(a),(f);
                \item $-1/4<\xi<0$ and $0\leq\xi<2$ in Fig.\ref{fig3}(c),(d);
                 \item $\xi=2$ in Fig.\ref{fig3}(b);
                \item $\xi=-1/4$ in Fig.\ref{fig3}(e).

    		\end{enumerate}
    The regions (i)-(ii) have studied by \cite{boo-36}. For the region (iii), there are two phase points on $\mathbb{R}$, and for the region (iv), there are four phase points on $\mathbb{R}$. The transition regions $\mathcal{P}_{I}$ and $\mathcal{P}_{II}$, as defined in equation \eqref{P}, are distinguished by focusing on the areas in the vicinity of the lines $\xi~~or~~\tau=2$ and $\xi~~or~~\tau=-1/4$ ($\xi=\tau+\mathcal{O}(t^{-1})$, see \eqref{5.3}, similar to \cite{boo-35}). These lines act as boundaries separating adjacent bulk asymptotic regions, and the characterization is achieved through the application of a particular double scaling limit. How to deal with the long-time asymptotics in transition regions is an remaining problem proposed by \cite{boo-36}, and we will solve it in our present paper.
    \begin{figure}[H]
	\centering
	\subfigure[$ \xi>2$]{\includegraphics[width=0.3\linewidth]{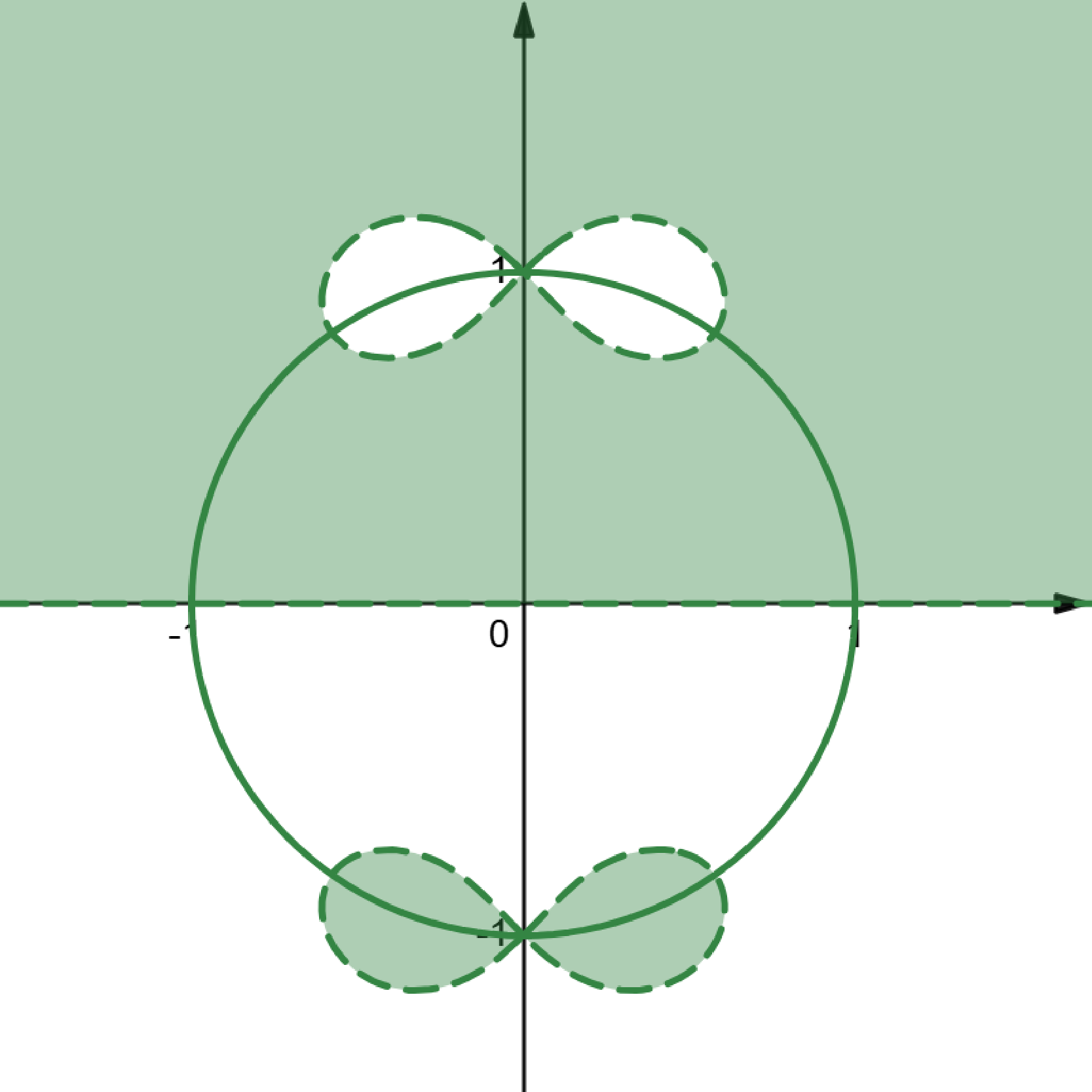}
	\label{fig:desmos-graph}}
	\subfigure[$\xi=2$]{\includegraphics[width=0.3\linewidth]{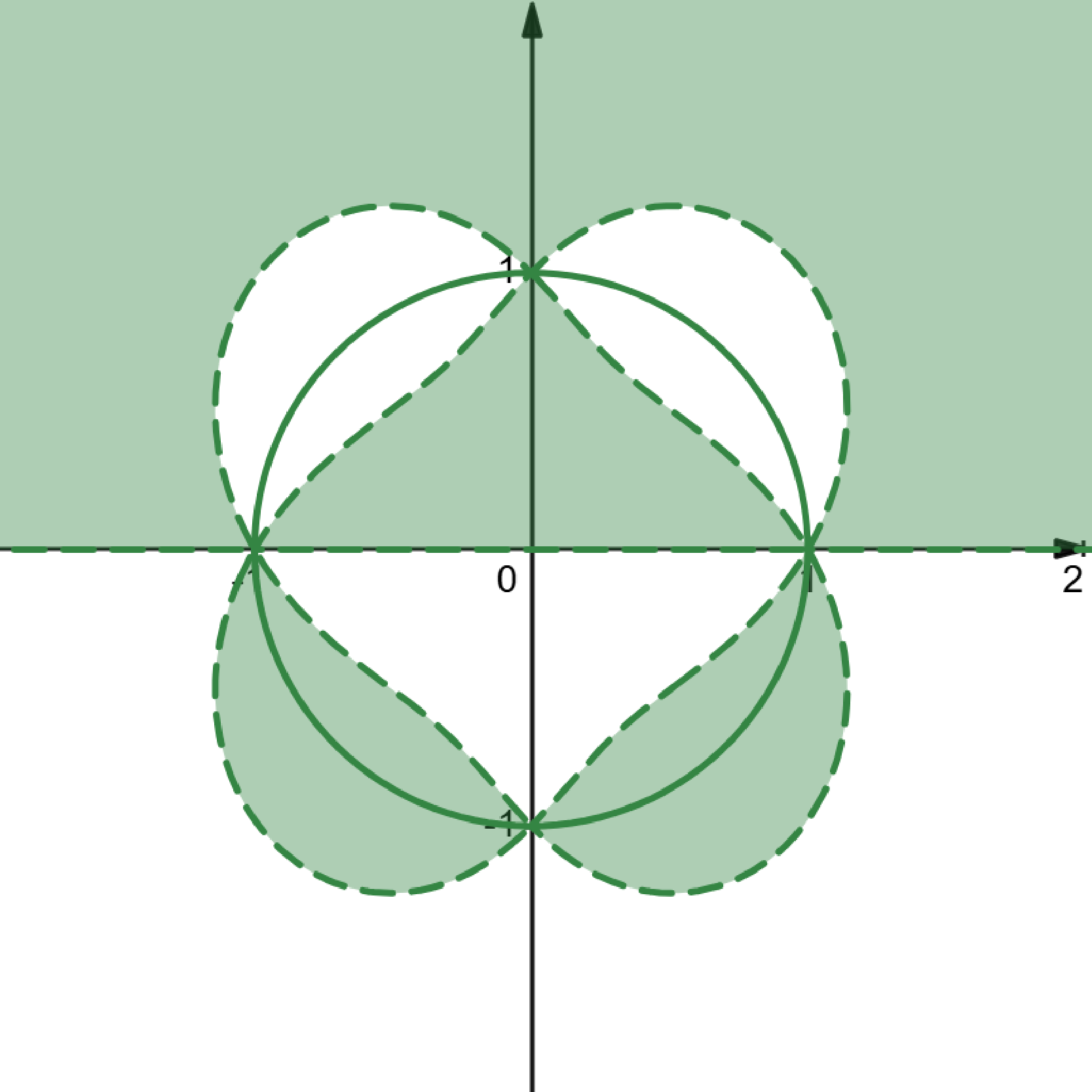}
	\label{fig:desmos-graph-1}}
	\subfigure[$0\leq\xi<2$]{\includegraphics[width=0.3\linewidth]{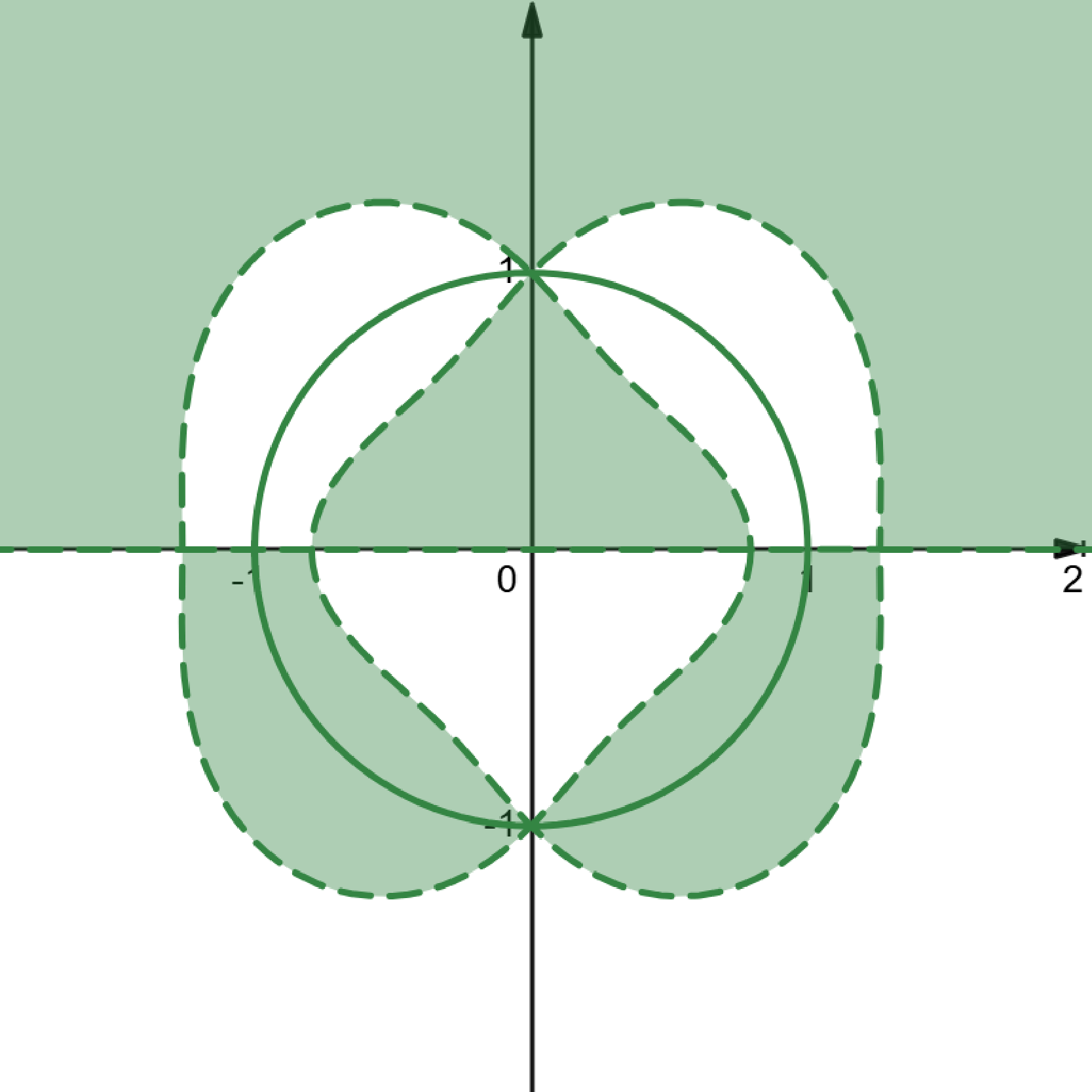}
	\label{fig:desmos-graph-5}}	
	\subfigure[$-\frac{1}{4}<\xi<0$]{\includegraphics[width=0.3\linewidth]{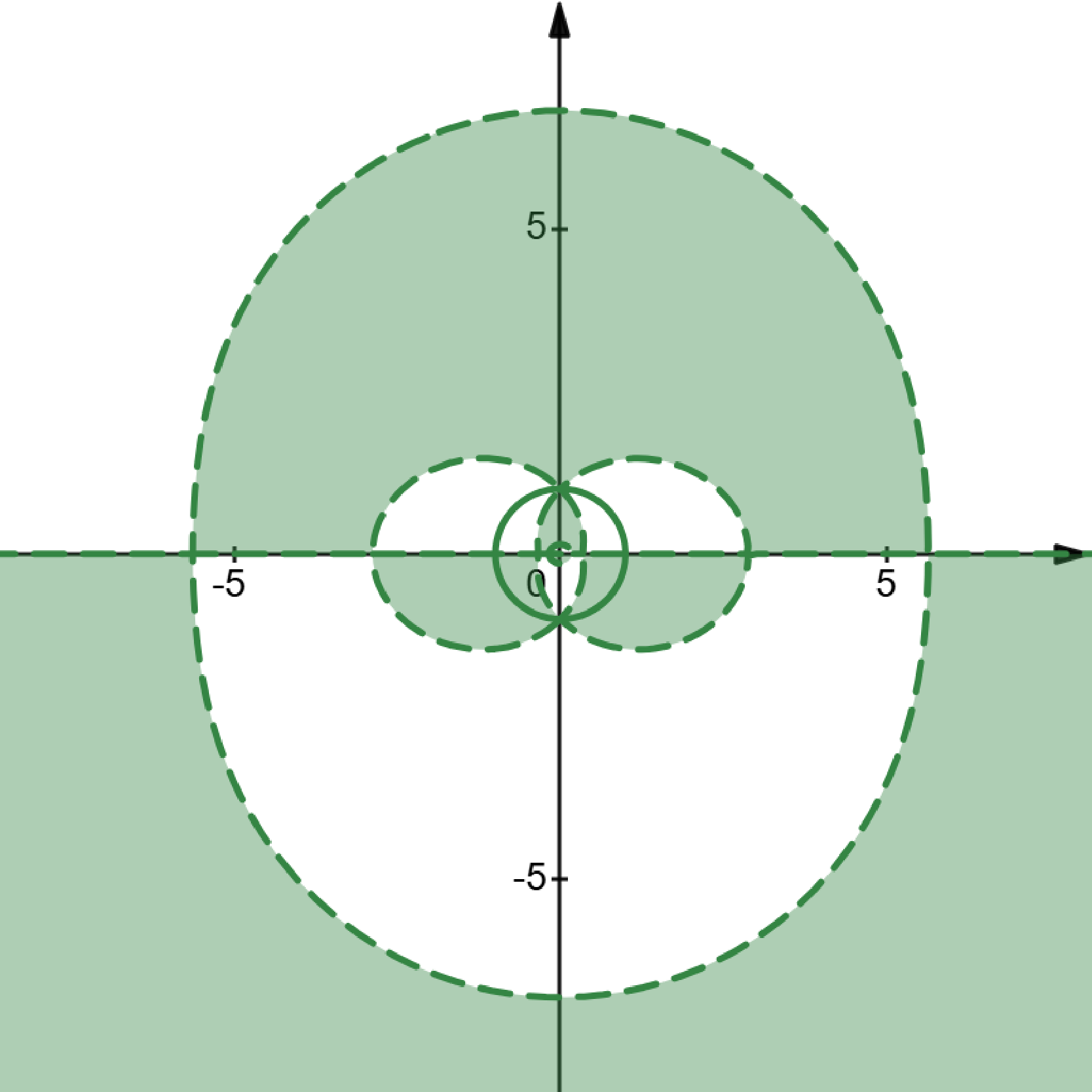}
	\label{fig:desmos-graph-3}}
	\subfigure[$\xi=-\frac{1}{4}$]{\includegraphics[width=0.3\linewidth]{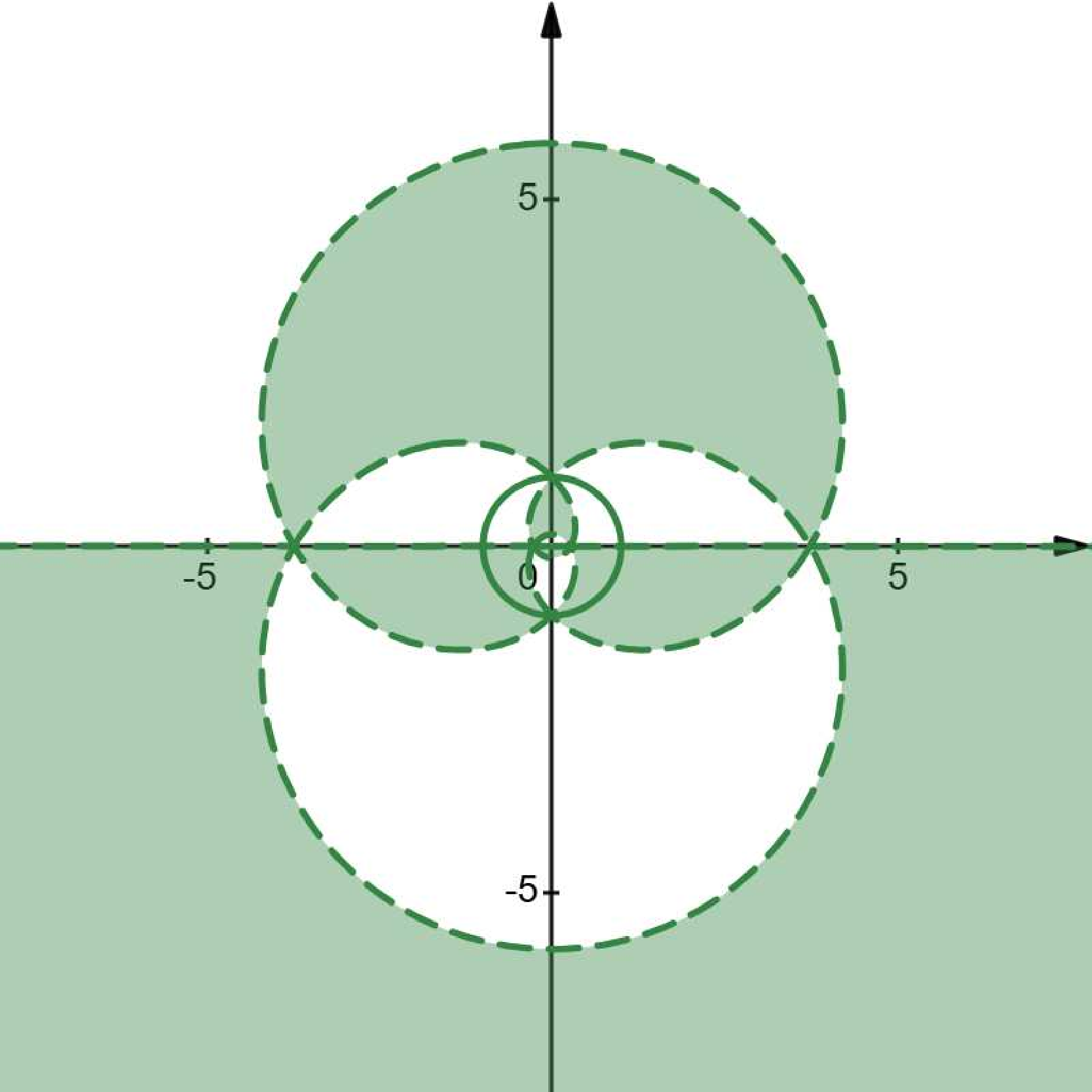}
	\label{fig:desmos-graph-4}}
	\subfigure[$\xi<-\frac{1}{4}$]{\includegraphics[width=0.3\linewidth]{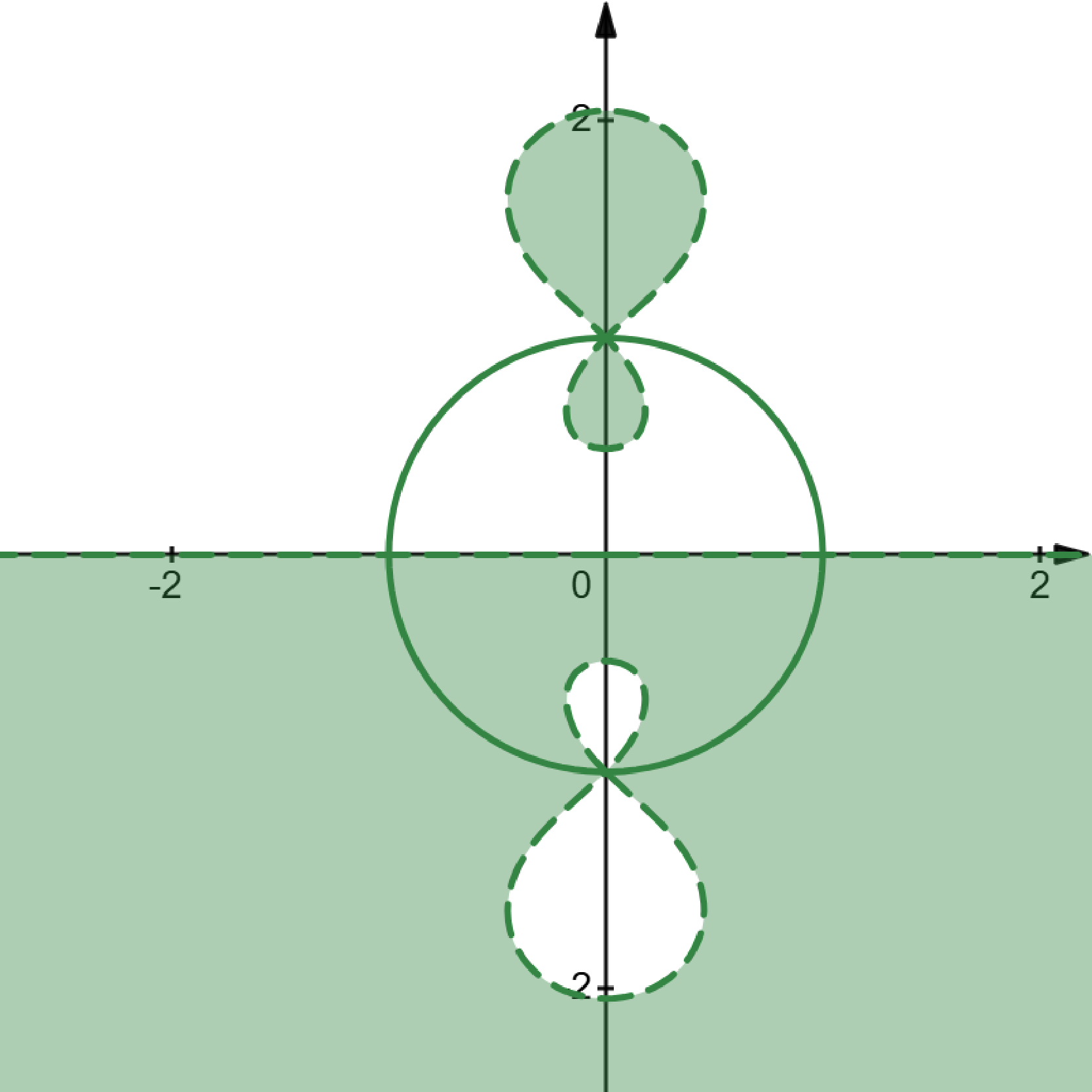}
	\label{fig:desmos-graph-6}}
	\caption{(a)-(f) describe the signature table of $\mathrm{Re}(2it\theta)$ and distribution of phase points. The white region,
$\mathrm{Re}(2it\theta)<0$ which implies that $e^{2it\theta(z)}\to0$ as $t\to+\infty$; The green region, $\mathrm{Re}(2it\theta)>0$ which
implies that $e^{-2it\theta(z)}\to0$ as $t\to+\infty$; The boundary between grey region and white region is critical line,
$\mathrm{Re}(2i\theta(z))=0$ with $|e^{-2it\theta(z)}|=1$.  }
	\label{fig3}
\end{figure}

\subsection{Dealing with the poles}\label{s:2.3}

In order to perform the long-time analysis via the $\bar{\partial}$-steepest descent method, we need
to reform RH problem by the following two essential operations:
\begin{enumerate}[($i$)]
  \item Decompose the jump matrix $V(z)$ into appropriate upper/lower triangular factorizations so that the oscillating factor $e^{\pm2it\theta(z)}$ are decaying in the corresponding sectors respectively;
  \item Interpolate poles far away from the critical line by swapping them into jumps on small closed loops around each pole \cite{boo-57}.
\end{enumerate}

Different from the case of \cite{boo-35}, the poles $\zeta_j$ are randomly distributed on the complex plane. We cannot simply remove the contribution of poles and have to deal with the poles near the critical line.

We denote the subscript set of all poles $\mathcal{N}\triangleq\{1,...,4N_1+2N_2\}$, and we need to define some notations
\begin{equation}\nonumber
\nabla=\left\{n\in\mathcal{N}:\mathrm{Im}\theta_n<0\right\},\Delta=\left\{n\in\mathcal{N}:\mathrm{Im}\theta_n>0\right\},
\Lambda=\left\{n\in\mathcal{N}:\left|\mathrm{Im}\theta_n\right|<\frac{\varrho}{2}\right\},\end{equation}
where
\begin{equation}\nonumber
\varrho=\frac{1}{2}\min\left\{\min_{j\in\mathcal{N}}\left\{|\mathrm{Im}\zeta_j|\right\},\min_{j\in\mathcal{N}\setminus\Lambda,
\mathrm{Im}\theta(z)=0}|\zeta_j-z|,\min_{j\in\mathcal{N}}|\zeta_j-i|,\min_{j,k\in\mathcal{N}}|\zeta_j-\zeta_k|\right\}.\end{equation}

We introduce the function
\begin{equation}\label{2.40}
    T(z,\xi):=\prod_{j\in\Delta}\left(\frac{z-\zeta_n}{\bar{\zeta}_n^{-1}z-1}\right)\exp\left\{-\frac{1}{2\pi i}\int_{I(\xi)}\frac{\log\left(1+|r(s)|^2\right)}{s-z}ds\right\},
\end{equation}
where
\begin{equation}\nonumber
    I(\xi)=\begin{cases}\emptyset, &|\xi-2|t^{2/3}\leqslant C,\\
    \mathbb{R}, &|\xi+\frac{1}{4}|t^{2/3}\leqslant C.
    \end{cases}\end{equation}

\begin{remark}
 In order to forge a link between the RH problem pertaining to $M(z)$ and the model problem presented in Appendix \ref{appendix}, we execute a series of transformations. During this process, by means of the second matrix decomposition of the jump matrix $V(z)$, we maintain the jumps along the contours $(z_2,z_1)$ and $(z_4,z_3)$ without splitting or "opening" them.
 \end{remark}

Then the small disks $\mathbb{D}_n:=\mathbb{D}(\zeta_n,\varrho)$ are pairwise disjoint, also disjoint with critical lines and the contours. Moreover, $i\notin\mathbb{D}_{n}$.
We construct the interpolation function
\begin{equation}\label{3.8}
G(z)=\begin{cases}
\begin{pmatrix}1&0\\-C_n(z-\zeta_n)^{-1}e^{-2it\theta_n}&1\end{pmatrix},&\mathrm{as~}z\in\mathbb{D}_n,n\in\nabla\setminus\Lambda;\\
\begin{pmatrix}1&-C_n^{-1}(z-\zeta_n)e^{2it\theta_n}\\0&1\end{pmatrix},&\mathrm{as~}z\in\mathbb{D}_n,n\in\Delta\setminus\Lambda;\\
\begin{pmatrix}1&\bar{C_n}(z-\bar{\zeta_n})^{-1}e^{2it\theta_n}\\0&1\end{pmatrix},&\mathrm{as~}z\in\mathbb{D}_n,n\in\nabla\setminus\Lambda;\\
\begin{pmatrix}1&0\\ \bar{C_n}^{-1}(z-\bar{\zeta_n})^{-1}e^{-2it\theta_n}&1\end{pmatrix},&\mathrm{as~}z\in\mathbb{D}_n,n\in\Delta\setminus\Lambda;\\
I &elsewhere\end{cases}.
\end{equation}
where
\begin{equation}\nonumber
    \mathbb{D}_n:=\{z:|z-\zeta_n|\leqslant\varrho\},\quad n=1,\cdots,4N_1+2N_2.
\end{equation}
We now introduce the following interpolation transformation,
\begin{equation}\label{3.7}
    M^{(1)}(z)=M^{(1)}(z;y,t):=M(z)G(z)T^{\sigma_3}(z).
\end{equation}
It is readily seen that $M^{(1)}(z)$ satisfies the following RH problem.
\begin{RHP}\label{RHP4}
    		Find a matrix valued function $M^{(1)}(z)\triangleq M^{(1)}(z;y,t)$ admits:
    		\begin{enumerate}[($i$)]
    			\item Analyticity:~$M^{(1)}(z)$ is holomorphic for $z\in\mathbb{C}\setminus\Sigma^{(1)}$, where
    \begin{equation}\nonumber
    \Sigma^{(1)}:=\mathbb{R}\cup\left[\bigcup_{n{\in}\mathcal{N}\backslash\Lambda}(\mathbb{D}_n\cup\mathbb{D}_n^*)\right];
\end{equation}
                \item Symmetry:$M^{(1)}(z)=\sigma_{3}\overline{M^{(1)}(-\bar{z})}\sigma_{3}
                    =\sigma_{2}\overline{M^{(1)}(\bar{z})}\sigma_{2}=F^{-2}\overline{M^{(1)}(-\bar{z}^{-1})}$;
    			\item Jump condition:
    			$M_+^{(1)}(z)=M_-^{(1)}(z)V^{(1)}(z), z\in\Sigma^{(1)},$
    where
    		\begin{equation}V^{(1)}(z)=\begin{cases}
    \begin{pmatrix}1&e^{2it\theta}\bar{r}(z)T^{-2}(z)\\0&1\end{pmatrix}\begin{pmatrix}1&0\\e^{-2it\theta}r(z)T^2(z)&1\end{pmatrix},&z\in\mathbb{R}\setminus I(\xi);\\
    \begin{pmatrix}1&0\\ \frac{e^{-2it\theta(z)}r(z)T_-^2(z)}{1+|r(z)|^2}&1\end{pmatrix}\begin{pmatrix}1&\frac{e^{2i\theta(z)}\bar{r}(z)T_+^{-2}(z)}{1+|r(z)|^2}\\0&1\end{pmatrix},&z\in I(\xi);\\
    \begin{pmatrix}1&0\\-C_n(z-\zeta_n)^{-1}T^2(z)e^{-2it\theta_n}&1\end{pmatrix},&z\in\partial\mathbb{D}_n,n\in\nabla\setminus\Lambda;\\
    \begin{pmatrix}1&-C_n^{-1}(z-\zeta_n)T^{-2}(z)e^{2it\theta_n}\\0&1\end{pmatrix},&z\in\partial\mathbb{D}_n,n\in\Delta\setminus\Lambda;\\
    \begin{pmatrix}1&\bar{C}_n(z-\bar{\zeta}_n)^{-1}T^{-2}(z)e^{2it\bar{\theta}_n}\\0&1\end{pmatrix},&z\in\partial\overline{\mathbb{D}}_n,n\in\nabla\setminus\Lambda;\\
    \begin{pmatrix}1&0\\\bar{C}_n^{-1}(z-\bar{\zeta}_n)e^{-2it\bar{\theta}_n}T^2(z)&1\end{pmatrix},&z\in\partial\overline{\mathbb{D}}_n,n\in\Delta\setminus\Lambda;\end{cases}\end{equation}
    			\item Asymptotic behavior:
                 \begin{equation}\nonumber
    M^{(1)}(z;y,t)=I+\mathcal{O}(z^{-1}),\quad z\to\infty,
\end{equation}
    			\begin{equation}\nonumber
    M^{(1)}(z;y,t)=F^{-1}\left[I+(z-i)\right]e^{\frac{1}{2}c_+\sigma_3}T(i)^{\sigma_3}\left(I-I_0\sigma_3(z-i)\mu_0^{(1)}\right)+\mathcal{O}\left((z-i)^2\right);
\end{equation}
\item Residue conditions: If $j\notin\Lambda$ for $j=1,2,\ldots,N,$, then $M^{(1)}$ is analytic in the region $\mathbb{C}\setminus\Sigma^{(1)}$. If there exist $\zeta_n$ and $\bar{\zeta}_n$ for $n\in\Lambda$, then $M^{(1)}$ admits the residue conditions:
    \begin{equation}\nonumber
    \operatorname*{\mathrm{Res}}_{z=\zeta_n}M^{(1)}(z)=\lim_{z\to\zeta_n}M^{(1)}(z)\begin{pmatrix}0&0\\C_ne^{-2it\theta_n}T^2(\zeta_n)&0\end{pmatrix},\end{equation}
    \begin{equation}\nonumber
    \operatorname*{\mathrm{Res}}_{z=\bar{\zeta}_n}M^{(1)}(z)=\lim_{z\to\bar{\zeta}_n}M^{(1)}(z)
    \begin{pmatrix}0&-\bar{C}_nT^{-2}(\bar{\zeta}_n)e^{2it\bar{\theta}_n}\\0&0\end{pmatrix}.\end{equation}

    		\end{enumerate}
    	\end{RHP}

    It is worth noting that, the off-diagonal elements in jump matrices on the circles $\partial\mathbb{D}_n$ and $\partial\overline{\mathbb{D}}_n$ are decaying exponentially to zero as $t\to\infty$ which can be shown in the following Lemma.
    \begin{lemma}\label{lem3.3}
    As $t\to\infty$, the jump matrix $V^{(1)}(z)$ satisfies
    \begin{equation}\left|\left|V^{(1)}(z)-\mathbb{I}\right|\right|_{L^\infty(\partial\mathbb{D}_n\cup\partial\overline{\mathbb{D}}_n)}\lesssim e^{-ct}.\end{equation}
\end{lemma}
\begin{proof}
For the above estimation , we mainly give the proof of the case $\partial\mathbb{D}_n$ for $n\in\nabla\setminus\Lambda$, the other cases can be proved similarly. We have
\begin{equation}
\left|\left|V^{(1)}(z)-\mathbb{I}\right|\right|_{L^\infty(\partial\mathbb{D}_n\cup\partial\overline{\mathbb{D}}_n)}=|-C_n(z-\zeta_n)^{-1}T^2(z)e^{-2it\theta_n}|.
\end{equation}
Then, based on the properties of $T(z)$, we know that $T^2(z)$ is bounded. Finally, considering $\theta_n$ is fix value and $z-\zeta_{n}=\varrho$, the remained part is a direct result.
\end{proof}
   On the basis of Lemma \ref{lem3.3}, the RH problem \ref{RHP4} is asymptotically equivalent to the following RH
problem with respect to $M^{(2)}(z;y,t)$.
    \begin{RHP}\label{RHP5}
    		Find a matrix valued function $M^{(2)}(z)\triangleq M^{(2)}(z;y,t)$ admits:
    		\begin{enumerate}[($i$)]
    			\item Analyticity:~$M^{(2)}(z)$ is meromorphic for $z\in\mathbb{C}\setminus\mathbb{R}$;

                \item Symmetry:$M^{(2)}(z)=\sigma_{3}\overline{M^{(2)}(-\bar{z})}\sigma_{3}=\sigma_{2}\overline{M^{(2)}(\bar{z})}\sigma_{2}=F^{-2}\overline{M^{(2)}(-\bar{z}^{-1})}$;
    			\item Jump condition:
    			$M_+^{(2)}(z)=M_-^{(2)}(z)V^{(2)}(z), z\in\mathbb{R},$
    where
    \begin{equation}\label{3.13}
    V^{(2)}(z)=\begin{cases}
    \begin{pmatrix}1&e^{2it\theta}\bar{r}(z)T^{-2}(z)\\0&1\end{pmatrix}\begin{pmatrix}1&0\\e^{-2it\theta}r(z)T^2(z)&1\end{pmatrix},&z\in\mathbb{R}\setminus I(\xi);\\
    \begin{pmatrix}1&0\\ \frac{e^{-2it\theta(z)}r(z)T_-^2(z)}{1+|r(z)|^2}&1\end{pmatrix}\begin{pmatrix}1&\frac{e^{2i\theta(z)}\bar{r}(z)T_+^{-2}(z)}{1+|r(z)|^2}\\0&1\end{pmatrix},&z\in I(\xi);
    \end{cases}\end{equation}

    			\item Asymptotic behavior:
                 \begin{equation}\nonumber
    M^{(2)}(z;y,t)=I+\mathcal{O}(z^{-1}),\quad z\to\infty,
\end{equation}
    			\begin{equation}\nonumber
    M^{(2)}(z;y,t)=F^{-1}\left[I+(z-i)\right]e^{\frac{1}{2}c_+\sigma_3}T(i)^{\sigma_3}\left(I-I_0\sigma_3(z-i)\mu_0^{(1)}\right)+\mathcal{O}\left((z-i)^2\right);
\end{equation}
\item Residue conditions: The residue condition is the same with RH problem \ref{RHP4} by replacing $M^{(1)}(z;y,t)$ with $M^{(2)}(z;y,t)$.
    		\end{enumerate}
    	\end{RHP}

Moreover, according to Lemma \ref{lem3.3}, the solution of RH problem \ref{RHP4} can be approximated by the solution of RH problem \ref{RHP5}
\begin{equation}\nonumber
    M^{(1)}(z)=M^{(2)}(z)\left(I+\mathcal{O}\left(e^{-ct}\right)\right),
\end{equation}
where $c>0$ is a constant. Next, we will perform the asymptotic analysis in different
transition zones based on RH problem \ref{RHP5}.

\section{Asymptotic Analysis in the Transition Region $\mathcal{P}_{I}$}\label{s:3}

In this section, we consider the asymptotics in the region $\mathcal{P}_{I}$ given by
\begin{equation}\nonumber
    \mathcal{P}_{I}:=\left\{(x,t)\in\mathbb{R}\times\mathbb{R}^+:0\leqslant\left|\frac{x}{t}-2\right|t^{2/3}\leqslant C\right\},
\end{equation}
where $C>0$ is a constant. At the end, we will show that $\tau=x/t$ is close to $\xi=y/t$ for large positive $t$. In fact, according to \eqref{5.3}, we can obtain $\xi=\tau+\mathcal{O}(t^{-1})$. We only provide a detailed analysis in $-C\leqslant(\xi-2)t^{2/3}\leqslant 0$ in this section, as the discussion for the other half zone is similar.

In the transition region $-C\leqslant(\xi-2)t^{2/3}\leqslant 0$, there are four phase points where the phase points $z_{1}$ and $z_{2}$ approach $-1$ and $z_{3}$ and $z_{4}$ approach $1$ at least as fast as $t^{-1/3}$ as $t\to+\infty$.

Recall that the expression of $\theta(z)$
\begin{align}\nonumber
    \theta(z)=-\frac{1}{4}(z-z^{-1})\left[\frac{y}{t}-\frac{8}{(z+z^{-1})^2}\right].
\end{align}
According to the obtained results in \cite{boo-35}, the four stationary points are expressed as
\begin{equation}\label{3.17}
    z_1=2\sqrt{s_+}+\sqrt{4s_++1},\quad z_2=-2\sqrt{s_+}+\sqrt{4s_++1},
\end{equation}
\begin{equation}\label{3.18}
    z_3=2\sqrt{s_+}-\sqrt{4s_++1},\quad z_4=-2\sqrt{s_+}-\sqrt{4s_++1},
\end{equation}
where
\begin{equation}\label{3.19}
    s_+:=\frac{1}{4\xi}\left(-\xi-1+\sqrt{1+4\xi}\right).
\end{equation}

\subsection{Transformation to a hybrid $\bar{\partial}$-RH problem}\label{s:3.1}

Through simple and direct calculations, it can be shown that
\begin{equation}\nonumber
    z_1=1/z_2=-1/z_3=-z_4,
\end{equation}
and as $\xi\to2^-$,
\begin{equation}\label{3.21}
    z_{1,2}\to1,\quad z_{3,4}\to-1.
\end{equation}
In the region $\mathcal{P}_{I}$ we have $I(\xi)=\emptyset$, thus
    \begin{equation}\label{3.5}
    T(z,\xi):=\prod_{j\in\Delta}\left(\frac{z-\zeta_n}{\bar{\zeta}_n^{-1}z-1}\right).
\end{equation}
By \eqref{3.13}, we can represent the jump matrix for $V^{(2)}(z)$ as
\begin{equation}\nonumber
    \begin{aligned}V^{(2)}(z)&=\begin{pmatrix}1&e^{2it\theta(z)}d(z)\\0&1\end{pmatrix}
    \begin{pmatrix}1&0\\e^{-2it\theta(z)}\bar{d}(z)&1\end{pmatrix},&z\in\mathbb{R},\end{aligned}
\end{equation}
where
\begin{equation}\label{3.22}
    d(z):=d\left(z;\xi\right)=\bar{r}(z)T^{-2}(z)\overset{(2.11)}{\operatorname*{=}}\bar{r}(z)\prod_{j\in\Delta}\left(\frac{z-\zeta_n}{\bar{\zeta}_n^{-1}z-1}\right)^{-2}.
\end{equation}
Select a sufficiently small angle $\varphi_0$ satisfying $0<\varphi_0<\pi/4$. Ensure that all the sectors formed by this angle are situated entirely within their corresponding decaying regions. These decaying regions are in accordance with the signature table of $\mathrm{Re}\left(2i\theta(z)\right)$.
$\Sigma^*_{j}$ denote the conjugate contours of $\Sigma_{j}$
respectively, as illustrated in Fig. \ref{fig4}.
For $j=1,\cdots,4$, we define
\begin{align}
\Omega_1:&=\{z\in\mathbb{C}:0\leqslant(\arg z-z_1)\leqslant\varphi_0\}\nonumber\\
\Omega_2:&=\{z\in\mathbb{C}:\pi-\varphi_0\leqslant(\arg z-z_2)\leqslant\pi,|\operatorname{Re}(z-z_2)|\leqslant(z_2-z_3)/2\},\nonumber\\
\Omega_3:&=\{z\in\mathbb{C}:0\leqslant(\arg z-z_3)\leqslant\varphi_0,|\operatorname{Re}(z-z_2)|\leqslant(z_2-z_3)/2\},\nonumber\\
 \Omega_4:&=\{z\in\mathbb{C}:0\leqslant(\arg z-z_4)\leqslant\varphi_0\},\nonumber
\end{align}
where $z_{j}$ are four phase points given in \eqref{3.17} and \eqref{3.18}, and $\varphi_0$ is a fixed angle such that the following conditions hold:
\begin{itemize}
\item  $2\tan \varphi_0<z_2-z_3$
\item each $\Omega_{j}$ doesn't intersect the set $\{z\in\mathbb{C}:\operatorname{Im}\theta(z)=0\}$,
\item each $\Omega_{j}$ doesn't intersect any small disks $\mathbb{D}_n$ and $\mathbb{D}_n^*$, $n=1,\cdots,4N_1+2N_2$.
 \end{itemize}
 To determine the decaying properties of the oscillating factors $e^{\pm2it\theta(z)}$ , we estimate
$\mathrm{Re}(2it\theta(z))$ on $\Omega$.
\begin{lemma}\label{lem3.5}
Let $(x,t)\in\mathcal{P}_{I}$, then the following estimates for $\theta(z)$ defined in \eqref{2.38} hold.
\begin{figure}[h]
	\centering
	\begin{tikzpicture}[node distance=2cm]
		\draw[yellow!30, fill=yellow!20] (0,1.2)--(1.5,0)--(2.5,0)--(4,1.2)--(4,0)--(-4,0)--(-4,1.2)--(-2.5,0)--(-1.5,0)--(0,1.2);
		\draw[blue!30, fill=green!20] (0,-1.2)--(1.5,0)--(2.5,0)--(4,-1.2)--(4,0)--(-4,0)--(-4,-1.2)--(-2.5,0)--(-1.5,0)--(0,-1.2);
		\draw[dash pattern={on 0.84pt off 2.51pt}][->](-3.6,0)--(4,0)node[right]{ $\operatorname{Im} z$};
		\draw[dash pattern={on 0.84pt off 2.51pt}][->](0,-1.8)--(0,1.8)node[above]{ $\operatorname{Im} z$};
		\draw(2.5,0)--(4,1.2)node[above]{\footnotesize$\Sigma_{1}$};
		\draw[-latex](2.5,0)--(3.225,0.6);
		\draw[-latex](2.5,0)--(3.225,-0.6);
		\draw[ blue][-latex](-4,-1.2)--(-3.225,-0.6);
		\draw[ red][-latex](-4,1.2)--(-3.225,0.6);

		\draw(2.5,0)--(4,-1.2)node[below]{\footnotesize$\Sigma^*_{1}$};
		\draw(-2.5,0)--(-4,1.2)node[above]{\footnotesize$\Sigma_{4}$};
		\draw(-2.5,0)--(-4,-1.2)node[below]{\footnotesize$\Sigma^*_{4}$};
		\draw[-latex](-1.5,0)--(-0.75,0.6)node[above]{\footnotesize$\Sigma_{3}$};
		\draw[-latex](-1.5,0)--(-0.75,-0.6)node[below]{\footnotesize$\Sigma^*_{3}$};
		\draw[-latex](0,1.2)--(0.75,0.6)node[above]{\footnotesize$\Sigma_{2}$};
 		\draw[-latex](0,-1.2)--(0.75,-0.6)node[below]{\footnotesize$\Sigma^*_{2}$};
        \coordinate (I) at (-2,0);
        \fill(I) circle (1pt) node[below] {$-1$};
        \coordinate (b) at (1.5,0);
			\fill (b) circle (1pt) node[below] {$z_2$};
			\coordinate (f) at (2.5,0);
			\fill (f) circle (1pt) node[below] {$z_1$};
        \coordinate (q) at (-2.5,0);
        \fill (q) circle (1pt) node[below] {$z_4$};
        \coordinate (w) at (-1.5,0);
			\fill (w) circle (1pt) node[below] {$z_3$};
			\coordinate (e) at (2,0);
			\fill (e) circle (1pt) node[below] {$1$};
		\draw(0,1.2)--(1.5,0);
		\draw(0,1.2)--(0,-1.2);
		\draw(0,-1.2)--(1.5,0);
		\draw(0,1.2)--(-1.5,0);
		\draw(0,-1.2)--(-1.5,0);
		\draw[-latex](-2.2,0)--(-2.1,0);
		\draw[-latex](1.8,0)--(1.9,0);
		\draw[ red][-latex](0,0)--(0,0.375);
		\draw[ blue][-latex](0,0)--(0,-0.375);
		\draw(-2.5,0)--(-1.5,0);
		\draw(2.5,0)--(1.5,0);
\draw[ red][-latex](-1.5,0)--(-0.75,0.6);
\draw[ red][-latex](0,1.2)--(0.75,0.6);
\draw[ red][-latex](2.5,0)--(3.225,0.6);
        \draw [ red] (-4,1.2)--(-2.5, 0);
        \draw [ red] (-1.5,0)--(0, 1.2);
        \draw [ red] (0,1.2)--(1.5,0);
        \draw [ red] (2.5,0)--(4,1.2);
\draw[ blue][-latex](-1.5,0)--(-0.75,-0.6);
\draw[ blue][-latex](0,-1.2)--(0.75,-0.6);
\draw[ blue][-latex](2.5,0)--(3.225,-0.6);
        \draw [ blue] (-4,-1.2)--(-2.5, 0);
        \draw [ blue] (-1.5,0)--(0, -1.2);
        \draw [ blue] (0,-1.2)--(1.5,0);
        \draw [ blue] (2.5,0)--(4,-1.2);
        \draw [ red] (0,0)--(0,1.2);
        \draw [ blue] (0,0)--(0,-1.2);
		\coordinate (C) at (-0.2,2.2);
		\coordinate (D) at (3.45,0.15);
		\fill (D) circle (0pt) node[right] {\tiny $\Omega_{1}$};
		\coordinate (D2) at (3.45,-0.15);
		\fill (D2) circle (0pt) node[right] {\tiny $\Omega^*_{1}$};
		\coordinate (k) at (1,0.2);
		\fill (k) circle (0pt) node[left] {\tiny $\Omega_{2}$};
		\coordinate (k2) at (1,-0.2);
		\fill (k2) circle (0pt) node[left] {\tiny $\Omega^*_{2}$};
		\coordinate (k) at (0,0.2);
		\fill (k) circle (0pt) node[left] {\tiny $\Omega_{3}$};
		\coordinate (k2) at (0,-0.2);
		\fill (k2) circle (0pt) node[left] {\tiny $\Omega^*_{3}$};
		\coordinate (D3) at (-3.45,0.15);
		\fill (D3) circle (0pt) node[left] {\tiny $\Omega_{4}$};
		\coordinate (D4) at (-3.45,-0.15);
		\fill (D4) circle (0pt) node[left] {\tiny $\Omega^*_{4}$};
		\coordinate (I) at (0.2,0);
		\fill (I) circle (0pt) node[below] {$0$};
		\node at (0,0.9) {\footnotesize $\Sigma_{2,3}$};
		\node at (0,-0.92) {\footnotesize $\Sigma^*_{2,3}$};
	\end{tikzpicture}
	\caption{  Open the jump contour $\mathbb{R}\setminus([z_{4},z_{3}]\cup[z_{2},z_{1}])$ along red rays and blue rays. In the green regions, $\mathrm{Re}\left(2it\theta(z)\right)<0$,
while in the yellow regions, $\mathrm{Re}\left(2it\theta(z)\right)>0$. }
	\label{fig4}
\end{figure}
\begin{enumerate}[($i$)]
  \item If $|\operatorname{Re}z(1-|z|^{-2})|\leqslant2$, we have
  \begin{equation}\label{3.26}
    \operatorname{Im}\theta(z)\leqslant-\frac{1}{2}\operatorname{Im}z(\operatorname{Re}z-z_j)^2,\quad z\in\Omega_j,\quad j=1,\cdots,4,
\end{equation}
\begin{equation}\label{3.27}
    \operatorname{Im}\theta(z)\geqslant\frac{1}{2}\operatorname{Im}z(\operatorname{Re}z-z_j)^2,\quad z\in\Omega_j^*,\quad j=1,\cdots,4.
\end{equation}
  \item If $|\operatorname{Re}z(1-|z|^{-2})|\geqslant2$, we have
  \begin{equation}\label{3.28}
    \operatorname{Im}\theta(z)\leqslant-2\xi\operatorname{Im}z,\quad z\in\Omega_j,\quad j=1,\cdots,4,
\end{equation}
\begin{equation}\label{3.29}
    \operatorname{Im}\theta(z)\geqslant2\xi\operatorname{Im}z,\quad z\in\Omega_j^*,\quad j=1,\cdots,4.
\end{equation}
\end{enumerate}
\end{lemma}
\begin{proof}
We only prove the estimates related to $z\in\Omega_1$. The derivations of the corresponding results in other regions can be carried out in an analogous manner.

For $z\in\Omega_{1}$, we set
\begin{equation}\nonumber
    z-1/z:=u+vi,u,v\in\mathbb{R},\quad\mathrm{and}\quad\xi_1:=z_1-1/z_1\in[0,2).
\end{equation}
Thus, $u=\operatorname{Re}z(1-1/|z|^2)$,$\upsilon=\operatorname{Im}z(1+1/|z|^2)$, and by equation \eqref{2.38},
\begin{equation}\label{3.31}
    \operatorname{Im}t\theta(z)=2tv\left[F(u,v)-\frac{\xi}{8}\right],\quad F(u,v):=\frac{4-v^2-u^2}{v^4-2(4-u^2)v^2+(4+u^2)^2}.
\end{equation}
From the definition of $z_1$ in \eqref{3.17}, it follows that $\xi=\frac{8(4-\xi_1^2)}{(4+\xi_1^2)^2}$. It is easily seen that
\begin{equation}\nonumber
    F(u,v)\leqslant\begin{cases}F(u,0),&u^2\leqslant4,\\0,&u^2\geqslant4.\end{cases}
\end{equation}
We further have
\begin{equation}\nonumber
    F(u,0)-\frac{\xi}{8}=\frac{4-u^2}{(4+u^2)^2}-\frac{4-\xi_1^2}
    {\left(4+\xi_1^2\right)^2}=-(u^2-\xi_1^2)\frac{\left(4-\xi_1^2\right)
    \left(8+u^2+\xi_1^2\right)+\left(4+\xi_1^2\right)^2}{\left(4+u^2\right)^2\left(4+\xi_1^2\right)^2}.
\end{equation}
Inserting the above two formulae into \eqref{3.31}, we then obtain \eqref{3.28} and \eqref{3.29} from the fact that
$u^2-\xi_1^2\geqslant(\operatorname{Re}z-z_1)^2$ and the ranges of $u$ and $\xi$.
\end{proof}
Since the function $d(z)$ defined in equation \eqref{3.22} is not an analytic function, the current approach is to introduce the functions $R_j(z)$, where $j=1,\cdots,4,$, and these functions will be subject to the following boundary conditions:
\begin{lemma}\label{lem3.6}
Let initial data $u(x)\in H^{4,2}(\mathbb{R})$. Then it is possible to define functions $R_{j}:\Omega_j\cup\partial\Omega_j\to\mathbb{C},j=1,2,3,4$, continuous on $\Omega_j\cup\partial\Omega_j$, with continuous first partials on $\Omega_{j}$, and boundary values
\begin{equation}\nonumber
    R_j(z)=\begin{cases}\bar{d}(z),&z\in\mathbb{R},\\\bar{d}(z_j)+\bar{d}^{\prime}(z_j)(z-z_j),&z\in\Sigma_j.\end{cases}
\end{equation}
One can give an explicit construction of each $R_j(z)$. assume that $\mathcal{X}$ is a function belonging to $C_0^\infty(\mathbb{R})$ for which
\begin{equation}\label{3.36}
    \mathcal{X}(x):=\begin{cases}0,&x\leqslant\frac{\varphi_0}{3},\\1,&x\geqslant\frac{2\varphi_0}{3}.\end{cases}
\end{equation}
 We have for $j=1,\cdots,4,$
 \begin{align}
 |R_j(z)|&\lesssim\sin^2\left(\frac{\pi}{2\varphi_0}\arg\left(z-z_j\right)\right)+\left(1+\mathrm{Re}(z)^2\right)^{-1/2},\label{3.37}\\
 |\bar{\partial}R_j(z)|&\lesssim|\operatorname{Re}z-z_j|^{-1/2}+\sin\left(\frac{\pi}{2\varphi_0}\arg\left(z-z_j\right)\mathcal{X}(\arg\left(z-z_j\right))\right),\label{3.39}\\
  |\bar{\partial}R_j(z)|&\lesssim|\operatorname{Re}z-z_j|^{1/2},\label{3.38}\\
  |\bar{\partial}R_j(z)|&\lesssim1.\label{3.40}
 \end{align}

\end{lemma}
\begin{proof}

Without restricting the generality of our argument, we take $j=1$ as our case of study. The proofs for $j=2,3,4,$ can be carried out in a comparable fashion.

The bound in \eqref{3.37} can be obtained from the proof similar to \cite{boo-36}, we omit
the details here. To show \eqref{3.38}-\eqref{3.40}, note that for $z=u+v i=le^{\varphi i}+k_{1}\in\Omega_{1}$, we have
\begin{equation}\nonumber
\bar{\partial}=\frac{e^{i\varphi}}{2}(\partial_{l}+il^{-1}\partial_{\varphi})=\frac{1}{2}(\partial_{u}+\partial_{v}).
\end{equation}
Applying $\bar{\partial}$ operator to $R_1$,
\begin{equation}\nonumber
    \begin{aligned}R_1(z):=&\left[\bar{d}(\operatorname{Re}z)-\bar{d}(z_1)-
    \bar{d}^{\prime}(z_1)\operatorname{Re}(z-z_1)\right]\cos\left(\frac{\pi\arg(z-z_1)
    \mathcal{X}(\arg(z-z_1))}{2\varphi_0}\right)\\&+\bar{d}(z_1)+\bar{d}^{\prime}(z_1)(z-z_1),\end{aligned}
\end{equation}
\begin{equation}\label{3.42}
    \begin{aligned}\bar{\partial}R_1(z)=&\frac{1}{2}\left(\bar{d}^{\prime}(u)-\bar{d}^{\prime}(z_1)\right)
    \cos\left(\frac{\pi\varphi\mathcal{X}(\varphi)}{2\varphi_0}\right)\\&-\frac{ie^{i\varphi}}{2l}
    [\bar{d}(u)-\bar{d}(z_1)-\bar{d}^{\prime}(z_1)(u-z_1)]\frac{\pi\mathcal{X}^{\prime}(\varphi)}{2\varphi_0}
    \sin\left(\frac{\pi}{2\varphi_0}\varphi\mathcal{X}(\varphi)\right).\end{aligned}
\end{equation}
On one hand, it follows from H\"older's inequality that
\begin{equation}\label{3.43}
    |\bar{d}^{\prime}(u)-\bar{d}^{\prime}(z_1)|=\left|\int_{z_1}^u\bar{d}^{\prime\prime}(\zeta)d\zeta\right|\leqslant\|\bar{d}^{\prime\prime}\|_2|u-z_1|^{1/2},
\end{equation}
or we have
\begin{equation}\label{3.44}
    |\bar{d}^{\prime}(u)-\bar{d}^{\prime}(z_1)|=\left|\int_{z_1}^u\bar{d}^{\prime\prime}(\zeta)d\zeta\right|
    =\left|\int_{z_1}^u\zeta^{-1}\zeta\bar{d}^{\prime\prime}(\zeta)d\zeta\right|\leqslant\|\bar{d}^{\prime\prime}\|_{2,1}|u-z_1|^{-1/2}.
\end{equation}
On the other hand, as
\begin{equation}\label{3.45}
    |\bar{d}^\prime(u)-\bar{d}^\prime(z_1)|\leqslant2\|\bar{d}^\prime\|_\infty,
\end{equation}
we have
\begin{equation}\label{3.46}
    |\bar{d}(u)-\bar{d}(z_1)-\bar{d}^{\prime}(z_1)(u-z_1)|
    =\left|\int_{z_1}^u\bar{d}^{\prime}(\zeta)-\bar{d}^{\prime}(z_1)d\zeta\right|\leqslant2\|\bar{d}^{\prime}\|_\infty|u-z_1|,
\end{equation}
or again by H\"older's inequality,
\begin{equation}\label{3.47}
    \left|\int_{z_1}^u\bar{d}^{\prime}(\zeta)-\bar{d}^{\prime}(z_1)d\zeta\right|=\left|\int_{z_1}^u\int_{z_1}^\zeta\bar{d}^{\prime\prime}(\eta)d\eta d\zeta\right|\leqslant\|\bar{d}^{\prime\prime}\|_2|u-z_1|^{3/2}.
\end{equation}

As a consequence, on account of \eqref{3.42}, from \eqref{3.43} and \eqref{3.47}, we obtain
\begin{equation}\nonumber
    |\bar{\partial}R_j(z)|\lesssim|\operatorname{Re}z-z_j|^{1/2},
\end{equation}
which is \eqref{3.38}.

We can also obtain \eqref{3.39} from \eqref{3.44} and \eqref{3.46}, obtain \eqref{3.40} from \eqref{3.45} and \eqref{3.46}. We omit it here.
\end{proof}

Additionally, by leveraging the generalizations of Lemma \ref{lem3.6}, we formulate a matrix function
\begin{equation}\nonumber
    R^{(2)}(z):=R^{(2)}\left(z;\xi\right)=\begin{cases}
    \begin{pmatrix}1&0\\-R_j(z)e^{-2it\theta(z)}&1\end{pmatrix},&z\in\Omega_j,j=1,\cdots,4,\\
    \begin{pmatrix}1&R_j^*(z)e^{2it\theta(z)}\\0&1\end{pmatrix},&z\in\Omega_j^*,j=1,\cdots,4,\\I,&\text{elsewhere,}\end{cases}
\end{equation}
and a contour
\begin{equation}\label{3.49}
    \Sigma^{(3)}:=\Sigma_{2,3}\cup\Sigma_{2,3}^*\cup\left(\bigcup_{k=1,\cdots,4}(\Sigma_k\cup\Sigma_k^*)\right)\cup(z_4,z_3)\cup(z_2,z_1).
\end{equation}
An illustration of this can be found in Fig.\ref{fig5}.
Then, the new matrix function defined by
\begin{equation}\label{3.50}
    M^{(3)}(z)=M^{(2)}(z)R^{(2)}(z)
\end{equation}
\begin{figure}[h]
	\centering
	\begin{tikzpicture}[node distance=2cm]
		\draw[dash pattern={on 0.84pt off 2.51pt}][->](-3.6,0)--(4,0)node[right]{ $\operatorname{Im} z$};
		\draw[dash pattern={on 0.84pt off 2.51pt}][->](0,-1.8)--(0,1.8)node[above]{ $\operatorname{Im} z$};
		\draw(2.5,0)--(4,1.2)node[above]{\footnotesize$\Sigma_{1}$};
		\draw[-latex](2.5,0)--(3.225,0.6);
		\draw[-latex](2.5,0)--(3.225,-0.6);
		\draw[ blue][-latex](-4,-1.2)--(-3.225,-0.6);
		\draw[ red][-latex](-4,1.2)--(-3.225,0.6);

		\draw(2.5,0)--(4,-1.2)node[below]{\footnotesize$\Sigma^*_{1}$};
		\draw(-2.5,0)--(-4,1.2)node[above]{\footnotesize$\Sigma_{4}$};
		\draw(-2.5,0)--(-4,-1.2)node[below]{\footnotesize$\Sigma^*_{4}$};
		\draw[-latex](-1.5,0)--(-0.75,0.6)node[above]{\footnotesize$\Sigma_{3}$};
		\draw[-latex](-1.5,0)--(-0.75,-0.6)node[below]{\footnotesize$\Sigma^*_{3}$};
		\draw[-latex](0,1.2)--(0.75,0.6)node[above]{\footnotesize$\Sigma_{2}$};
 		\draw[-latex](0,-1.2)--(0.75,-0.6)node[below]{\footnotesize$\Sigma^*_{2}$};
        \coordinate (I) at (-2,0);
        \fill(I) circle (1pt) node[below] {$-1$};
        \coordinate (b) at (1.5,0);
			\fill (b) circle (1pt) node[below] {$z_2$};
			\coordinate (f) at (2.5,0);
			\fill (f) circle (1pt) node[below] {$z_1$};
        \coordinate (q) at (-2.5,0);
        \fill (q) circle (1pt) node[below] {$z_4$};
        \coordinate (w) at (-1.5,0);
			\fill (w) circle (1pt) node[below] {$z_3$};
			\coordinate (e) at (2,0);
			\fill (e) circle (1pt) node[below] {$1$};
		\draw(0,1.2)--(1.5,0);
		\draw(0,1.2)--(0,-1.2);
		\draw(0,-1.2)--(1.5,0);
		\draw(0,1.2)--(-1.5,0);
		\draw(0,-1.2)--(-1.5,0);
		\draw[-latex](-2.2,0)--(-2.1,0);
		\draw[-latex](1.8,0)--(1.9,0);
		\draw[ red][-latex](0,0)--(0,0.375);
		\draw[ blue][-latex](0,0)--(0,-0.375);
		\draw(-2.5,0)--(-1.5,0);
		\draw(2.5,0)--(1.5,0);
\draw[ red][-latex](-1.5,0)--(-0.75,0.6);
\draw[ red][-latex](0,1.2)--(0.75,0.6);
\draw[ red][-latex](2.5,0)--(3.225,0.6);
        \draw [ red] (-4,1.2)--(-2.5, 0);
        \draw [ red] (-1.5,0)--(0, 1.2);
        \draw [ red] (0,1.2)--(1.5,0);
        \draw [ red] (2.5,0)--(4,1.2);
\draw[ blue][-latex](-1.5,0)--(-0.75,-0.6);
\draw[ blue][-latex](0,-1.2)--(0.75,-0.6);
\draw[ blue][-latex](2.5,0)--(3.225,-0.6);
        \draw [ blue] (-4,-1.2)--(-2.5, 0);
        \draw [ blue] (-1.5,0)--(0, -1.2);
        \draw [ blue] (0,-1.2)--(1.5,0);
        \draw [ blue] (2.5,0)--(4,-1.2);
        \draw [ red] (0,0)--(0,1.2);
        \draw [ blue] (0,0)--(0,-1.2);
		\coordinate (C) at (-0.2,2.2);
		\coordinate (D) at (3.45,0.15);
		\fill (D) circle (0pt) node[right] {\tiny $\Omega_{1}$};
		\coordinate (D2) at (3.45,-0.15);
		\fill (D2) circle (0pt) node[right] {\tiny $\Omega^*_{1}$};
		\coordinate (k) at (1,0.2);
		\fill (k) circle (0pt) node[left] {\tiny $\Omega_{2}$};
		\coordinate (k2) at (1,-0.2);
		\fill (k2) circle (0pt) node[left] {\tiny $\Omega^*_{2}$};
		\coordinate (k) at (0,0.2);
		\fill (k) circle (0pt) node[left] {\tiny $\Omega_{3}$};
		\coordinate (k2) at (0,-0.2);
		\fill (k2) circle (0pt) node[left] {\tiny $\Omega^*_{3}$};
		\coordinate (D3) at (-3.45,0.15);
		\fill (D3) circle (0pt) node[left] {\tiny $\Omega_{4}$};
		\coordinate (D4) at (-3.45,-0.15);
		\fill (D4) circle (0pt) node[left] {\tiny $\Omega^*_{4}$};
		\coordinate (I) at (0.2,0);
		\fill (I) circle (0pt) node[below] {$0$};
		\node at (0,0.9) {\footnotesize $\Sigma_{2,3}$};
		\node at (0,-0.92) {\footnotesize $\Sigma^*_{2,3}$};
	\end{tikzpicture}
	\caption{ The jump contours $\Sigma^{(3)}$. }
	\label{fig5}
\end{figure}
satisfies the following mixed $\bar{\partial}$-RH problem.
\begin{dbar-RHP}\label{RHP6}
    		Find a matrix valued function $M^{(3)}(z)\triangleq M^{(3)}(z;y,t)$ admits:
    		\begin{enumerate}[($i$)]
    			\item Continuity:~$M^{(3)}(z)$ is continuous in $\mathbb{C}\setminus\left(\Sigma^{(3)}\cup\left\{\zeta_{n},\bar{\zeta}_{n}\right\}_{n\in\Lambda}\right)$;
                \item Symmetry: $M^{(3)}(z)=\sigma_3\overline{M^{(3)}(-\bar{z})}\sigma_3=F^{-2}\overline{M^{(3)}(-\bar{z}^{-1})}=F^2\sigma_3M^{(3)}(-z^{-1})\sigma_3$;
    			\item Jump condition:
    			\begin{equation}\nonumber
    M_+^{(3)}(z)=M_-^{(3)}(z)V^{(3)}(z),
\end{equation}
    where
    		\begin{equation}\label{3.52}
    \begin{aligned}V^{(3)}(z)&=\begin{cases}\begin{pmatrix}1&e^{2it\theta(z)}d(z)\\0&1\end{pmatrix}
    \begin{pmatrix}1&0\\e^{-2it\theta(z)}\bar{d}(z)&1\end{pmatrix},&z\in(z_4,z_3)\cup(z_2,z_1),\\
    R^{(2)}(z)^{-1},&z\in\Sigma_j,j=1,2,3,4,\\R^{(2)}(z),&z\in\Sigma_j^*,j=1,2,3,4,\\
    \begin{pmatrix}1&0\\(R_2(z)-R_3(z))e^{-2i\theta(z)}&1\end{pmatrix},&z\in\Sigma_{2,3},\\
    \begin{pmatrix}1&(R_2^*(z)-R_3^*(z))e^{2i\theta(z)}\\0&1\end{pmatrix},&z\in\Sigma_{2,3}^*.\end{cases}\end{aligned}
\end{equation}
    			\item Asymptotic behavior:
                 \begin{equation}\nonumber
    M^{(2)}(z)=I+\mathcal{O}(z^{-1}),\quad z\to\infty,
\end{equation}
    			\begin{equation}\nonumber
    \begin{aligned}M^{(2)}(z)=&F^{-1}\left[I+(z-i)\begin{pmatrix}0&-\frac{1}{2}(u+u_x)\\
    -\frac{1}{2}(u-u_x)&0\end{pmatrix}\right]\\&e^{\frac{1}{2}c_{+}\sigma_{3}}T(i)^{\sigma_{3}}
    \left(I-I_{0}\sigma_{3}(z-i)\right)+\mathcal{O}\left((z-i)^{2}\right)+\mathcal{O}\left((z-i)^{2}\right);\end{aligned}
\end{equation}

            \item For $z\in\mathbb{C}$, we have the $\bar{\partial}$-derivative relation
            \begin{equation}\nonumber
    \bar{\partial}M^{(3)}(z)=M^{(3)}(z)\bar{\partial}R^{(2)}(z),
\end{equation}
where
\begin{equation}\label{3.56}
    \bar{\partial}R^{(2)}(z)=\begin{cases}\begin{pmatrix}0&0\\
    -\bar{\partial}R_j(z)e^{-2it\theta(z)}&0\end{pmatrix},&z\in\Omega_j,j=1,\cdots,4,\\
    \begin{pmatrix}0&\bar{\partial}R_j^*(z)e^{2it\theta(z)}\\0&0\end{pmatrix},&z\in\Omega_j^*,j=1,\cdots,4,\\0,&\text{elsewhere.}\end{cases}
\end{equation}

\item Residue conditions: If $j\notin\Lambda$ for $j=1,2,\ldots,N,$, then $M^{(3)}$ is continuous in the region $\mathbb{C}\setminus\Sigma^{(1)}$. If there exist $\zeta_n$ and $\bar{\zeta}_n$ for $n\in\Lambda$, then $M^{(3)}$ admits the residue conditions:
    \begin{equation}\operatorname*{\mathrm{Res}}_{z=\zeta_n}M^{(3)}(z)=\lim_{z\to\zeta_n}M^{(3)}(z)\begin{pmatrix}0&0\\C_ne^{-2it\theta_n}T^2(\zeta_n)&0\end{pmatrix},\end{equation}
    \begin{equation}\operatorname*{\mathrm{Res}}_{z=\bar{\zeta}_n}M^{(3)}(z)=\lim_{z\to\bar{\zeta}_n}M^{(3)}(z)
    \begin{pmatrix}0&-\bar{C}_nT^{-2}(\bar{\zeta}_n)e^{2it\bar{\theta}_n}\\0&0\end{pmatrix}.\end{equation}
    		\end{enumerate}
    	\end{dbar-RHP}

Up to now, in order to examine the long-term asymptotics of the initial RH problem \ref{RHP2} for $M(z)$, we have derived the hybrid $\bar{\partial}$-RH problem \ref{RHP6} for $M^{(3)}(z)$. Subsequently, we will proceed with the construction of the solution $M^{(3)}(z)$ in the following manner:
\begin{itemize}
  \item We first eliminate the $\bar{\partial}$ component from the solution $M^{(3)}(z)$. Subsequently, we establish the existence of a solution for the resultant pure RH problem. In addition to this, we compute the asymptotic expansion of the obtained solution.
  \item Conjugating off the solution of the first step, a pure $\bar{\partial}$-problem can be obtained.
Subsequently, we prove the existence of a solution for this newly obtained problem and bound its size.
\end{itemize}
\subsection{Asymptotic analysis on a pure RH problem}\label{s:3.2}
In the present section, we construct a solution $M^{rhp}(z)$ for the pure RH problem that is part of the $\bar{\partial}$-RH problem \ref{RHP6} associated with the function $M^{(3)}(z)$. Additionally, we calculate the asymptotic expansion of this solution $M^{rhp}(z)$. By excluding the $\bar{\partial}$ component of $M^{(3)}(z)$, the function $M^{rhp}(z)$ adheres to the following pure RH problem.
\begin{RHP}\label{RHP7}
    		Find a matrix valued function $M^{rhp}(z)\triangleq M^{rhp}(z;y,t)$ admits:
    		\begin{enumerate}[($i$)]
    			\item Analyticity:~$M^{rhp}(z)$ is analytical in $\mathbb{C}\setminus\left(\Sigma^{(3)}\cup\left\{\zeta_{n},\bar{\zeta}_{n}\right\}_{n\in\Lambda}\right)$;

                \item Symmetry:$M^{rhp}(z)=\sigma_3\overline{M^{rhp}(-\bar{z})}\sigma_3=F^{-2}\overline{M^{rhp}(-\bar{z}^{-1})}=F^2\sigma_3M^{rhp}(-z^{-1})\sigma_3$;
    			\item Jump condition:
    			\begin{equation}\nonumber
    M^{rhp}_+(z)=M^{rhp}_-(z)V^{(3)}(z),
\end{equation}
	where $V^{(3)}(z)$ is given by \eqref{3.52}.
    			\item Asymptotic behavior:
                $M^{rhp}(z)$ has the same asymptotics with $M^{(3)}(z)$;
\item Residue conditions:  $M^{rhp}(z)$ possesses the same residue condition with $M^{3}(z)$.
    		\end{enumerate}
    	\end{RHP}

\subsubsection{Solving the pure RH problem}\label{s:3.3.1}
Firstly, we define the small neighborhood $U_z(\pm1)$ (shown in Fig. \ref{fig6}) as
\begin{equation}\nonumber
U_z(\pm1)=\left\{z:|z-(\pm1)|<\varrho/2\right\}.
\end{equation}
Then, we could decompose $M^{rhp}$ into there parts
\begin{equation}\label{3.65}
    M^{rhp}(z)=\begin{cases}
    E(z)M^{out}(z),\quad z\in\mathbb{C}\setminus(U_z(1)\cup U_z(-1)),\\
    E(z)M^{out}(z)M^{rhp1}(z),\quad z\in U_z(1),&\\
    E(z)M^{out}(z)M^{rhp2}(z),\quad z\in U_z(-1).\end{cases}
\end{equation}
On the basis of the definition of $\rho$, it implies that $M^{rhp1}(z)$ possesses no poles in $U_z(1)$. Additionally, $M^{out}(z)$ solves a model RH problem, $M^{rhp1}(z)$ and $M^{rhp2}(z)$ can be solved by using a known Painlev\'{e} II model in $U_z(1)$ and $U_z(-1)$, and $E(z)$ is an error function which is a solution of a small-norm RH problem.

\begin{figure}[h]
	\centering
	\begin{tikzpicture}[node distance=2cm]
		\draw[dash pattern={on 0.84pt off 2.51pt}][->](-3.6,0)--(4,0)node[right]{ $\operatorname{Im} z$};
		\draw[dash pattern={on 0.84pt off 2.51pt}][->](0,-1.8)--(0,1.8)node[above]{ $\operatorname{Im} z$};

\draw(0.65,0.68)--(1.5,0);
\draw[-latex](0.65,0.68)--(1.075,0.34);
\draw(0.65,-0.68)--(1.5,0);
\draw[-latex](0.65,-0.68)--(1.075,-0.34);
\draw(2.5,0)--(3.35,0.68);
\draw[-latex](2.5,0)--(3.125,0.5);
\draw(2.5,0)--(3.35,-0.68);
\draw[-latex](2.5,0)--(3.125,-0.5);

\draw(-0.65,0.68)--(-1.5,0);
\draw[-latex](-1.5,0)--(-0.875,0.5);
\draw(-0.65,-0.68)--(-1.5,0);
\draw[-latex](-1.5,0)--(-0.875,-0.5);
\draw(-2.5,0)--(-3.35,0.68);
\draw[-latex](-3.35,0.68)--(-2.925,0.34);
\draw(-2.5,0)--(-3.35,-0.68);
\draw[-latex](-3.35,-0.68)--(-2.925,-0.34);
\draw(-2.5,0)--(-1.5,0);
\draw[-latex](-2.5,0)--(-2,0);
\draw(1.5,0)--(2.5,0);
\draw[-latex](1.5,0)--(2,0);		
        \coordinate (I) at (-2,0);
        \fill(I) circle (1pt) node[below] {$-1$};
        \coordinate (b) at (1.5,0);
			\fill (b) circle (1pt) node[below] {$z_2$};
			\coordinate (f) at (2.5,0);
			\fill (f) circle (1pt) node[below] {$z_1$};
        \coordinate (q) at (-2.5,0);
        \fill (q) circle (1pt) node[below] {$z_4$};
        \coordinate (w) at (-1.5,0);
			\fill (w) circle (1pt) node[below] {$z_3$};
			\coordinate (e) at (2,0);
			\fill (e) circle (1pt) node[below] {$1$};

        \draw[thick,blue] (2,0) circle (1.5);
        \draw[thick,blue] (-2,0) circle (1.5);
        \node at (2,1) {\footnotesize $ U^{(1)}$};
        \node at (-2,1) {\footnotesize $ U^{(2)}$};
		\coordinate (C) at (-0.2,2.2);
		
		\coordinate (I) at (0.2,0);
		\fill (I) circle (0pt) node[below] {$0$};

	\end{tikzpicture}
	\caption{ The small neighborhood $U_z(\pm1)$. }
	\label{fig6}
\end{figure}

Next, we will study $M^{out}(z)$, $M^{rhpj}(z)$, and $E(z)$ separately. For $M^{out}(z)$, we have
\begin{equation}\nonumber
M^{out}(z)=M_{\Lambda}^{out}(z)(I+\mathcal{O}(e^{-ct})),~~t\to\infty,\end{equation}
$M^{out}(z)$ ignores jumping, which is $M_{\Lambda}^{out}(z)$. $M_{\Lambda}^{out}(z)$ solves the following RH problem.
\begin{RHP}\label{RHP3.9}
    		Find a matrix valued function $M_{\Lambda}^{out}(z)=M_{\Lambda}^{out}(z,y,t)$ admits:
    		\begin{enumerate}[($i$)]
    			\item Analyticity:~$M_{\Lambda}^{out}(z)$ is analytical in $\mathbb{C}\setminus\left\{\zeta_n,\bar{\zeta}_n\right\}_{n\in\Lambda}$;

                \item Symmetry:$M_{\Lambda}^{out}(z)=\sigma_3\overline{M_{\Lambda}^{out}(-\bar{z})}\sigma_3=F^{-2}\overline{M_{\Lambda}^{out}(-\bar{z}^{-1})}=F^2\sigma_3M_{\Lambda}^{out}(-z^{-1})\sigma_3$;
    			\item Asymptotic behavior:
                \begin{equation}\nonumber
                M_\Lambda^{out}(z)=I+\mathcal{O}(z^{-1}),\quad z\to\infty;\end{equation}
\item Residue conditions:$M_{\Lambda}^{out}(z)$ has simple poles at each point $\zeta_n$ and $\bar{\zeta}_n$ for $n\in\Lambda$ with:
\begin{equation}\nonumber
\operatorname*{\mathrm{Res}}_{z=\zeta_n}M_\Lambda^{out}(z)=\lim_{z\to\zeta_n}M_\Lambda^{out}(z)
\begin{pmatrix}0&0\\C_ne^{-2it\theta_n}T^2(\zeta_n)&0\end{pmatrix},\end{equation}
\begin{equation}\nonumber
\left.\operatorname*{Res}_{z=\bar{\zeta}_n}M_\Lambda^{out}(z)=\lim_{z\to\bar{\zeta}_n}M_\Lambda^{out}(z)
\left(\begin{array}{cc}0&-\bar{C}_nT^{-2}(\bar{\zeta}_n)e^{2it\bar{\theta}_n}\\0&0\end{array}\right.\right).\end{equation}
    		\end{enumerate}
    	\end{RHP}
Morever, denote the asymptotic expansion of $M_{\Lambda}^{out}(z)$ as $z\to i$:
\begin{equation}\label{Mout}
M_{\Lambda}^{out}(z)=M_\Lambda^{out}(i)+M_{\Lambda,1}^{out}(z-i)+\mathcal{O}((z-i)^{-2}).\end{equation}

\begin{proposition}\label{prop3.6}
The RH problem \ref{RHP3.9} possesses unique solution. This fact could be guaranteed by the Liouville's theorem. Moreover, $M_{\Lambda}^{out}(z)$ possesses equivalent a reflectionless solution
to the original RH problem \ref{RHP2} with modified scattering data $\tilde{\mathcal{D}}_{\Lambda}=\left\{0,\left\{\zeta_n,C_nT^2(\zeta_n)\right\}_{n\in\Lambda}\right\}$
as follows:

(i). If $\Lambda=\emptyset$, then
\begin{equation}\nonumber
M_{\Lambda}^{out}(z)=I.\end{equation}

(ii). If $\Lambda\neq\emptyset$, without loss of generality, we assume that there exist s discrete spectral points
belonging to $\Lambda$, i.e., $\Lambda=\{j_{1},j_{2},\ldots,j_{s}\}$, then
\begin{equation}\nonumber
M_{\Lambda}^{out}(z)=I+\sum_{k=1}^{s}\begin{pmatrix}\frac{\beta_k}{z-\zeta_{j_k}}&\frac{-\overline{\alpha_k}}{z-\zeta_{j_k}}\\
\frac{\alpha_k}{z-\zeta_{j_k}}&\frac{\overline{\beta_k}}{z-\overline{\zeta_{j_k}}}\end{pmatrix},\end{equation}
where $\alpha_k=\alpha_k(y,t)$ and $\beta_k=\beta_k(y,t)$ with linearly dependant equations:
\begin{equation}\nonumber
c_{j_k}^{-1}T(z_{j_k})^{-2}e^{-2i\theta(z_{j_k})t}\beta_k=\sum_{h=1}^{\mathcal{N}}\frac{-\overline{\alpha_h}}{\zeta_{j_k}-\overline{\zeta}_{j_h}},\end{equation}
\begin{equation}\nonumber
c_{j_k}^{-1}T(z_{j_k})^{-2}e^{-2i\theta(z_{j_k})t}\alpha_k=1+\sum_{h=1}^{\mathcal{N}}\frac{\overline{\beta_h}}{\zeta_{j_k}-\overline{\zeta}_{j_h}},\end{equation}
for $k=1,2,\ldots,s$.
\end{proposition}

Then, we could derive the soliton solutions for the $M_{\Lambda}^{out}(z)$. With reflection coefficients $r(s)\equiv0$, the scattering matrices $S(z)\equiv I$. Denote $u_{p}(y,t,\tilde{\mathcal{D}}_{\Lambda})$ is the $N(\Lambda)$-soliton with scattering data $\tilde{\mathcal{D}}_{\Lambda}=\left\{0,\left\{\zeta_{n},C_{n}T^{2}(\zeta_{n})\right\}_{n\in\Lambda}\right\}$. Consequently, the solution $u_{p}(y,t,\tilde{\mathcal{D}}_{\Lambda})$ is given by
\begin{equation}\label{ur}
u_{p}(y,t,\tilde{\mathcal{D}}_{\Lambda})=\lim_{z\to i}\frac{1}{z-i}\left(1-\frac{([M_{\Lambda}^{out}]_{11}(z)+[M_{\Lambda}^{out}]_{21}(z))([M_{\Lambda}^{out}]_{12}(z)
+[M_{\Lambda}^{out}]_{22}(z))}{([M_{\Lambda}^{out}]_{11}(i)+[M_{\Lambda}^{out}]_{21}(i))([M_{\Lambda}^{out}]_{12}(i)
+[M_{\Lambda}^{out}]_{22}(i))}\right),
\end{equation}

where
\begin{equation}\label{c+}
x(y,t;\tilde{\mathcal{D}}_\Lambda)=y+c_+^{out}(x,t;\tilde{\mathcal{D}}_\Lambda)
=y-\ln\left(\frac{[M_\Lambda^{out}]_{12}(i)+[M_\Lambda^{out}]_{22}(i)}{[M_\Lambda^{out}]_{11}(i)+[M_\Lambda^{out}]_{21}(i)}\right).\end{equation}
When $\Lambda=\emptyset$,
\begin{equation}\nonumber
u_{p}(y,t;\tilde{\mathcal{D}}_\Lambda)=c_+^{out}(y,t;\tilde{\mathcal{D}}_\Lambda)=0.
\end{equation}
When $\Lambda\neq\varnothing$ with $\Lambda=\{\zeta_{j_{k}}\}_{k=1}^{\mathcal{N}}$
\begin{equation}\label{342}
\begin{aligned}
u_{p}(y,t;\tilde{\mathcal{D}}_\Lambda)=&\left[\sum_{k=1}^{\mathcal{N}}\left(\frac{-\overline{\alpha_{k}}}{(i-\bar{\zeta}_{j_{k}})^{2}}
+\frac{\overline{\beta_{k}}}{(i-\bar{\zeta}_{j_{k}})^{2}}\right)\right]/\left[1+\sum_{k=1}^{\mathcal{N}}
\left(\frac{-\overline{\alpha_{k}}}{i-\bar{\zeta}_{j_{k}}}+\frac{\overline{\beta_{k}}}{i-\bar{\zeta}_{j_{k}}}\right)\right]\\
&+\left[\sum_{k=1}^{\mathcal{N}}\frac{\beta_k}{(i-\zeta_{j_k})^2}+\frac{\alpha_k}{(i-\zeta_{j_k})^2}\right]/
\left[1+\sum_{k=1}^{\mathcal{N}}\left(\frac{\beta_k}{i-\zeta_{j_k}}+\frac{\alpha_k}{i-\zeta_{j_k}}\right)\right],
\end{aligned}
\end{equation}
\begin{equation}\nonumber
x(y,t;\tilde{\mathcal{D}}_\Lambda)=y-\ln\left(\frac{1+\sum_{k=1}^{\mathcal{N}}\left(\frac{-\overline{\alpha_{k}}}{i-\zeta_{j_{k}}}
+\frac{\overline{\beta_{k}}}{i-\zeta_{j_{k}}}\right)}{1+\sum_{k=1}^{\mathcal{N}}\left(\frac{\beta_{k}}{i-\zeta_{j_{k}}}+\frac{\alpha_{k}}{i-\zeta_{j_{k}}}\right)}\right).
\end{equation}

Next, we study $M^{rhpj}(z)$, We analyze the local properties of the phase function $t\theta(z)$ near $z=\pm1$. As $z\to-1$, we find that
\begin{equation}\label{3.58}
    t\theta(z)=-\tilde{s}\tilde{k}-\frac{4}{3}\tilde{k}^3+\mathcal{O}(\tilde{k}^4t^{-1/3}),
\end{equation}
where
\begin{equation}\label{3.59}
    \tilde{s}=6^{-2/3}\left(\frac{y}{t}-2\right)t^{2/3},\tilde{k}=\left(\frac{9t}{2}\right)^{1/3}(z-1).
\end{equation}
Significantly, the first two terms $-\tilde{s}\tilde{k}-\frac{4}{3}\tilde{k}^3$ play a central and decisive role in the task of making the local region conform to the Painlev\'{e} model.

It is worth noting that in the transition region $\mathcal{P}_{I}$, as $t\to\infty$, according to the formula \eqref{3.21},
 $z_1$ and $z_2$ merge to $1$ and $z_3$ and $z_4$ merge to $-1$ in the $z$-plane. Correspondingly, we demonstrate
that two scaled phase points $\tilde{k}_j=\left(\frac{9t}{2}\right)^{1/3}(z_j-1)$, $j=1,2$, $\tilde{k}_j=\left(\frac{9t}{2}\right)^{1/3}(z_j+1)$, $j=3,4$, are always in a fixed interval
in the $k$-plane.
\begin{proposition}
In the transition region $\mathcal{P}_{I}$ and under scaling transformation \eqref{3.59}, then for large enough t, we have
\begin{equation}\nonumber
    |\tilde{k}_j|\leqslant\left(\frac{9}{2}\right)^{1/3}\sqrt{2C}, j=1,2,3,4.
\end{equation}
\end{proposition}
\begin{proof}
if $\xi>0$ and $-C\leqslant(\xi-2)t^{2/3}\leqslant0$, it follows from \eqref{3.19} that
\begin{equation}\nonumber
    0\leqslant s_{+}\leqslant\frac{C}{2}t^{-2/3}.
\end{equation}
Thus, by \eqref{3.17} and \eqref{3.18}
\begin{equation}\nonumber
    |z_j-1|\leqslant\sqrt{2C}t^{-1/3},j=1,2,\quad\mathrm{and}\quad|z_j+1|\leqslant\sqrt{2C}t^{-1/3},j=3,4.
\end{equation}
Finally, by \eqref{3.59} we get
\begin{equation}\nonumber
    |\tilde{k}_j|\leqslant\left(\frac{9}{2}\right)^{1/3}\sqrt{2C}, j=1,2,\quad\mathrm{and}\quad|\tilde{k}_j|\leqslant\left(\frac{9}{2}\right)^{1/3}\sqrt{2C}, j=3,4.
\end{equation}
\end{proof}
For a fix constant
\begin{equation}\nonumber
    c_0:=\min\{1/2,2(z_1-1)t^{\delta_1}\},\quad\delta_1\in(1/27,1/12),
\end{equation}
we further define two open disks
 \begin{equation}\nonumber
    U^{(1)}=\{z:|z-1|\leqslant c_0\},\quad U^{(2)}=\{z:|z+1|\leqslant c_0\}.
\end{equation}
This particularly implies that
\begin{equation}\label{3.4444}
c_0\lesssim t^{\delta_1-1/3}\to0, t\to+\infty,
\end{equation}
hence, $U^{(1)}$ and $U^{(2)}$ are two shrinking
disks with respect to $t$. As previously stated, which appeal us to construct
the solution of $M^{rhpj}(z)$ as follows:
\begin{RHP}\label{RHP8}
    		Find a matrix valued function $M^{rhpj}(z)\triangleq M^{rhpj}(z;y,t)$ admits:
    		\begin{enumerate}[($i$)]
    			\item Analyticity:~$M^{rhpj}(z)$ is analytical in $\mathbb{C}\setminus\Sigma^{(j)}$, where
    \begin{equation}\nonumber
    \Sigma^{(j)}:=U^{(j)}\cap\Sigma^{(3)}.
\end{equation}

    			\item Jump condition:
    			\begin{equation}\nonumber
    M^{rhpj}_+(z)=M^{rhpj}_-(z)V^{(3)}(z),
\end{equation}
	where $V^{(3)}(z)$ is given by \eqref{3.52}.
    			\item Asymptotic behavior:
                $M^{rhpj}(z)$ has the same asymptotics with $M^{(3)}(z)$.

    		\end{enumerate}
    	\end{RHP}
To solve the RH problem for $M^{rhp1}(z;y,t)$, we split the analysis into the following steps.\\

\textbf{Step I: Variable transformation.}
Under the change of variable \eqref{3.59}, the contour $\Sigma^{(1)}$ is
changed into a contour $\tilde{\Sigma}^{(1)}$ in the $\tilde{k}$-plane.
\begin{equation}\label{3.68}
    \tilde{\Sigma}^{(1)}:=\bigcup_{j=1,2}\left(\tilde{\Sigma}_j^{(1)}\cup\tilde{\Sigma}_j^{(1)*}\right)\cup(\tilde{k}_2,\tilde{k}_1).
\end{equation}
Here, $\tilde{k}_j=\left(\frac{9t}{2}\right)^{1/3}(z_j-1)$ and
\begin{equation}\nonumber
    \left.\tilde{\Sigma}_1^{(1)}=\begin{cases}\tilde{k}:\tilde{k}-\tilde{k}_1=le^{(\varphi_0)i},0\leqslant l\leqslant c_0\left(\frac{9t}{2}\right)^{1/3}&\end{cases}\right\},
\end{equation}
\begin{equation}\nonumber
    \left.\tilde{\Sigma}_2^{(1)}=\begin{cases}\tilde{k}:\tilde{k}-\tilde{k}_2=le^{(\pi-\varphi_0)i},0\leqslant l\leqslant c_0\left(\frac{9t}{2}\right)^{1/3}&\end{cases}\right\}.
\end{equation}
Further, RH problem \ref{RHP8} becomes the following RH problem in the $\tilde{k}$-plane.

\begin{RHP}
    		Find a matrix valued function $M^{rhp1}(\tilde{k})\triangleq M^{rhp1}(\tilde{k};y,t)$ admits:
    		\begin{enumerate}[($i$)]
    			\item Analyticity:~$M^{rhp1}(z)$ is analytical in $\mathbb{C}\setminus \tilde{\Sigma}^{(1)}$, where $\tilde{\Sigma}^{(1)}$ is given in \eqref{3.68};
    			\item Jump condition: For $\tilde{k}\in\tilde{\Sigma}^{(1)}$
    			\begin{equation}\nonumber
    M^{rhpj}_+(\tilde{k})=M^{rhpj}_-(\tilde{k})V^{rhp1}(\tilde{k}),
\end{equation}
	where
\begin{equation}\nonumber
    V^{rhp1}(\tilde{k})=\begin{cases}
    \begin{pmatrix}1+|d\left(\left(\frac{9t}{2}\right)^{-1/3}\tilde{k}+1\right)|^2&e^{2i\theta\left(\left(\frac{9t}{2}\right)^{-1/3}\tilde{k}+1\right)}
    d\left(\left(\frac{9t}{2}\right)^{-1/3}\tilde{k}+1\right)\\e^{-2i\theta\left(\left(\frac{9_t}{2}\right)^{-1/3}\tilde{k}+1\right)}
    \bar{d}\left(\left(\frac{9t}{2}\right)^{-1/3}\tilde{k}+1\right)&1\end{pmatrix} &\tilde{k}\in(\tilde{k}_2,\tilde{k}_1),\\
    \begin{pmatrix}1&0\\\left(\bar{d}(z_j)+\bar{d}^{\prime}(z_j)
    \left(\left(\frac{9t}{2}\right)^{-1/3}\tilde{k}+1-z_j\right)\right)
    e^{-2i\theta\left(\left(\frac{9t}{2}\right)^{-1/3}\tilde{k}+1\right)}&1\end{pmatrix},&\tilde{k}\in\tilde{\Sigma}_{j}^{(1)},\\
    \begin{pmatrix}1&\left(d(z_j)+d^{\prime}(z_j)
    \left(\left(\frac{9t}{2}\right)^{-1/3}\tilde{k}+1-z_j\right)\right)
    e^{2i\theta\left(\left(\frac{9t}{2}\right)^{-1/3}\tilde{k}+1\right)}\\0&1\end{pmatrix},&\tilde{k}\in\tilde{\Sigma}_{j}^{(1)*}.\end{cases}
\end{equation}
    			\item Asymptotic behavior:
                $M^{rhp1}(\tilde{k})$ has the same asymptotics with $M^{(3)}(z)$.

    		\end{enumerate}
    	\end{RHP}

\textbf{Step II: Matching the model.}
In order to use the standard model to approximate this RH problem for $M^{rhpj}(z)$, the following proposition is a prerequisite.

In view of \eqref{3.58} and the fact that $d(1)=r(1)\in\mathbb{R}$ (see \eqref{3.22}), it is natural to expect that $M^{rhp1}(\tilde{k})$ is
well-approximated by the following model RH problem for $\tilde{M}^{rhp1}(\tilde{k})$.
\begin{RHP}\label{RHP9}
    		Find a matrix valued function $\tilde{M}^{rhp1}(\tilde{k})$ admits:
    		\begin{enumerate}[($i$)]
    			\item Analyticity:~$\tilde{M}^{rhp1}(\tilde{k})$ is analytical in $\mathbb{C}\setminus\tilde{\Sigma}^{(1)}$, where $\tilde{\Sigma}^{(1)}$ is given in \eqref{3.68};
    			\item Jump condition:
    			\begin{equation}\nonumber
    \tilde{M}_+^{rhp1}(\tilde{k})=\tilde{M}_-^{rhp1}(\tilde{k})\tilde{V}^{rhp1}(\tilde{k}),
\end{equation}
	where
\begin{equation}\nonumber
    \begin{aligned}\tilde{V}^{rhp1}(\tilde{k})&=\begin{cases}\begin{pmatrix}1&r(1)e^{-2i\left(\tilde{s}\tilde{k}+\frac{4}{3}\tilde{k}^{3}\right)}\\0&1\end{pmatrix}
    \begin{pmatrix}1&0\\r(1)e^{2i\left(\tilde{s}\tilde{k}+\frac{4}{3}\tilde{k}^{3}\right)}&1\end{pmatrix},&\tilde{k}\in(\tilde{k}_2,\tilde{k}_1),\\
    \begin{pmatrix}1&0\\r(1)e^{2i\left(\tilde{s}\tilde{k}+\frac{4}{3}\tilde{k}^{3}\right)}&1\end{pmatrix},&\tilde{k}\in\tilde{\Sigma}_j^{(1)},\\
    \begin{pmatrix}1&r(1)e^{-2i\left(\tilde{s}\tilde{k}+\frac{4}{3}\tilde{k}^{3}\right)}\\0&1\end{pmatrix},&\tilde{k}\in\tilde{\Sigma}_j^{(1)*},\end{cases}\end{aligned}
\end{equation}
for $j=1,2.$
    			\item Asymptotic behavior:
                \begin{equation}\nonumber
    \tilde{M}^{rhp1}(\tilde{k})=I+\mathcal{O}(\tilde{k}^{-1}).
\end{equation}

    		\end{enumerate}
    	\end{RHP}

Subsequently, we will provide an explicit solution to the aforementioned RH problem by leveraging the Painlev\'{e} II parametrix. Consideration in the rest part is to establish the error between the
RH problems for $M^{rhp1}(\tilde{k})$ and $\tilde{M}^{rhp1}(\tilde{k})$ for large $t$.

\begin{proposition}\label{prop3.12}
As $t\to+\infty$, we have
\begin{equation}\label{3.75}
M^{rhp1}(\tilde{k})=\tilde{M}^{rhp1}(\tilde{k})+\mathcal{O}\left(t^{-1/3+4\delta_1}\right),\quad t\to+\infty.
\end{equation}
\end{proposition}

\begin{proof}
 Note that $\tilde{M}^{rhp1}(\tilde{k})$ is invertible, we
define
\begin{equation}\label{3.74}
    \Xi(\tilde{k}):=M^{rhp1}(\tilde{k})\tilde{M}^{rhp1}(\tilde{k})^{-1}.
\end{equation}
By \eqref{3.74}, it is easily seen that $\Xi(\tilde{k})$ is holomorphic for $\tilde{k}\in\mathbb{C}\setminus\tilde{\Sigma}^{(1)}$ and satisfies the jump
condition
\begin{equation}\nonumber
    \Xi_+(\tilde{k})=\Xi_-(\tilde{k})J_\Xi(\tilde{k}),\quad\tilde{k}\in\tilde{\Sigma}^{(1)},
\end{equation}
\begin{equation}\nonumber
    \begin{aligned}\Xi_+(\tilde{k})&=\Xi_-(\tilde{k})J_{\Xi}(\tilde{k})\\
    &=M^{rhp1}_+(\tilde{k})\tilde{M}_+^{rhp1}(\tilde{k})^{-1}\\
    &=M^{rhp1}_-(\tilde{k})V^{rhp1}(\tilde{k})\tilde{V}^{rhp1}(\tilde{k})^{-1}\tilde{M}_-^{rhp1}(\tilde{k})^{-1}\\
    &=M^{rhp1}_-(\tilde{k})\tilde{M}_-^{rhp1}(\tilde{k})^{-1}\tilde{M}_-^{rhp1}
    (\tilde{k})V^{rhp1}(\tilde{k})\tilde{V}^{rhp1}(\tilde{k})^{-1}\tilde{M}_-^{rhp1}(\tilde{k})^{-1}\\
    &=\Xi_-(\tilde{k})\tilde{M}_-^{rhp1}
    (\tilde{k})V^{rhp1}(\tilde{k})\tilde{V}^{rhp1}(\tilde{k})^{-1}\tilde{M}_-^{rhp1}(\tilde{k})^{-1},
    \end{aligned}
\end{equation}
thus
\begin{equation}\nonumber
    J_{\Xi}(\tilde{k}):=\tilde{M}_-^{rhp1}(\tilde{k})V^{rhp1}(\tilde{k})\tilde{V}^{rhp1}(\tilde{k})^{-1}\tilde{M}_-^{rhp1}(\tilde{k})^{-1}.
\end{equation}
Due to the fact that the function $\tilde{M}^{rhp1}(\tilde{k})$ (see \eqref{3.95}) is bounded, in order to assess the error with respect to the variable $t$, it is adequate to estimate the differences between the jump matrices $V^{rhp1}(\tilde{k})$ and $\tilde{V}^{rhp1}(\tilde{k})$, respectively.

For $\tilde{k}\in(\tilde{k}_2,\tilde{k}_1)$, we have $|e^{-2i(\tilde{s}\tilde{k}+\frac{4}{3}\tilde{k}^3)}|=|e^{2it\theta(\tilde{k})}|=1$, we get

\begin{equation}\label{3.51}
\begin{aligned}
&\left|V^{rhp1}(\tilde{k})-\tilde{V}^{rhp1}(\tilde{k})\right|
\leq\left|d\left(\left(\frac{9t}{2}\right)^{-1/3}\tilde{k}+1\right)e^{2it\theta(\tilde{k})}-d(1)e^{-2i\left(\tilde{s}\tilde{k}+\frac{4}{3}\tilde{k}^3\right)}\right|\\
&\lesssim\left|d\left(\left(\frac{9t}{2}\right)^{-1/3}\tilde{k}+1\right)-d(1)\right|
\leqslant\left\|d^{\prime}\right\|_\infty\left|\left(\frac{9t}{2}\right)^{-1/3}\tilde{k}\right|\lesssim t^{-1/3}|\tilde{k}|.
\end{aligned}\end{equation}
Further by \eqref{3.4444} and \eqref{3.59}, we have
\begin{equation}\label{kkkk}
|\tilde{k}|\lesssim t^{\delta_1},\end{equation}
which together with \eqref{3.51} gives the estimate

\begin{equation}\left|V^{rhp1}(\tilde{k})-\tilde{V}^{rhp1}(\tilde{k})\right|\lesssim t^{-\frac{1}{3}+\delta_{1}}.\end{equation}
For $\tilde{k}\in\tilde{\Sigma}_{1}^{(1)}$, as $\mathrm{Re}\left[2i(\tilde{s}\tilde{k}+\frac{4}{3}\tilde{k}^{3})\right]<0,$ by using \eqref{kkkk}, we have
$\left|e^{2i\theta(\tilde{k})}-e^{-2i(\tilde{s}\tilde{k}+\frac{4}{3}\tilde{k}^3)}\right|\lesssim t^{-1/3+4\delta_1}$.
Moreover,
\begin{equation}\nonumber
    \begin{aligned}\left|V^{rhp1}(\tilde{k})-\tilde{V}^{rhp1}(\tilde{k})\right|
&\leqslant\left|(d(z_{1})+d^{\prime}(z_{1})\left(\left(\frac{9t}{2}\right)^{-1/3}\tilde{k}+1-z_{1}\right))e^{2it\theta(\tilde{k})}
    -r(1)e^{-2i\left(\tilde{s}\tilde{k}+\frac{4}{3}\tilde{k}^3\right)}\right|\\
    &\leqslant t^{-1/3+4\delta_1}.\end{aligned}
\end{equation}
In a similar way, we have the estimate
\begin{equation}\nonumber
    \left|V^{rhp1}(\tilde{k})-\tilde{V}^{rhp1}(\tilde{k})\right|\leqslant t^{-1/3+4\delta_1}, \tilde{k}\in\tilde{\Sigma}_2^{(1)}\cup\tilde{\Sigma}_1^{(1)*}\cup\tilde{\Sigma}_2^{(1)*}.
\end{equation}
By using the small-norm theorem, we obtain the relation \eqref{3.75}.
\end{proof}

\begin{figure}[h]
	\centering
	\begin{tikzpicture}[node distance=2cm]
		\draw[dash pattern={on 0.84pt off 2.51pt}][->](-3.6,0)--(4,0)node[right]{ };
		\draw[dash pattern={on 0.84pt off 2.51pt}][->](0,-1.8)--(0,1.8)node[above]{ };
        \draw(-2.5,1)--(-1,0);
       \draw(-2.5,-1)--(-1,0);
\draw(2.5,1)--(1,0);
\draw(2.5,-1)--(1,0);
\draw(-1,0)--(1,0);
\draw[-latex](-1,0)--(-0.3,0);
\draw[-latex](-2.5,1)--(-1.725,0.5);
\draw[-latex](-2.5,-1)--(-1.725,-0.5);
\draw[color=blue] (-1.5,1.5)--(1.5,-1.5);
\draw[color=blue] (1.5,1.5)--(-1.5,-1.5);
\draw(1,0)--(2.5,1)node[above]{\footnotesize$\tilde{\Sigma}_1^{(1)}$};
\draw(1,0)--(2.5,-1)node[above]{\footnotesize$\tilde{\Sigma}_1^{(1)*}$};
\draw(-1,0)--(-2.5,1)node[above]{\footnotesize$\tilde{\Sigma}_2^{(1)}$};
\draw(-1,0)--(-2.5,-1)node[above]{\footnotesize$\tilde{\Sigma}_2^{(1)*}$};
        \coordinate (b) at (1,0);
			\fill (b) circle (1pt) node[below] {$\tilde{k}_{1}$};
	\coordinate (f) at (-1,0);
			\fill (f) circle (1pt) node[below] {$\tilde{k}_{2}$};
\draw[-latex](1,0)--(1.725,0.5);
\draw[-latex](1,0)--(1.725,-0.5);
\draw[color=blue][-latex](-1.5,-1.5)--(-0.4,-0.4);		
\draw[color=blue][-latex](0,0)--(0.6,0.6);		
\draw[color=blue][-latex](0,0)--(0.6,-0.6);			
\draw[color=blue][-latex](-1.5,1.5)--(-0.4,0.4);

        \node at (0.8,0.3) {\footnotesize $\tilde{\Omega}$};
        \node at (0.6,-0.3) {\footnotesize $\tilde{\Omega}^*$};
        \node at (-0.6,-0.3) {\footnotesize $\tilde{\Omega}^*$};
        \node at (-0.8,0.3) {\footnotesize $ \tilde{\Omega}$};

		\coordinate (I) at (0.2,0);
		\fill (I) circle (0pt) node[below] {$0$};

	\end{tikzpicture}
	\caption{ The jump contours of the RH problems for $\tilde{M}^{rhp1}(\tilde{k})$ (black) and $M^P(z;s,\kappa)$ (blue). }
	\label{fig7}
\end{figure}

We could solve the RH problem for $\tilde{M}^{rhp1}(\tilde{k})$ explicitly by using the Painlev\'{e} II parametrix $M^P(z;s,\kappa)$
introduced in the Appendix \ref{appendix}. More precisely, define
\begin{equation}\label{3.95}
    \tilde{M}^{rhp1}(\tilde{k})=e^{\frac{\pi i}{4}\sigma_3}M^P(\tilde{k};\tilde{s},-r(1))e^{-\frac{\pi i}{4}\sigma_3}\tilde{H}(\tilde{k}),
\end{equation}
where
\begin{equation}\nonumber
    \begin{aligned}\tilde{H}(\tilde{k})&=\begin{cases}\begin{pmatrix}1&0\\r(1)e^{2i\left(\tilde{s}\tilde{k}+\frac{4}{3}\tilde{k}^3\right)}&1\end{pmatrix},
    &\tilde{k}\in\tilde{\Omega},\\\begin{pmatrix}1&-r(1)e^{-2i\left(\tilde{s}\tilde{k}+\frac{4}{3}\tilde{k}^3\right)}\\0&1\end{pmatrix},
    &\tilde{k}\in\tilde{\Omega}^*,\\I,&\text{elsewhere,}\end{cases}\end{aligned}
\end{equation}
and where the region $\tilde{\Omega}$ is illustrated in Fig.\ref{fig7}.

In view of RH problem \ref{PII}, it is readily seen that $\tilde{M}^{rhp1}(\tilde{k})$ in \eqref{3.95} indeed solves RH problem \ref{RHP9}.
Moreover,
as a corollary of Proposition \ref{prop3.12}, we have the following result.
\begin{corollary}
As $\tilde{k}\to\infty$,
\begin{equation}\label{3.91}
    M^{rhp1}(\tilde{k})=I+\frac{M_{1}^{rhp1}}{\tilde{k}}+\mathcal{O}(\tilde{k}^{-2}),
\end{equation}
where
\begin{equation}\label{3.94}
    M_1^{rhp1}=\frac{i}{2}\left(\begin{array}{cc}-\int_{\tilde{s}}^{+\infty}v^2(\zeta)d\zeta&v(\tilde{s})
    \\-v(\tilde{s})&\int_{\tilde{s}}^{+\infty}v^2(\zeta)d\zeta\end{array}\right)+\mathcal{O}\left(t^{-1/3+4\delta_1}\right),
\end{equation}
where $v(\tilde{s})$ is the unique solution of Painlev\'{e} II equation \eqref{v}, fixed
by the boundary condition
\begin{equation}\nonumber
    v(\tilde{s})\to -r(1)\mathrm{Ai}(\tilde{s}),\quad\tilde{s}\to+\infty.
\end{equation}
\end{corollary}

Finally, the RH problem associated with the function $M^{rhp2}(z)$ can be resolved through a procedure that closely parallels the methods employed previously. In fact, when considering the situation where the complex variable $z$ is an element of the set $U^{(2)}$ and the parameter $t$ attains a sufficiently large value, we have
\begin{equation}\nonumber
    \theta(z)=-\tilde{s}\check{k}-\frac{4}{3}\check{k}^3+\mathcal{O}\left(\check{k}^4t^{-1/3}\right),
\end{equation}
where $\tilde{s}$ is defined in \eqref{3.59} and
\begin{equation}\nonumber
    \check{k}=\left(\frac{9t}{2}\right)^{1/3}(z+1)
\end{equation}
is the scaled spectral parameter in this case. Follow the steps we just performed above, we
could approximate $M^{rhp2}(\check{k})$
\begin{equation}\nonumber
    M^{rhp2}(\check{k})=I+\frac{M_1^{rhp2}}{\check{k}}+\mathcal{O}(\check{k}^{-2}),
\end{equation}
where
\begin{equation}\label{3.103}
    M_1^{rhp2}=\frac{i}{2}\left(\begin{array}{cc}-\int_{\tilde{s}}^{+\infty}v^2(\zeta)d\zeta&-v(\tilde{s})\\
    v(\tilde{s})&\int_{\tilde{s}}^{+\infty}v^2(\zeta)d\zeta\end{array}\right)+\mathcal{O}\left(t^{-1/3+4\delta_1}\right),
\end{equation}
with the same $v$ in \eqref{3.94}. Here, we need to use the fact that $\overline{r(\bar{z}^{-1})}=-\overline{r(-\bar{z})}$,
which particularly implies that $r(1)=-r(-1)\in\mathbb{R}$.

\subsubsection{Small-norm RH problem}\label{s:3.3.1}
On the basis of \eqref{3.65} and the boundedness of $M^{out}(z)$ and $M^{rhpj}(z)$, the unknown error function
$E(z)$ can be shown as
\begin{equation}\label{3.106}
    E(z)=\begin{cases}M^{rhp}(M^{out})^{-1},&z\in\mathbb{C}\setminus\left(U^{(1)}\cup U^{(2)}\right),\\
    M^{rhp}M^{rhp1}(z)^{-1}(M^{out})^{-1},&z\in U^{(1)},\\
    M^{rhp}M^{rhp2}(z)^{-1}(M^{out})^{-1},&z\in U^{(2)}.\end{cases}
\end{equation}
It is then readily seen that $E(z)$ satisfies the following RH problem.
\begin{RHP}\label{RHP10}
    		Find a matrix valued function $E(z)$ admits:
    		\begin{enumerate}[($i$)]
    			\item Analyticity:~$E(z)$ is analytical in $z\in\mathbb{C}\setminus\Sigma^{(E)}$, where
    \begin{equation}\nonumber
    \Sigma^{(E)}:=\partial U^{(1)}\cup\partial U^{(2)}\cup\left(\Sigma^{(3)}\setminus(U^{(1)}\cup U^{(2)})\right), see~~~Fig.\ref{fig8};
\end{equation}

    			\item Jump condition:
    			\begin{equation}\nonumber
    E_+(z)=E_-(z)V^{(E)}(z),
\end{equation}
	where
\begin{equation}\nonumber
    V^{(E)}(z)=\begin{cases}
    M^{out}V^{(3)}(z)(M^{out})^{-1},&z\in\Sigma^{(E)}\setminus\left(U^{(1)}\cup U^{(2)}\right),\\
    M^{out}M^{rhp1}(z)(M^{out})^{-1},&z\in\partial U^{(1)},\\
    M^{out}M^{rhp2}(z)(M^{out})^{-1},&z\in\partial U^{(2)}.\end{cases}
\end{equation}
    			\item Asymptotic behavior:
                \begin{equation}\nonumber
    E(z)=I+\mathcal{O}(z^{-1}).
\end{equation}

    		\end{enumerate}
    	\end{RHP}
   \begin{figure}[htbp]
	\centering
	\tikzset{every picture/.style={line width=0.75pt}} 
	\begin{tikzpicture}[scale=1.7]

\draw(-0.64,0.15)--(0,0.5);
\draw[->](-0.64,0.15)--(-0.4,0.28);
\draw(-1.35,0.16)--(-2.5,0.6);
\draw[-<](-1.35,-0.16)--(-1.75,-0.31);
\draw(-0.64,-0.15)--(0,-0.5);
\draw[-<](-1.35,0.16)--(-1.75,0.31);
\draw(-1.35,-0.16)--(-2.5,-0.6);
\draw[->](-0.64,-0.15)--(-0.4,-0.28);
\draw[dashed](-3.5,0)--(3.5,0)node[right]{Re$z$};
\draw [-latex](3.5,0)--(3.6,0);
\draw(0.64,0.15)--(0,0.5);
\draw[-<](0.64,0.15)--(0.4,0.28);
\draw(1.35,0.16)--(2.5,0.6);
\draw[->](1.35,-0.16)--(1.75,-0.31);
\draw(0.64,-0.15)--(0,-0.5);
\draw[->](1.35,0.16)--(1.75,0.31);
\draw(1.35,-0.16)--(2.5,-0.6);
\draw[-<](0.64,-0.15)--(0.4,-0.28);
\draw(0,0)--(0,0.5);
\draw(0,0)--(0,-0.5);
\draw[->](0,0)--(0,0.2);
\draw[->](0,0)--(0,-0.2);
\coordinate (I) at (0,0);
\fill (I) circle (1pt) node[below right] {$0$};

\coordinate (b) at (-1.2,0);
\fill (b) circle (0.8pt) node[below] {$z_4$};
\coordinate (c) at (-0.8,0);
\fill (c) circle (0.8pt) node[below] {$z_3$};
\coordinate (p) at (0.8,0);
\fill (p) circle (0.8pt) node[below] {$z_2$};
\coordinate (l) at (1.2,0);
\fill (l) circle (0.8pt) node[below] {$z_1$};

\coordinate (f) at (1,0);
\draw[thick,blue](1,0) circle (0.4);
\node at (1,0.6) {\footnotesize $ \partial U^{(1)}$};
\node at (-1,0.6) {\footnotesize $ \partial U^{(2)}$};
\draw[thick,blue](-1,0) circle (0.4);
\end{tikzpicture}
	\caption{The jump contour $\Sigma^{(E)}$ of the RH problem for $E$, where the two blue circles are $\partial U^{(1)}$ and $\partial U^{(2)}$. }
	\label{fig8}
\end{figure}

Next we will show the existence of the solution $E(z)$. Firstly, we estimate the error
of $V^{(E)}(z)-I$. A simple calculation shows that
    \begin{equation}\nonumber
    \parallel V^{(E)}(z)-I\parallel_p=\begin{cases}\mathcal{O}(\exp\left\{-ct^{3\delta_1}\right\}),&z\in\Sigma^{(E)}\setminus\left(U^{(1)}\cup U^{(2)}\right),\\\mathcal{O}(t^{-\kappa_p}),&z\in\partial U^{(1)}\cup\partial U^{(2)},\end{cases}
\end{equation}
for some positive $c$ and $\kappa_{p}=\frac{p-1}{p}\delta_{1}+\frac{1}{3p}$. Subsequently, by virtue of the small norm RH
problem theory \cite{boo-49}, it is evident that for significantly large positive $t$, there exists a unique solution to RH problem \ref{RHP10}. Moreover, according to \cite{boo-50}, we have
\begin{equation}\label{3.112}
    E(z)=I+\frac{1}{2\pi i}\int_{\Sigma^{(E)}}\frac{\mu_{E}(\zeta)(V^{(E)}(\zeta)-I)}{\zeta-z}d\zeta,
\end{equation}
where $\mu_{E}\in I+L^2(\Sigma^{(E)})$ is the unique solution of the Fredholm-type equation
\begin{equation}\label{3.113}
    \mu_{E}=I+C_E\mu_{E}.
\end{equation}
Here, $C_E{:}L^2(\Sigma^{(E)})\to L^2(\Sigma^{(E)})$ is an integral operator defined by $C_E(f)(z)=C_-\left(f(V^{(E)}(z)-I)\right)$
with $C_{-}$ being the Cauchy projection operator on $\Sigma^{(E)}$. Thus, we have
\begin{equation}\nonumber
    \parallel C_E\parallel_{L^2(\Sigma^E)}\leqslant\parallel C_-\parallel_{L^2(\Sigma^E)}\parallel V^{(E)}(z)-I\parallel_{L^\infty(\Sigma^E)}\lesssim t^{-\delta_1},
\end{equation}
which guarantees the existence of the resolvent operator $(1-C_E)^{-1}$. So, $\mu_{E}$ exists uniquely with
\begin{equation}\nonumber
    \mu_{E}=I+(1-C_E)^{-1}(C_EI).
\end{equation}
On the other hand, \eqref{3.113} can be rewritten as
\begin{equation}\nonumber
    \mu_{E}=I+\sum_{j=1}^4C_E^jI+(1-C_E)^{-1}(C_E^5I),
\end{equation}
where for $j=1,\cdots,4,$ we have the estimates
\begin{equation}\nonumber
    \|C_E^jI\|_{L^2(\Sigma^E)}\lesssim\|C_E^{j-1}I\|_{L^2(\Sigma^E)}\parallel C_E\parallel_{L^2(\Sigma^E)}\lesssim t^{-(j-1)\delta_1}t^{-(1/6+\delta_1/2)}  \lesssim t^{-(1/6+j\delta_1-\delta_1/2)},
\end{equation}
\begin{equation}\nonumber
    \|\mu_{E}-I-\sum_{j=1}^4C_E^jI\parallel_{L^2(\Sigma^E)}\lesssim\frac{\|C_E^5I\parallel_{L^2(\Sigma^E)}}{1-\|C_E\|}\lesssim t^{-(1/6+9\delta_1/2)}.
\end{equation}
In order to retrieve the potential using the reconstruction formula \eqref{2.35}, a thorough investigation of the properties of the function $E(z)$ at the points of $i$ is of utmost importance. As a subsequent step, we will conduct an expansion of the function $E(z)$ at $z=i$. By \eqref{3.112}, it is readily seen that
\begin{equation}\label{3.119}
    E(z)=E(i)+E_1(z-i)+O((z-i)^2),\quad z\to i,
\end{equation}
where
\begin{equation}\nonumber
    E(i)=I+\frac{1}{2\pi i}\int_{\Sigma^{(E)}}\frac{\mu_{E}(\zeta)(V^{(E)}(\zeta)-I)}{\zeta-i}d\zeta,
\end{equation}

\begin{equation}\nonumber
    E_1=\frac{1}{2\pi i}\int_{\Sigma^{(E)}}\frac{\mu_{E}(\zeta)(V^{(E)}(\zeta)-I)}{(\zeta-i)^2}d\zeta.
\end{equation}

\begin{proposition}\label{prop3.14}
As $t\to\infty$, $E_1$ and $E(i)$ can be estimated as follows:
\begin{equation}\label{3.122}
    E(i)=I-t^{-1/3}\left(\frac{2}{9}\right)^{1/3}\left(M^{out}(1)\frac{\tilde{M}_1^{rhp1}(\tilde{s})}{1-i}M^{out}(1)^{-1}
    -M^{out}(-1)\frac{\check{M}_1^{rhp2}(\tilde{s})}{1+i}M^{out}(-1)^{-1}\right)
    +\mathcal{O}\left(t^{-2/3+4\delta_1}\right),
\end{equation}
\begin{equation}\label{3.123}
    E_1=-t^{-1/3}\frac{i}{2}\left(\frac{2}{9}\right)^{1/3}\left(M^{out}(1)\tilde{M}_1^{rhp1}(\tilde{s})M^{out}(1)^{-1}
    -M^{out}(-1)\check{M}_1^{rhp2}(\tilde{s})M^{out}(-1)^{-1}\right)+\mathcal{O}\left(t^{-2/3+4\delta_1}\right),
\end{equation}
where
\begin{equation}\nonumber
   \tilde{M}_1^{rhp1}(\tilde{s})=\frac{i}{2}\left(\begin{array}{cc}-\int_{\tilde{s}}^{+\infty}v^2(\zeta)d\zeta&v(\tilde{s})
    \\-v(\tilde{s})&\int_{\tilde{s}}^{+\infty}v^2(\zeta)d\zeta\end{array}\right),
\end{equation}
\begin{equation}\nonumber
   \check{M}_1^{rhp2}(\tilde{s})=\frac{i}{2}\left(\begin{array}{cc}-\int_{\tilde{s}}^{+\infty}v^2(\zeta)d\zeta&-v(\tilde{s})\\
    v(\tilde{s})&\int_{\tilde{s}}^{+\infty}v^2(\zeta)d\zeta\end{array}\right).
\end{equation}
\end{proposition}
\begin{proof}
Before starting the proof, we introduce the following notation
\begin{equation}
Y=M^{out}(\zeta)(M^{rhp1}(\zeta)-I)M^{out}(\zeta)^{-1};~~~V=M^{out}(\zeta)(M^{rhp2}(\zeta)-I)M^{out}(\zeta)^{-1}.
\end{equation}
\begin{equation}
Y_1=M^{out}(\zeta)M_{1}^{rhp1}M^{out}(\zeta)^{-1};~~~V_1=M^{out}(\zeta)M_{1}^{rhp2}M^{out}(\zeta)^{-1}.
\end{equation}
 \begin{align*}
 E(i)=&I+\frac{1}{2\pi i}\oint_{\partial U^{(1)}}\frac{M^{out}(\zeta)(M^{rhp1}(\zeta)-I)M^{out}(\zeta)^{-1}}{\zeta-i}d\zeta\\
 &+\frac{1}{2\pi i}\oint_{\partial U^{(2)}}\frac{M^{out}(\zeta)(M^{rhp2}(\zeta)-I)M^{out}(\zeta)^{-1}}{\zeta-i}d\zeta\\
 &+\frac{1}{2\pi i}\oint_{\partial U^{(1)}}\frac{C_-(C_-(Y)(Y))(\zeta)(Y)}{\zeta-i}d\zeta\\
 &+\frac{1}{2\pi i}\oint_{\partial U^{(2)}}\frac{C_-(C_-(V)(V))(\zeta)(V)}{\zeta-i}d\zeta+\mathcal{O}(t^{-1/3-5\delta_1})\\
=&I+t^{-1/3}\left(\frac{2}{9}\right)^{1/3}\frac{1}{2\pi i}\oint_{\partial U^{(1)}}\frac{M^{out}(\zeta)\tilde{M}_{1}^{rhp1}(\tilde{s})M^{out}(\zeta)^{-1}}{(\zeta-i)(\zeta-1)}d\zeta\\
&+t^{-1/3}\left(\frac{2}{9}\right)^{1/3}\frac{1}{2\pi i}\oint_{\partial U^{(2)}}\frac{M^{out}(\zeta)\check{M}_{1}^{rhp2}(\tilde{s})M^{out}(\zeta)^{-1}}{(\zeta-i)(\zeta+1)}d\zeta\\
&+t^{-2/3}\left(\frac{2}{9}\right)^{2/3}\frac{1}{2\pi i}\oint_{\partial U^{(1)}}C_-\left(\frac{1}{(\cdot)-1}\right)(\zeta)\frac{\left(Y_1\right)^2}{(\zeta-i)(\zeta-1)}d\zeta\\
&+t^{-2/3}\left(\frac{2}{9}\right)^{2/3}\frac{1}{2\pi i}\oint_{\partial U^{(2)}}C_-\left(\frac{1}{(\cdot)+1}\right)(\zeta)\frac{\left(V_1\right)^2}{(\zeta-i)(\zeta+1)}d\zeta\\
&+\frac{2t^{-1}}{9}\frac{1}{2\pi i}\oint_{\partial U^{(1)}}C_-\left(\frac{C_-\left(\frac{1}{(\cdot)+1}\right)}{(\cdot)+1}\right)(\zeta)\frac{\left(Y_1\right)^3}{(\zeta-i)(\zeta-1)}d\zeta\\
&+\frac{2t^{-1}}{9}\frac{1}{2\pi i}\oint_{\partial U^{(2)}}C_-\left(\frac{C_-(\frac{1}{(\cdot)+1})}{(\cdot)+1}\right)(\zeta)\frac{\left(V_1\right)^3}{(\zeta-i)(\zeta+1)}d\zeta+\mathcal{O}(t^{-2/3+4\delta_1})+\mathcal{O}(t^{-1/3-5\delta_1}).
 \end{align*}
 Then, similar to \cite{boo-35}, as $C_{-}(\frac{1}{(\cdot)\pm1})(\zeta)=0$, an appeal to the residue theorem gives us \eqref{3.122}. In a similar way, we have the estimates \eqref{3.123}.
\end{proof}
We can define that
\begin{equation}\nonumber
E(i)=I+t^{-1/3}H^{(0)}+\mathcal{O}\left(t^{-2/3+4\delta_1}\right),
\end{equation}
\begin{equation}\nonumber
E_1=t^{-1/3}H^{(1)}+\mathcal{O}\left(t^{-2/3+4\delta_1}\right),
\end{equation}
where
\begin{equation}\label{H0}
    H^{(0)}=-\left(\frac{2}{9}\right)^{1/3}\left(M^{out}(1)\frac{\tilde{M}_{1}^{rhp1}(\tilde{s})}{1-i}M^{out}(1)^{-1}
    -M^{out}(-1)\frac{\check{M}_{1}^{rhp2}(\tilde{s})}{1+i}M^{out}(-1)^{-1}\right),
\end{equation}
\begin{equation}\label{H1}
    H^{(1)}=-\frac{i}{2}\left(\frac{2}{9}\right)^{1/3}\left(M^{out}(1)\tilde{M}_{1}^{rhp1}(\tilde{s})M^{out}(1)^{-1}
    -M^{out}(-1)\check{M}_{1}^{rhp2}(\tilde{s})M^{out}(-1)^{-1}\right).
\end{equation}
\subsection{Asymptotic analysis on a pure $\bar{\partial}$-problem}\label{s:3.3}
Since we have successfully demonstrated the existence of the solution $M^{rhp}(z)$, we are now in a position to utilize $M^{rhp}(z)$ to transform $M^{(3)}(z)$ into a pure $\bar{\partial}$-problem which will be analyzed in this section. We define the function
\begin{equation}\label{3.125}
    M^{(4)}(z)=M^{(3)}(z)(M^{rhp}(z))^{-1}.
\end{equation}
By the properties of $M^{(3)}(z)$ and $M^{rhp}(z)$, we find $M^{(4)}(z)$ satisfies the following $\bar{\partial}$-problem.
\begin{Dbarproblem}\label{RHP11}
    		Find a matrix valued function $M^{(4)}(z)$ admits:
    		\begin{enumerate}[($i$)]
    			\item Continuity:~$M^{(4)}(z)$ is continuous and has sectionally continuous first partial derivatives in $\mathbb{C}$;
          \item Asymptotic behavior:
                \begin{equation}\nonumber
   M^{(4)}(z)=I+\mathcal{O}(z^{-1}).
\end{equation}
    			
    			 \item $\bar{\partial}$-Derivative:
    \begin{equation}\nonumber
    \bar{\partial}M^{(4)}(z)=M^{(4)}(z)W^{(3)}(z),z\in\mathbb{C},
\end{equation}
where
\begin{equation}\nonumber
    W^{(3)}(z)=M^{rhp}\bar{\partial}R^{(2)}(z)(M^{rhp})^{-1},
\end{equation}
and $\bar{\partial}R^{(2)}(z)$ has been given in \eqref{3.56}.
    		\end{enumerate}
    	\end{Dbarproblem}
$\bar{\partial}$-problem \ref{RHP11} is equivalent to the integral equation
    \begin{equation}\label{3.129}
    M^{(4)}(z)=I+\frac{1}{\pi}\iint_\mathbb{C}\frac{M^{(4)}(\zeta)W^{(3)}(\zeta)}{\zeta-z}d\mu(\zeta),
\end{equation}
which can be written as an operator equation
\begin{equation}\label{3.131}
(I-S)M^{(5)}(z)=I
\end{equation}
where
\begin{equation}\label{3.130}
    Sf(z)=\frac{1}{\pi}\iint_\mathbb{C}\frac{f(\zeta)W^{(3)}(\zeta)}{\zeta-z}d\mu(\zeta),
\end{equation}
\begin{proposition}\label{prop3.17}
Let $S$ be the operator defined in \eqref{3.130}, we have
\begin{equation}\label{3.132}
    \parallel S\parallel_{L^\infty\to L^\infty}\lesssim t^{-1/3},\quad t\to+\infty,
\end{equation}
which implies that $\|(I-S)^{-1}\|$ is uniformly bounded for large positive $t$.
\end{proposition}
\begin{proof}
Similar to \cite{boo-35},
for any $f\in L^{\infty}$,
\begin{equation}\nonumber
    \parallel Sf\parallel_\infty\leqslant\parallel f\parallel_\infty\frac{1}{\pi}\iint_\mathbb{C}\frac{|W^{(3)}(\zeta)|}{|z-\zeta|}d\mu(\zeta),
\end{equation}
\begin{equation}\label{3.134}
    \frac{1}{\pi}\iint_{\Omega_1}\frac{|W^{(3)}(\zeta)|}{|z-\zeta|}d\mu(\zeta)\lesssim\frac{1}{\pi}\iint_{\Omega_1}
    \frac{|\bar{\partial}R_1(\zeta)e^{2it\theta}|}{|z-\zeta|}d\mu(\zeta).
\end{equation}
We next divide $\Omega_1$ into two regions: $\{z\in\Omega_1:\operatorname{Re}z(1-|z|^{-2})\leqslant2\}$ and $\{z\in\Omega_1:\operatorname{Re}z(1-|z|^{-2})\geqslant2\}$. These two regions belong to
\begin{equation}\label{3.135}
    \Omega_A=\{z\in\Omega_1:\mathrm{Re}z\leqslant3\},\quad\Omega_B=\{z\in\Omega_1:\mathrm{Re}z\geqslant2\},
\end{equation}
respectively. By setting
\begin{equation}\nonumber
    \zeta=u+k_1+vi,\quad z=x+yi,\quad u,v,x,y\in\mathbb{R},
\end{equation}
it is easy to see from \eqref{3.39} and \eqref{3.40} that
\begin{equation}\label{3.137}
    \iint_{\Omega_1}\frac{|\bar{\partial}R_1(\zeta)|e^{2t\operatorname{Im}\theta}}{|z-\zeta|}d\mu(\zeta)\lesssim I_1+I_2+I_3,
\end{equation}
where
\begin{equation}\nonumber
    \begin{aligned}&I_{1}:=\iint_{\Omega_{A}}\frac{e^{-tu^{2}v}}{|z-\zeta|}dudv,\quad I_{2}:=\iint_{\Omega_{B}}\frac{e^{-tv}\sin\left(\frac{\pi}{2\varphi_{0}}\arg(\zeta-k_{1})\mathcal{X}(\arg(\zeta-k_{1}))\right)}{|z-\zeta|}dudv,\\
    &I_{3}:=\iint_{\Omega_{B}}\frac{|u|^{-1/2}e^{-tv}}{|z-\zeta|}dudv.\end{aligned}
\end{equation}
Our task now is to estimate the integrals $I_{i}$, $i=1,2,3,$ respectively. Using H\"older's inequality and the following basic inequalities
\begin{equation}\nonumber
    \||z-\zeta|^{-1}\|_{L_u^q(v,\infty)}\lesssim|v-y|^{-1+1/q},\quad\|e^{-tvu^2}\|_{L_u^p(v,\infty)}\lesssim(tv)^{-1/2p},
\end{equation}
where $1/p+1/q=1,$ it follows that
\begin{equation}\label{3.139}
    \begin{aligned}I_{1}&\leqslant\int_{0}^{(3-k_{1})\tan\varphi_{0}}\int_{v}^{3}\frac{e^{-tv u^{2}}}{|z-\zeta|}dudv\\&\lesssim t^{-1/4}\int_{0}^{(3-k_{1})\tan\varphi_{0}}|v-y|^{-1/2}v^{-1/4}e^{-tv^{3}}dv\lesssim t^{-1/3}.\end{aligned}
\end{equation}
To estimate $I_{2}$, we obtain from the definition of $\mathcal{X}(x)$ in \eqref{3.36} that
\begin{equation}\label{3.140}
    I_2\leqslant\int_2^{+\infty}\int_{u\tan(\varphi_0/3)}^{u\tan\varphi_0}\frac{e^{-tv}}{|z-\zeta|}dv du.
\end{equation}
As
\begin{equation}\nonumber
    \left(\int_{u\tan(\varphi_0/3)}^{u\tan\varphi_0}e^{-ptv}dv\right)^{1/p}=(pt)^{-1/p}
    e^{-tu\tan(\varphi_0/3)}\left(1-e^{-ptu(\tan\varphi_0-\tan(\varphi_0/3))}\right)^{1/p},
\end{equation}
it follows from H\"older's inequality again that
\begin{equation}\nonumber
    I_2\lesssim\int_2^{+\infty}t^{-1/p}e^{-tu\tan(\varphi_0/3)}|u-x|^{-1+1/q}du\lesssim t^{-1}.
\end{equation}
As for $I_3$, using arguments similar to \cite{boo-36}, one has
\begin{equation}\nonumber
I_3\lesssim t^{-1/2}.
\end{equation}
Thus
\begin{equation}\nonumber
    \frac{1}{\pi}\iint_{\Omega_1}\frac{|W^{(3)}(\zeta)|}{|z-\zeta|}d\mu(\zeta)\lesssim t^{-1/3}.
\end{equation}
The integration of other regions can be estimated in a similar way, which finally leads to \eqref{3.132}.
\end{proof}

Proposition \ref{prop3.17} suggests that there is a unique solution to the operator equation \eqref{3.131}. To reconstruct $u(x,t)$ from \eqref{2.35}, we must determine the $z=i$ value of $M^{(4)}(z)$. By \eqref{3.129}, it is readily seen that
\begin{equation}\nonumber
    M^{(4)}(z)=M^{(4)}(i)+M_1^{(4)}(z-i)+O((z-i)^2),\quad z\to i,
\end{equation}
where
\begin{equation}\label{3.143}
    M^{(4)}(i)=I+\frac{1}{\pi}\iint_\mathbb{C}\frac{M^{(4)}(\zeta)W^{(3)}(\zeta)}{\zeta-i}d\mu(\zeta),
\end{equation}
\begin{equation}\label{3.144}
    M_1^{(4)}=\frac{1}{\pi}\iint_\mathbb{C}\frac{M^{(4)}(\zeta)W^{(3)}(\zeta)}{(\zeta-i)^2}d\mu(\zeta).
\end{equation}
\begin{proposition}\label{prop3.18}
Let $(x,t)\in\mathcal{P}_{I}$, we have, as $t\to+\infty$
\begin{equation}\nonumber
    |M^{(4)}(i)-I|,|M_1^{(4)}|\lesssim t^{-5/6}.
\end{equation}
\end{proposition}
\begin{proof}
Similarly to the proof of Proposition \ref{prop3.17}, we take $\Omega_{1}$ as an example. With
the regions $\Omega_{A}$ and $\Omega_{B}$ defined in \eqref{3.135}, it follows that
\begin{equation}\nonumber
    \iint_{\Omega_1}\leqslant\iint_{\Omega_A}+\iint_{\Omega_B}.
\end{equation}
Noting that $|\zeta|^{-1},|\zeta-i|^{-1}\leqslant1$ for $\zeta\in\Omega_1$ and the support of $W^{(3)}(\zeta)$ is bounded away from $\zeta=i$, by Proposition \ref{prop3.17}, it is then sufficient to consider the integral
\begin{equation}\nonumber
    \iint_{\Omega_A}M^{(4)}(\zeta)W^{(3)}(\zeta)d\mu(\zeta)\lesssim\iint_{\Omega_A}|\bar{\partial}R_1(\zeta)e^{2it\theta}|d\mu(\zeta).
\end{equation}
Let $\zeta=u+k_{1}+v i=k_{1}+le^{\varphi i}$ with $u,v,l,\varphi\in\mathbb{R}$, we observe from Lemma \ref{lem3.5} and \eqref{3.38} that
\begin{equation}\nonumber
    \begin{gathered}\iint_{\Omega_{A}}|\bar{\partial}R_{1}(\zeta)e^{2it\theta}|d\mu(\zeta)\leqslant\int_{0}^{(3-k_{1})\tan\varphi_{0}}
    \int_{v}^{3}(u^{2}+v^{2})^{1/4}e^{-tu^{2}v}dudv\\=t^{-5/6}\int_{0}^{\infty}\int_{v}^{\infty}(u^{2}+v^{2})^{1/4}e^{-u^{2}v}dudv.\end{gathered}
\end{equation}
Using the polar coordinates $u=l\cos\varphi$, $v=l\sin\varphi$, it is readily seen that
\begin{equation}\nonumber
    \begin{gathered}\begin{aligned}\int_{0}^{\infty}\int_{v}^{\infty}\left(u^{2}+v^{2}\right)^{1/4}e^{-u^{2}v}dud
    v=\int_{0}^{\pi/4}\int_{0}^{+\infty}l^{3/2}e^{l^{3}\cos^{2}\varphi\sin\varphi}dld\varphi\end{aligned}\\
    =\int_0^{\pi/4}\cos^{-5/3}\varphi\sin^{-5/6}\varphi\int_0^{+\infty}\frac{1}{3}e^{-w}w^{-2/3}dwd\varphi\leqslant\Gamma(5/6)B(1,1/12),\end{gathered}
\end{equation}
where $w=l^3\cos^2\varphi\sin\varphi$, $\Gamma(\cdot)$ and $B(\cdot,\cdot)$ are the Gamma function
and Beta function, respectively. It can be inferred that
\begin{equation}\nonumber
    \iint_{\Omega_A}|\bar{\partial}R_1(\zeta)e^{2i\theta}|d\mu(\zeta)\lesssim t^{-5/6}.
\end{equation}
To estimate the integral over $\zeta\in\Omega_B$, we also use \eqref{3.38} and the facts that $|\zeta-i|^{-1},|\zeta-i|^{-2}\lesssim|\zeta|^{-1}$. Thus,
\begin{equation}\nonumber
    \begin{gathered}
    \left|\frac{1}{\pi}\iint_{\Omega_{B}}\frac{M^{(4)}(\zeta)W^{(3)}(\zeta)}{(\zeta-i)^{2}}d\mu(\zeta)\right|
    \lesssim\left|\frac{1}{\pi}\iint_{\Omega_B}\frac{M^{(4)}(\zeta)W^{(3)}(\zeta)}{\zeta}d\mu(\zeta)\right
    |\lesssim\left|\frac{1}{\pi}\iint_{\Omega_B}\frac{\bar{\partial}R_1(\zeta)e^{2it\theta}}{\zeta}d\mu(\zeta)\right|,\end{gathered}
\end{equation}

\begin{equation}\nonumber
    \begin{gathered}\left|\frac{1}{\pi}\iint_{\Omega_{B}}\frac{M^{(4)}(\zeta)W^{(3)}(\zeta)}{(\zeta-i)}d\mu(\zeta)\right|
    \lesssim\left|\frac{1}{\pi}\iint_{\Omega_B}\frac{M^{(4)}(\zeta)W^{(3)}(\zeta)}{\zeta}d\mu(\zeta)\right
    |\lesssim\left|\frac{1}{\pi}\iint_{\Omega_B}\frac{\bar{\partial}R_1(\zeta)e^{2it\theta}}{\zeta}d\mu(\zeta)\right|.\end{gathered}
\end{equation}
As
\begin{equation}\nonumber
    \begin{aligned}\left|\frac{1}{\pi}\iint_{\Omega_{B}}\frac{\bar{\partial}R_{1}(\zeta)e^{2it\theta}}{\zeta}d\mu(\zeta)\right
    |&\leqslant\int_{2}^{+\infty}\int_{0}^{u}e^{-tv}(u^{2}+v^{2})^{-1/2}dvdu\\&\leqslant t^{-1}\int_{2}^{+\infty}u^{-1}(1-e^{-tu})du\lesssim t^{-1}.\end{aligned}
\end{equation}
By combining the above three estimates, we have obtained the expected results.
\end{proof}

\subsection{Painlev\'{e} Asymptotics in the Transition Region $\mathcal{P}_{I}$}\label{s:3.4}

For transition region $\mathcal{P}_{I}$, inverting the sequence of transformations \eqref{3.7}, \eqref{3.50}, \eqref{3.106}, and \eqref{3.125}, we conclude that, as $t\to+\infty$,
\begin{equation}\nonumber
    M(z)=M^{(4)}(z)E(z)M^{out}(z)R^{(2)}(z)^{-1}T(z)^{-\sigma_3}G(z)^{-1}+\mathcal{O}(e^{-ct}),
\end{equation}
where $c$ is a constant, $E(z)$, $R^{(2)}(z)$, $T(z)$ and $G(z)$ are defined in \eqref{3.106}, \eqref{3.56}, \eqref{2.40}, and \eqref{3.8}, respectively.

As $G(z)=R^{(2)}(z)=I$ on a neighborhood of $i$, and $T(z)=T(i)+\mathcal{O}((z-i)^2),z\to i$. We can obtain that as $z\to i$
\begin{align}
    M(z)=&\left(I+H^{(0)}t^{-1/3}+E_1(z-i)\right)\left(M_\Lambda^{out}(i)+M_{\Lambda,1}^{out}(z-i)\right)T(i)^{-\sigma_3}\nonumber\\
    &+\mathcal{O}\left((z-i)^2\right)+O(t^{-5/6})+\mathcal{O}\left(t^{-2/3+4\delta_1}\right)\nonumber,
\end{align}
\begin{equation}\nonumber
    M(i)=\left(I+H^{(0)}t^{-1/3}\right)\left(M_\Lambda^{out}(i)\right)T(i)^{-\sigma_3}+\mathcal{O}\left(t^{-2/3+4\delta_1}\right),
\end{equation}
where $E(i)$ and $E_1$ are given in \eqref{3.122} and \eqref{3.123}, respectively. Here, the error term $\mathcal{O}(t^{-5/6})$ comes
from the pure $\bar{\partial}$-problem. From the reconstruction formula \eqref{2.35}, we then obtain the Painlev\'{e} asymptotics of the mCH equation in the transition region $\mathcal{P}_{I}$
\begin{equation}\nonumber
u(y,t)=u_{p}(y,t;\tilde{\mathcal{D}}_\Lambda)+f_{11}t^{-1/3}+\mathcal{O}(t^{-2/3+4\delta_1}),\end{equation}
\begin{equation}\nonumber
x(y,t)=y-2\ln\left(T(i)\right)+c_+^{out}(y,t;\tilde{\mathcal{D}}_\Lambda)+f_{12}t^{-1/3}+\mathcal{O}(t^{-2/3+4\delta_1}),\end{equation}
where $T(i)$, $u_{p}(y,t;\tilde{\mathcal{D}}_\Lambda)$ and $c_+^{out}(y,t;\tilde{\mathcal{D}}_\Lambda)$ are show in \eqref{3.5}, \eqref{ur} and \eqref{c+},
\begin{equation}\label{f11}
\begin{aligned}
f_{11}=&-u_{p}(y,t;\tilde{\mathcal{D}}_{\Lambda})\left(\sum_{j,k=1,2}[H^{(0)}]_{jk}\right)+\sum_{j,k=1,2}[H^{(1)}]_{jk}\\
&+\left([H^{(0)}]_{11}+[H^{(0)}]_{21}\right)\frac{[M_{\Lambda,1}^{out}]_{11}+T(i)^{2}e^{c_{+}^{out}(y,t;\tilde{\mathcal{D}}_{\Lambda})}[M_{\Lambda,1}^{out}]_{12}}
{[M_{\Lambda}^{out}]_{11}(i)+[M_{\Lambda}^{out}]_{21}(i)}\\
&+\left([H^{(0)}]_{12}+[H^{(0)}]_{22}\right)\frac{[M_{\Lambda,1}^{out}]_{21}+T(i)^{2}e^{c_{+}^{out}(y,t;\tilde{\mathcal{D}}_{\Lambda})}[M_{\Lambda,1}^{out}]_{22}}
{[M_{\Lambda}^{out}]_{11}(i)+[M_{\Lambda}^{out}]_{21}(i)},\end{aligned}\end{equation}
\begin{equation}\label{f12}
f_{12}=\left([H^{(0)}]_{11}+[H^{(0)}]_{21}-[H^{(0)}]_{12}-[H^{(0)}]_{22}\right)([M_\Lambda^{out}]_{21}(i)-[M_\Lambda^{out}]_{11}(i)),\end{equation}
and $H^{(0)}$, $H^{(1)}$, $M_{\Lambda}^{out}$, $M_{\Lambda,1}^{out}$ are show in \eqref{H0}, \eqref{H1} and \eqref{Mout}.

\section{Asymptotic Analysis in the Transition Region $\mathcal{P}_{II}$}\label{s:4}

We now consider the asymptotics in the region $\mathcal{P}_{II}$ given by
\begin{equation}\nonumber
    \mathcal{P}_{II}:=\left\{(x,t):0\leqslant\left|\frac{x}{t}+1/4\right|t^{2/3}\leqslant C\right\},
\end{equation}
where $C>0$ is a constant which corresponds to Fig.\ref{fig3}(e). At the end, we will show that $\tau=x/t$ is close to $\xi=y/t$ for large positive $t$. In fact, according to \eqref{5.3}, we can obtain $\xi=\tau+\mathcal{O}(t^{-1})$. We only provide a detailed analysis in $0\leqslant(\xi+1/4)t^{2/3}\leqslant C$ in this section, as the discussion for the other half zone is similar.

Recall that the expression of $\theta(z)$
\begin{align}\nonumber
    \theta(z)=-\frac{1}{4}(z-z^{-1})\left[\xi-\frac{8}{(z+z^{-1})^2}\right], \xi=\frac{y}{t}.
\end{align}
According to the obtained results in \cite{boo-35}, the eight stationary points are expressed as
\begin{equation}\label{4.15}
    z_1=-z_8=2\sqrt{s_+}+\sqrt{4s_++1},\quad z_4=-z_5=-2\sqrt{s_+}+\sqrt{4s_++1},
\end{equation}
\begin{equation}\label{4.16}
    z_2=-z_7=2\sqrt{s_-}+\sqrt{4s_-+1},\quad z_3=-z_6=-2\sqrt{s_-}+\sqrt{4s_-+1},
\end{equation}
where
\begin{equation}\label{4.17}
    s_{+}=\frac{1}{4\xi}\left(-\xi-1+\sqrt{1+4\xi}\right),\quad s_{-}=\frac{1}{4\xi}\left(-\xi-1-\sqrt{1+4\xi}\right).
\end{equation}
In the transition region $\mathcal{P}_{II}$, there are eight phase points where $z_{1,2}\to2+\sqrt3,\quad z_{3,4}\to2-\sqrt3,\quad z_{5,6}\to-2+\sqrt3,\quad z_{7,8}\to-2-\sqrt3$ at least the speed of $t^{-1/3}$ as $t\to+\infty$. Since we have already presented a detailed discussion of the case in $\mathcal{P}_{I}$, we will emphasize the differences between the two cases here. For similar parts, we will
present the results without including their proofs.

\subsection{Transformation to a hybrid $\bar{\partial}$-RH problem}\label{s:4.1}

By \eqref{2.40}, we get the function
\begin{equation}\label{4.5}
    T(z,\xi):=\prod_{j\in\Delta}\left(\frac{z-\zeta_n}{\bar{\zeta}_n^{-1}z-1}\right)\exp\left\{-\frac{1}{2\pi i}\int_{\mathbb{R}}\frac{\log\left(1+|r(s)|^2\right)}{s-z}ds\right\}.
\end{equation}
Next, we will analysis the Region $0\leqslant (\xi+1/4)t^{2/3}\leqslant C$.

It follows from straightforward calculations that
\begin{equation}\nonumber
    z_1=1/z_4=-1/z_5=-z_8,\quad z_2=1/z_3=-1/z_6=-z_7,
\end{equation}
and as $\xi\to\left(-\frac{1}{4}\right)^+$,
\begin{equation}\label{4.19}
    z_{1,2}\to2+\sqrt3,\quad z_{3,4}\to2-\sqrt3,\quad z_{5,6}\to-2+\sqrt3,\quad z_{7,8}\to-2-\sqrt3.
\end{equation}
By \eqref{3.13}, we can represent the jump matrix for $V^{(2)}(z)$ as
\begin{equation}\nonumber
    V^{(2)}(z)=
    \begin{pmatrix}1&0\\ \frac{e^{-2it\theta(z)}r(z)T_-^2(z)}{1+|r(z)|^2}&1\end{pmatrix}\begin{pmatrix}1&\frac{e^{2it\theta(z)}\bar{r}(z)T_+^{-2}(z)}{1+|r(z)|^2}\\0&1\end{pmatrix},z\in \mathbb{R}.
    \end{equation}

We note that
\begin{equation}\label{4.20}
    d(z)=\frac{\bar{r}(z)T_+^{-2}(z)}{1+|r(z)|^2}=\bar{r}(z)\prod_{j\in\Delta}
    \left(\frac{z-\zeta_n}{\bar{\zeta}_n^{-1}z-1}\right)^{-2}\exp\left\{\frac{1}{\pi i}\int_{\mathbb{R}}\frac{\log\left(1+|r(s)|^2\right)}{s-z}ds\right\}.
\end{equation}

\begin{figure}[h]
\centering
		\begin{tikzpicture}
        \draw[yellow!30, fill=green!20] (0,1)--(1,0)--(1.8,0)--(3.1,1)--(4.4,0)--(5,0)--(6.5,0.9)--(6.5,0)
        --(-6.5,0)--(-6.5,0.9)--(-5,0)--(-4.4,0)--(-3.1,1)--(-1.8,0)--(-1,0)--(0,1);
		\draw[blue!30, fill=yellow!20] (0,-1)--(1,0)--(1.8,0)--(3.1,-1)--(4.4,0)--(5,0)--(6.5,-0.9)--(6.5,0)
        --(-6.5,0)--(-6.5,-0.9)--(-5,0)--(-4.4,0)--(-3.1,-1)--(-1.8,0)--(-1,0)--(0,-1);
		\draw[dashed](-6.5,0)--(6.5,0)node[right]{ Re$z$};

		\coordinate (I) at (0,0);
		\fill (I) circle (1pt) node[below] {$0$};
		\coordinate (c) at (-3,0);
		\fill[red] (c) circle (1pt) node[below] {\scriptsize$-1$};
		\coordinate (D) at (3,0);
		\fill[red] (D) circle (1pt) node[below] {\scriptsize$1$};
			\draw [red](-1,0)--(-0,1);
		\draw[red][-latex](-1,0)--(-0.5,0.5)node[above]{\scriptsize$\Sigma_{5}$};
	
		\draw [ blue](-1,0)--(-0,-1);

		\draw[ blue][-latex](-1,0)--(-0.5,-0.5)node[below]{\scriptsize$\Sigma_5^*$};
		\draw[ red][-latex](-3.1,1)--(-2.45,0.5)node[above]{\scriptsize$\Sigma_{6}$};

		\draw[red](-1.8,0)--(-3.1,1);

		\draw[ blue](-1.8,0)--(-3.1,-1);
		\draw[ blue][-latex](-3.1,-1)--(-2.45,-0.5)node[below]{\scriptsize$\Sigma_6^*$};
		\draw[ red](-5,0)--(-6.5,0.9)node[above]{\scriptsize$\Sigma_{8}$};
		\draw[red][-latex](-6.5,0.9)--(-5.75,0.45);

		\draw[ blue](-5,0)--(-6.5,-0.9)node[below]{\scriptsize$\Sigma_8^*$};

		\draw[ blue][-latex](-6.5,-0.9)--(-5.75,-0.45);
		\draw[red](-4.4,0)--(-3.1,1);
		\draw[ red][-latex](-4.4,0)--(-3.75,0.5)node[above]{\scriptsize$\Sigma_{7}$};
		\draw[ blue][-latex](-4.4,0)--(-3.75,-0.5)node[below]{\scriptsize$\Sigma_7^*$};

		\draw[ red](-3.1,0)--(-3.1,1)node[above]{\scriptsize$\Sigma_{6,7}$};
		\draw[ blue](-3.1,0)--(-3.1,-1)node[below]{\scriptsize$\Sigma^*_{6,7}$};
		\draw[ red](1,0)--(0,1);
		\draw[ red][-latex](0,1)--(0.5,0.5)node[above]{\scriptsize$\Sigma_{4}$};

		\draw[ blue][-latex](0,-1)--(0.5,-0.5)node[below]{\scriptsize$\Sigma_4^*$};
		\draw[ red][-latex](1.8,0)--(2.45,0.5)node[above]{\scriptsize$\Sigma_{3}$};

		\draw[ red](1.8,0)--(3.1,1);

		\draw[ blue][-latex](1.8,0)--(2.45,-0.5)node[below]{\scriptsize$\Sigma_3^*$};
		\draw[ red](5,0)--(6.5,0.9)node[above]{\scriptsize$\Sigma_{1}$};
		\draw[ red][-latex](5,0)--(5.75,0.45);
        \draw[blue](-4.4,0)--(-3.1,-1);
        \draw[blue](0,-1)--(1,0);
        \draw[blue](1.8,0)--(3.1,-1);

		\draw[ blue](5,0)--(6.5,-0.9)node[below]{\scriptsize$\Sigma_1^*$};

		\draw[ blue][-latex](5,0)--(5.75,-0.45);
		\draw[ red](4.4,0)--(3.1,1);
		\draw[ red][-latex](3.1,1)--(3.75,0.5)node[above]{\scriptsize$\Sigma_{2}$};

		\draw[ blue](4.4,0)--(3.1,-1);

		\draw[ blue][-latex](3.1,-1)--(3.75,-0.5)node[below]{\scriptsize$\Sigma_2^*$};

		\draw[ red](3.1,0)--(3.1,1)node[above]{\scriptsize$\Sigma_{2,3}$};
		\draw[ blue](3.1,0)--(3.1,-1)node[below]{\scriptsize$\Sigma^*_{2,3}$};
		\draw[ red](0,0)--(0,1)node[above]{\scriptsize$\Sigma_{4,5}$};
		\draw[ blue](0,0)--(0,-1)node[below]{\scriptsize$\Sigma^*_{4,5}$};
		\coordinate (A) at (-5,0);
		\fill (A) circle (1pt) node[below] {$z_8$};
		\coordinate (b) at (-4.4,0);
		\fill (b) circle (1pt) node[below] {$z_7$};
		\coordinate (C) at (-1,0);
		\fill (C) circle (1pt) node[below] {$z_5$};
		\coordinate (d) at (-1.8,0);
		\fill (d) circle (1pt) node[below] {$z_6$};
		\coordinate (E) at (5,0);
		\fill (E) circle (1pt) node[below] {$z_1$};
		\coordinate (R) at (4.4,0);
		\fill (R) circle (1pt) node[below] {$z_2$};
		\coordinate (T) at (1,0);
		\fill (T) circle (1pt) node[below] {$z_4$};
		\coordinate (Y) at (1.8,0);
		\fill (Y) circle (1pt) node[below] {$z_3$};
		\coordinate (q) at (6.2,-0.1);
		\fill (q) circle (0pt) node[above] {\tiny$\Omega_{1}$};
		\coordinate (q1) at (6.2,0.05);
		\fill (q1) circle (0pt) node[below] {\tiny$\Omega_1^*$};

		\coordinate (r) at (3.4,-0.1);
		\fill (r) circle (0pt) node[above] {\tiny$\Omega_{2}$};
		\coordinate (r1) at (3.4,0.05);
		\fill (r1) circle (0pt) node[below] {\tiny$\Omega_2^*$};
		\coordinate (t) at (2.7,-0.1);
		\fill (t) circle (0pt) node[above] {\tiny$\Omega_{3}$};
		\coordinate (t1) at (2.7,0.05);
		\fill (t1) circle (0pt) node[below] {\tiny$\Omega_3^*$};
		\coordinate (k) at (0.3,-0.1);
		\fill (k) circle (0pt) node[above] {\tiny$\Omega_{4}$};
		\coordinate (k1) at (0.3,0.1);
		\fill (k1) circle (0pt) node[below] {\tiny$\Omega_4^*$};
			\coordinate (q8) at (-6.2,-0.1);
		\fill (q8) circle (0pt) node[above] {\tiny$\Omega_{8}$};
		\coordinate (q18) at (-6.2,0.05);
		\fill (q18) circle (0pt) node[below] {\tiny$\Omega_8^*$};
		\coordinate (7r) at (-3.4,-0.1);
		\fill (7r) circle (0pt) node[above] {\tiny$\Omega_{7}$};
		\coordinate (r17) at (-3.4,0.05);
		\fill (r17) circle (0pt) node[below] {\tiny$\Omega_7^*$};
		\coordinate (t6) at (-2.7,-0.1);
		\fill (t6) circle (0pt) node[above] {\tiny$\Omega_{6}$};
		\coordinate (t16) at (-2.7,0.05);
		\fill (t16) circle (0pt) node[below] {\tiny$\Omega_6^*$};

		\coordinate (k5) at (-0.3,-0.1);
		\fill (k5) circle (0pt) node[above] {\tiny$\Omega_{5}$};
		\coordinate (k15) at (-0.3,0.1);
		\fill (k15) circle (0pt) node[below] {\tiny$\Omega_5^*$};
        \draw(-5,0)--(-4.4,0);
		\draw[-latex](-5,0)--(-4.7,0);
        \draw(-1.8,0)--(-1,0);
		\draw[-latex](-2.2,0)--(-1.3,0);
        \draw(1,0)--(1.8,0);
		\draw[-latex](1,0)--(1.5,0);
        \draw(4.4,0)--(5,0);
		\draw[-latex](4,0)--(4.8,0);
        \draw[ blue][-latex](-3.1,-1)--(-3.1,-0.5);
        \draw[ red][-latex](-3.1,1)--(-3.1,0.5);
        \draw[ blue][-latex](0,-1)--(0,-0.5);
        \draw[ red][-latex](0,1)--(0,0.5);
        \draw[ blue][-latex](3.1,-1)--(3.1,-0.5);
        \draw[ red][-latex](3.1,1)--(3.1,0.5);;
		\end{tikzpicture}
	\caption{Open the jump contour $\mathbb{R}\setminus([z_{8},z_{7}]\cup[z_{6},z_{5}]\cup[z_{4},z_{3}]\cup[z_{2},z_{1}])$ along red rays and blue rays. In the green regions, $\mathrm{Re}\left(2it\theta(z)\right)<0$,
while in the yellow regions, $\mathrm{Re}\left(2it\theta(z)\right)>0$.}
	\label{fig9}
\end{figure}
Select a sufficiently small angle $\varphi_0$ satisfying $0<\varphi_0<\pi/4$. Ensure that all the sectors formed by
this angle are situated entirely within their corresponding decaying regions. These decaying regions
are in accordance with the signature table of $\mathrm{Re}\left(2it\theta(z)\right)$.

$\Sigma^*_{j}$ denote the conjugate contours of $\Sigma_{j}$
respectively, as illustrated in Fig. \ref{fig9}. For $j=1,\cdots,8$, we define
\begin{align}
\Omega_1:&=\{z\in\mathbb{C}:0\leqslant(\arg z-z_1)\leqslant\varphi_0\},\nonumber\\
\Omega_2:&=\{z\in\mathbb{C}:\pi-\varphi_0\leqslant(\arg z-z_2)\leqslant\pi,|\operatorname{Re}(z-z_2)|\leqslant(z_2-z_3)/2\},\nonumber\\
\Omega_3:&=\{z\in\mathbb{C}:0\leqslant(\arg z-z_3)\leqslant\varphi_0,|\operatorname{Re}(z-z_3)|\leqslant(z_3-z_4)/2\},\nonumber\\
\Omega_4:&=\{z\in\mathbb{C}::\pi-\varphi_0\leqslant(\arg z-z_4)\leqslant\pi,|\operatorname{Re}(z-z_4)|\leqslant(z_4-z_5)/2\},\nonumber\\
\Omega_5:&=\{z\in\mathbb{C}::0\leqslant(\arg z-z_5)\leqslant\varphi_0,|\operatorname{Re}(z-z_5)|\leqslant(z_5-z_6)/2\},\nonumber\\
\Omega_6:&=\{z\in\mathbb{C}::\pi-\varphi_0\leqslant(\arg z-z_6)\leqslant\pi,|\operatorname{Re}(z-z_6)|\leqslant(z_6-z_7)/2\},\nonumber\\
\Omega_7:&=\{z\in\mathbb{C}::0\leqslant(\arg z-z_7)\leqslant\varphi_0,|\operatorname{Re}(z-z_7)|\leqslant(z_7-z_8)/2\},\nonumber\\
\Omega_8:&=\{z\in\mathbb{C}:\pi-\varphi_0\leqslant(\arg z-z_8)\leqslant\pi\},\nonumber
\end{align}
where $z_{j}$ are eight phase points given in \eqref{4.15} and \eqref{4.16}. To determine the decaying properties of the oscillating factors $e^{\pm2it\theta(z)}$ , we estimate $\mathrm{Re}(2it\theta(z))$ on $\Omega$.

As the function $d(z)$ in \eqref{4.20} is not an analytic function, the idea now is to introduce the functions $R_j(z)$, $j=1,\cdots,8,$ with boundary conditions:
\begin{lemma}\label{lem4.4}
Let initial data $u(x)\in H^{4,2}(\mathbb{R})$. Then it is possible to define functions $R_{j}:\Omega_j\cup\partial\Omega_j\to\mathbb{C},j=1,\cdots,8$, continuous on $\Omega_j\cup\partial\Omega_j$, with continuous first partials on $\Omega_{j}$, and boundary values
\begin{equation}\label{4.29}
    R_j(z)=\begin{cases}-d(z),&z\in\mathbb{R},\\-d(z_j),&z\in\Sigma_j.\end{cases}
\end{equation}
One can give an explicit construction of each $R_j(z)$. Indeed, let $\mathcal{X}\in C_0^\infty(\mathbb{R})$ be such that
\begin{equation}\label{4.30}
    \mathcal{X}(x):=\begin{cases}0,&x\leqslant\frac{\varphi_0}{3},\\1,&x\geqslant\frac{2\varphi_0}{3}.\end{cases}
\end{equation}
We have for $j=1,\cdots,8,$
\begin{align}
 |R_j(z)|&\lesssim\sin^2\left(\frac{\pi}{2\varphi_0}\arg\left(z-z_j\right)\right)+\left(1+\mathrm{Re}(z)^2\right)^{-1/2},\label{4.31}\\
 |\bar{\partial}R_j(z)|&\lesssim|\operatorname{Re}z-z_j|^{-1/2}+\sin\left(\frac{\pi}{2\varphi_0}\arg\left(z-z_j\right)\mathcal{X}(\arg\left(z-z_j\right))\right),\label{4.33}\\
  |\bar{\partial}R_j(z)|&\lesssim|\operatorname{Re}z-z_j|^{1/2},\label{4.32}\\
  |\bar{\partial}R_j(z)|&\lesssim1.\label{4.34}
 \end{align}
\end{lemma}
Furthermore, using the extensions of Lemma \ref{lem4.4}, we define a matrix function
\begin{equation}\label{4.35}
    R^{(2)}(z):=R^{(2)}\left(z;\xi\right)=\begin{cases}
    \begin{pmatrix}1&R_j(z)e^{2it\theta(z)}\\0&1\end{pmatrix},&z\in\Omega_j,j=1,\cdots,8,\\
    \begin{pmatrix}10&\\-R_j^*(z)e^{-2it\theta(z)}&1\end{pmatrix},&z\in\Omega_j^*,j=1,\cdots,8,\\I,&\text{elsewhere,}\end{cases}
\end{equation}
and a contour
\begin{equation}\label{4.36}
    \Sigma^{(3)}:=\bigcup_{j=1}^8(\Sigma_j\cup\Sigma_j^*)\cup\bigcup_{j=1,3,5,7}(z_{j+1},z_j)\cup\bigcup_{j=2,4,6}\left(\Sigma_{j,j+1}\cup\Sigma_{j,j+1}^*\right).
\end{equation}
An illustration of this can be found in Fig.\ref{fig10}. Then, the new matrix function defined by
\begin{equation}\label{4.37}
    M^{(3)}(z)=M^{(2)}(z)R^{(2)}(z)
\end{equation}
satisfies the following mixed $\bar{\partial}$-RH problem.
\begin{dbar-RHP}\label{RHP4.5}
    		Find a matrix valued function $M^{(3)}(z)\triangleq M^{(3)}(z;y,t)$ admits:
    		\begin{enumerate}[($i$)]
    			\item Continuity:~$M^{(3)}(z)$ is continuous in $\mathbb{C}\setminus\left(\Sigma^{(3)}\cup\left\{\zeta_{n},\bar{\zeta}_{n}\right\}_{n\in\Lambda}\right)$;

    			\item Jump condition:
    			\begin{equation}\nonumber
    M_+^{(3)}(z)=M_-^{(3)}(z)V^{(3)}(z), z\in\Sigma^{(3)},
\end{equation}
    where
    		\begin{equation}\label{4.39}
    \begin{aligned}V^{(3)}(z)&=\begin{cases}\begin{pmatrix}1&0\\e^{-2it\theta(z)}\bar{d}(z)&1\end{pmatrix}
    \begin{pmatrix}1&e^{2it\theta(z)}d(z)\\0&1\end{pmatrix},&z\in\bigcup_{j=1,3,5,7}(z_{j+1},z_j),\\
    R^{(2)}(z)^{-1},&z\in\Sigma_j,j=1,\cdots,8,\\R^{(2)}(z),&z\in\Sigma_j^*,j=1,\cdots,8,\\
    \begin{pmatrix}1&(R_{j}(z)-R_{j+1}(z))e^{2it\theta(z)}\\0&1\end{pmatrix},&z\in\Sigma_{j,j+1},j=2,4,6,\\
    \begin{pmatrix}1&0\\-(R_{j+1}^*(z)-R_{j}^*(z))e^{-2it\theta(z)}&1\end{pmatrix},&z\in\Sigma_{j,j+1}^*,j=2,4,6.\end{cases}\end{aligned}
\end{equation}
    			\item Asymptotic behavior:
                 \begin{equation}\nonumber
    M^{(3)}(z)=I+\mathcal{O}(z^{-1}),\quad z\to\infty;
\end{equation}
            \item For $z\in\mathbb{C}$, we have the $\bar{\partial}$-derivative relation
            \begin{equation}\nonumber
    \bar{\partial}M^{(3)}(z)=M^{(3)}(z)\bar{\partial}R^{(2)}(z),
\end{equation}
where
\begin{equation}\label{4.43}
    \bar{\partial}R^{(2)}(z)=\begin{cases}\begin{pmatrix}0&\bar{\partial}R_j(z)e^{2it\theta(z)}\\
    0&0\end{pmatrix},&z\in\Omega_j,j=1,\cdots,8,\\
    \begin{pmatrix}0&0\\-\bar{\partial}R_j^*(z)e^{-2it\theta(z)}&0\end{pmatrix},&z\in\Omega_j^*,j=1,\cdots,8,\\0,&\text{elsewhere.}\end{cases}
\end{equation}
\item Residue conditions: If $j\notin\Lambda$ for $j=1,2,\ldots,N,$, then $M^{(3)}(z)$ is continuous $\mathbb{C}\setminus\Sigma^{(3)}$. If there exist $\zeta_n$ and $\bar{\zeta}_n$ for $n\in\Lambda$, then $M^{(3)}(z)$ admits the residue conditions:
    \begin{equation}\operatorname*{\mathrm{Res}}_{z=\zeta_n}M^{(3)}(z)=\lim_{z\to\zeta_n}M^{(3)}(z)\begin{pmatrix}0&0\\C_ne^{-2it\theta_n}T^2(\zeta_n)&0\end{pmatrix},\end{equation}
    \begin{equation}\operatorname*{\mathrm{Res}}_{z=\bar{\zeta}_n}M^{(3)}(z)=\lim_{z\to\bar{\zeta}_n}M^{(3)}(z)
    \begin{pmatrix}0&-\bar{C}_nT^{-2}(\bar{\zeta}_n)e^{2it\bar{\theta}_n}\\0&0\end{pmatrix}.\end{equation}
    		\end{enumerate}
    	\end{dbar-RHP}
    \begin{figure}[h]
\centering
		\begin{tikzpicture}
		\draw[dashed](-6.5,0)--(6.5,0)node[right]{ Re$z$};

		\coordinate (I) at (0,0);
		\fill (I) circle (1pt) node[below] {$0$};
		\coordinate (c) at (-3,0);
		\fill[red] (c) circle (1pt) node[below] {\scriptsize$-1$};
		\coordinate (D) at (3,0);
		\fill[red] (D) circle (1pt) node[below] {\scriptsize$1$};
			\draw [red](-1,0)--(-0,1);
		\draw[red][-latex](-1,0)--(-0.5,0.5)node[above]{\scriptsize$\Sigma_{5}$};
	
		\draw [ blue](-1,0)--(-0,-1);

		\draw[ blue][-latex](-1,0)--(-0.5,-0.5)node[below]{\scriptsize$\Sigma_5^*$};
		\draw[ red][-latex](-3.1,1)--(-2.45,0.5)node[above]{\scriptsize$\Sigma_{6}$};

		\draw[red](-1.8,0)--(-3.1,1);

		\draw[ blue](-1.8,0)--(-3.1,-1);
		\draw[ blue][-latex](-3.1,-1)--(-2.45,-0.5)node[below]{\scriptsize$\Sigma_6^*$};
		\draw[ red](-5,0)--(-6.5,0.9)node[above]{\scriptsize$\Sigma_{8}$};
		\draw[red][-latex](-6.5,0.9)--(-5.75,0.45);

		\draw[ blue](-5,0)--(-6.5,-0.9)node[below]{\scriptsize$\Sigma_8^*$};

		\draw[ blue][-latex](-6.5,-0.9)--(-5.75,-0.45);
		\draw[red](-4.4,0)--(-3.1,1);
		\draw[ red][-latex](-4.4,0)--(-3.75,0.5)node[above]{\scriptsize$\Sigma_{7}$};
		\draw[ blue][-latex](-4.4,0)--(-3.75,-0.5)node[below]{\scriptsize$\Sigma_7^*$};

		\draw[ red](-3.1,0)--(-3.1,1)node[above]{\scriptsize$\Sigma_{6,7}$};
		\draw[ blue](-3.1,0)--(-3.1,-1)node[below]{\scriptsize$\Sigma^*_{6,7}$};
		\draw[ red](1,0)--(0,1);
		\draw[ red][-latex](0,1)--(0.5,0.5)node[above]{\scriptsize$\Sigma_{4}$};

		\draw[ blue][-latex](0,-1)--(0.5,-0.5)node[below]{\scriptsize$\Sigma_4^*$};
		\draw[ red][-latex](1.8,0)--(2.45,0.5)node[above]{\scriptsize$\Sigma_{3}$};

		\draw[ red](1.8,0)--(3.1,1);

		\draw[ blue][-latex](1.8,0)--(2.45,-0.5)node[below]{\scriptsize$\Sigma_3^*$};
		\draw[ red](5,0)--(6.5,0.9)node[above]{\scriptsize$\Sigma_{1}$};
		\draw[ red][-latex](5,0)--(5.75,0.45);
        \draw[blue](-4.4,0)--(-3.1,-1);
        \draw[blue](0,-1)--(1,0);
        \draw[blue](1.8,0)--(3.1,-1);

		\draw[ blue](5,0)--(6.5,-0.9)node[below]{\scriptsize$\Sigma_1^*$};

		\draw[ blue][-latex](5,0)--(5.75,-0.45);
		\draw[ red](4.4,0)--(3.1,1);
		\draw[ red][-latex](3.1,1)--(3.75,0.5)node[above]{\scriptsize$\Sigma_{2}$};

		\draw[ blue](4.4,0)--(3.1,-1);

		\draw[ blue][-latex](3.1,-1)--(3.75,-0.5)node[below]{\scriptsize$\Sigma_2^*$};

		\draw[ red](3.1,0)--(3.1,1)node[above]{\scriptsize$\Sigma_{2,3}$};
		\draw[ blue](3.1,0)--(3.1,-1)node[below]{\scriptsize$\Sigma^*_{2,3}$};
		\draw[ red](0,0)--(0,1)node[above]{\scriptsize$\Sigma_{4,5}$};
		\draw[ blue](0,0)--(0,-1)node[below]{\scriptsize$\Sigma^*_{4,5}$};
		\coordinate (A) at (-5,0);
		\fill (A) circle (1pt) node[below] {$z_8$};
		\coordinate (b) at (-4.4,0);
		\fill (b) circle (1pt) node[below] {$z_7$};
		\coordinate (C) at (-1,0);
		\fill (C) circle (1pt) node[below] {$z_5$};
		\coordinate (d) at (-1.8,0);
		\fill (d) circle (1pt) node[below] {$z_6$};
		\coordinate (E) at (5,0);
		\fill (E) circle (1pt) node[below] {$z_1$};
		\coordinate (R) at (4.4,0);
		\fill (R) circle (1pt) node[below] {$z_2$};
		\coordinate (T) at (1,0);
		\fill (T) circle (1pt) node[below] {$z_4$};
		\coordinate (Y) at (1.8,0);
		\fill (Y) circle (1pt) node[below] {$z_3$};
		\coordinate (q) at (6.2,-0.1);
		\fill (q) circle (0pt) node[above] {\tiny$\Omega_{1}$};
		\coordinate (q1) at (6.2,0.05);
		\fill (q1) circle (0pt) node[below] {\tiny$\Omega_1^*$};

		\coordinate (r) at (3.4,-0.1);
		\fill (r) circle (0pt) node[above] {\tiny$\Omega_{2}$};
		\coordinate (r1) at (3.4,0.05);
		\fill (r1) circle (0pt) node[below] {\tiny$\Omega_2^*$};
		\coordinate (t) at (2.7,-0.1);
		\fill (t) circle (0pt) node[above] {\tiny$\Omega_{3}$};
		\coordinate (t1) at (2.7,0.05);
		\fill (t1) circle (0pt) node[below] {\tiny$\Omega_3^*$};
		\coordinate (k) at (0.3,-0.1);
		\fill (k) circle (0pt) node[above] {\tiny$\Omega_{4}$};
		\coordinate (k1) at (0.3,0.1);
		\fill (k1) circle (0pt) node[below] {\tiny$\Omega_4^*$};
			\coordinate (q8) at (-6.2,-0.1);
		\fill (q8) circle (0pt) node[above] {\tiny$\Omega_{8}$};
		\coordinate (q18) at (-6.2,0.05);
		\fill (q18) circle (0pt) node[below] {\tiny$\Omega_8^*$};
		\coordinate (7r) at (-3.4,-0.1);
		\fill (7r) circle (0pt) node[above] {\tiny$\Omega_{7}$};
		\coordinate (r17) at (-3.4,0.05);
		\fill (r17) circle (0pt) node[below] {\tiny$\Omega_7^*$};
		\coordinate (t6) at (-2.7,-0.1);
		\fill (t6) circle (0pt) node[above] {\tiny$\Omega_{6}$};
		\coordinate (t16) at (-2.7,0.05);
		\fill (t16) circle (0pt) node[below] {\tiny$\Omega_6^*$};

		\coordinate (k5) at (-0.3,-0.1);
		\fill (k5) circle (0pt) node[above] {\tiny$\Omega_{5}$};
		\coordinate (k15) at (-0.3,0.1);
		\fill (k15) circle (0pt) node[below] {\tiny$\Omega_5^*$};
        \draw(-5,0)--(-4.4,0);
		\draw[-latex](-5,0)--(-4.7,0);
        \draw(-1.8,0)--(-1,0);
		\draw[-latex](-2.2,0)--(-1.3,0);
        \draw(1,0)--(1.8,0);
		\draw[-latex](1,0)--(1.5,0);
        \draw(4.4,0)--(5,0);
		\draw[-latex](4,0)--(4.8,0);
        \draw[ blue][-latex](-3.1,-1)--(-3.1,-0.5);
        \draw[ red][-latex](-3.1,1)--(-3.1,0.5);
        \draw[ blue][-latex](0,-1)--(0,-0.5);
        \draw[ red][-latex](0,1)--(0,0.5);
        \draw[ blue][-latex](3.1,-1)--(3.1,-0.5);
        \draw[ red][-latex](3.1,1)--(3.1,0.5);;
		\end{tikzpicture}
	\caption{The contour $\Sigma^{(3)}$.}
	\label{fig10}
\end{figure}

Similar to region $\mathcal{P}_{I}$, in order to examine the long-term asymptotics of the initial RH problem \ref{RHP2} for $M(z)$, we have derived the hybrid $\bar{\partial}$-RH problem \ref{RHP4.5} for $M^{(3)}(z)$. Subsequently, we will proceed with the construction of the solution $M^{(3)}(z)$ in the following manner:
\begin{itemize}
  \item We first eliminate the $\bar{\partial}$ component from the solution $M^{(3)}(z)$. Subsequently, we establish the existence of a solution for the resultant pure RH problem. In addition to this, we compute the asymptotic expansion of the obtained solution.
  \item Conjugating off the solution of the first step, a pure $\bar{\partial}$-problem can be obtained.
Subsequently, we prove the existence of a solution for this newly obtained problem and bound its size.
\end{itemize}
\subsection{Asymptotic analysis on a pure RH problem}\label{s:4.2}
In the present section, we construct a solution $M^{rhp}(z)$ for the pure RH problem that is part of the $\bar{\partial}$-RH problem \ref{RHP4.5} associated with the function $M^{(3)}(z)$. Additionally, we calculate the asymptotic expansion of this solution $M^{rhp}(z)$. By excluding the $\bar{\partial}$ component of $M^{(3)}(z)$, the function $M^{rhp}(z)$ adheres to the following pure RH problem.
\begin{RHP}\label{RHP4.6}
    		Find a matrix valued function $M^{rhp}\triangleq M^{rhp}(z;y,t)$ admits:
    		\begin{enumerate}[($i$)]
    			\item Analyticity:~$M^{rhp}$ is analytical in $\mathbb{C}\setminus\left(\Sigma^{(3)}\cup\left\{\zeta_{n},\bar{\zeta}_{n}\right\}_{n\in\Lambda}\right)$;
                \item Symmetry:$M^{rhp}(z)=\sigma_3\overline{M^{rhp}(-\bar{z})}\sigma_3=F^{-2}\overline{M^{rhp}(-\bar{z}^{-1})}=F^2\sigma_3M^{rhp}(-z^{-1})\sigma_3$;
    			\item Jump condition:
    			\begin{equation}\nonumber
    M^{rhp}_+(z)=M^{rhp}_-(z)V^{(3)}(z),
\end{equation}
	where $V^{(3)}(z)$ is given by \eqref{4.39}.
    			\item Asymptotic behavior:
                $M^{rhp}$ has the same asymptotics with $M^{(3)}(z)$;
\item Residue conditions:  $M^{rhp}(z)$ has the same residue condition with $M^{3}(z)$.
    		\end{enumerate}
    	\end{RHP}

    \subsubsection{Solving the pure RH problem}\label{s:3.3.1}
Firstly, we define the small neighborhood $U^{(j)}$ (shown in Fig. \ref{fig11}) as
\begin{equation}\nonumber
    U^{(1)}=\left\{z:|z-\left(2+\sqrt{3}\right)|\leqslant \varrho/2\right\},\quad U^{(2)}=\left\{z:|z-\left(2-\sqrt{3}\right)|\leqslant \varrho/2\right\},
\end{equation}
\begin{equation}\nonumber
    U^{(3)}=\left\{z:|z-(-2+\sqrt{3})|\leqslant \varrho/2\right\},\quad U^{(4)}=\left\{z:|z-(-2-\sqrt{3})|\leqslant \varrho/2\right\},
\end{equation}
\begin{figure}[h]
\centering
		\begin{tikzpicture}
		\draw[dashed](-6.5,0)--(6.5,0)node[right]{ Re$z$};

		\coordinate (I) at (0,0);
		\fill (I) circle (1pt) node[below] {$0$};
		
			\draw (1,0)--(1.8,0);
\draw [-latex](1,0)--(1.4,0);
			\draw (3.8,0)--(4.6,0);
\draw [-latex](3.8,0)--(4.2,0);
			\draw (-1.8,0)--(-1,0);
\draw [-latex](-1.8,0)--(-1.4,0);
			\draw (-4.6,0)--(-3.8,0);
\draw [-latex](-4.6,0)--(-4.2,0);

\draw (0.5,0.5)--(1,0);
\draw [-latex](0.5,0.5)--(0.8,0.2);
\draw (0.5,-0.5)--(1,0);
\draw [-latex](0.5,-0.5)--(0.8,-0.2);
\draw (1.8,0)--(2.3,0.5);
\draw [-latex](1.8,0)--(2.1,0.3);
\draw (1.8,0)--(2.3,-0.5);
\draw [-latex](1.8,0)--(2.1,-0.3);

\draw (3.3,0.5)--(3.8,0);
\draw [-latex](3.3,0.5)--(3.6,0.2);
\draw (3.3,-0.5)--(3.8,0);
\draw [-latex](3.3,-0.5)--(3.6,-0.2);
\draw (4.6,0)--(5.1,0.5);
\draw [-latex](4.6,0)--(4.9,0.3);
\draw (4.6,0)--(5.1,-0.5);
\draw [-latex](4.6,0)--(4.9,-0.3);

\draw (-0.5,0.5)--(-1,0);
\draw [-latex](-1,0)--(-0.7,0.3);
\draw (-0.5,-0.5)--(-1,0);
\draw [-latex](-1,0)--(-0.7,-0.3);
\draw (-1.8,0)--(-2.3,0.5);
\draw [-latex](-2.3,0.5)--(-2.0,0.2);
\draw (-1.8,0)--(-2.3,-0.5);
\draw [-latex](-2.3,-0.5)--(-2.0,-0.2);

\draw (-3.3,0.5)--(-3.8,0);
\draw [-latex](-3.8,0)--(-3.5,0.3);
\draw (-3.3,-0.5)--(-3.8,0);
\draw [-latex](-3.8,0)--(-3.5,-0.3);
\draw (-4.6,0)--(-5.1,0.5);
\draw [-latex](-5.1,0.5)--(-4.8,0.2);
\draw (-4.6,0)--(-5.1,-0.5);
\draw [-latex](-5.1,-0.5)--(-4.8,-0.2);

		\coordinate (A) at (-4.6,0);
		\fill (A) circle (1pt) node[below] {$z_8$};

		\coordinate (b) at (-3.8,0);
		\fill (b) circle (1pt) node[below] {$z_7$};

		\coordinate (C) at (-1,0);
		\fill (C) circle (1pt) node[below] {$z_5$};

		\coordinate (d) at (-1.8,0);
		\fill (d) circle (1pt) node[below] {$z_6$};

		\coordinate (E) at (4.6,0);
		\fill (E) circle (1pt) node[below] {$z_1$};

		\coordinate (R) at (3.8,0);
		\fill (R) circle (1pt) node[below] {$z_2$};

		\coordinate (T) at (1,0);
		\fill (T) circle (1pt) node[below] {$z_4$};

		\coordinate (Y) at (1.8,0);
		\fill (Y) circle (1pt) node[below] {$z_3$};

\node at (4.2,0.7) {\footnotesize $ U^{(1)}$};
\node at (1.4,0.7) {\footnotesize $ U^{(2)}$};
\node at (-1.4,0.7) {\footnotesize $ U^{(3)}$};
\node at (-4.2,0.7) {\footnotesize $ U^{(4)}$};

        \draw[thick,blue] (1.4,0) circle (1);
        \draw[thick,blue] (-1.4,0) circle (1);
        \draw[thick,blue] (4.2,0) circle (1);
        \draw[thick,blue] (-4.2,0) circle (1);
		\end{tikzpicture}
	\caption{The small neighborhood $U^{(j)}$.}
	\label{fig11}
\end{figure}
Then, we could decompose $M^{rhp}$ into five parts
\begin{equation}\label{4.51}
    M^{rhp}=\begin{cases}E(z),\quad z\in\mathbb{C}\setminus(U^{(1)}\cup U^{(2)}\cup U^{(3)}\cup U^{(4)}),\\
    E(z)M^{out}(z)M^{rhp1}(z),\quad z\in U^{(1)},&\\
    E(z)M^{out}(z)M^{rhp2}(z),\quad z\in U^{(2)},&\\
    E(z)M^{out}(z)M^{rhp3}(z),\quad z\in U^{(3)},&\\
    E(z)M^{out}(z)M^{rhp4}(z),\quad z\in U^{(4)}.&\\\end{cases}
\end{equation}
On the basis of the definition of $\rho$, it implies that $M^{rhpj}(z)$ possesses no poles in $U^{(j)}$. Additionally, $M^{out}(z)$ solves a model RH problem, $M^{rhpj}(z)$ can be solved by using a known Painlev\'{e} II model in $U^{(j)}$, and $E(z)$ is an error function which is a solution of a small-norm RH problem.

Next, we will study $M^{out}(z)$, $M^{rhpj}(z)$, and $E(z)$ separately. For $M^{out}(z)$, similar to the previous analysis, we have
\begin{RHP}\label{RHP4.4}
    		Find a matrix valued function $M_{\Lambda}^{out}(z)=M_{\Lambda}^{out}(z,y,t)$ admits:
    		\begin{enumerate}[($i$)]
    			\item Analyticity:~$M_{\Lambda}^{out}(z)$ is analytical in $\mathbb{C}\setminus\left\{\zeta_n,\bar{\zeta}_n\right\}_{n\in\Lambda}$;

                \item Symmetry:$M_{\Lambda}^{out}(z)=\sigma_3\overline{M_{\Lambda}^{out}(-\bar{z})}\sigma_3=F^{-2}\overline{M_{\Lambda}^{out}(-\bar{z}^{-1})}=F^2\sigma_3M_{\Lambda}^{out}(-z^{-1})\sigma_3$;
    			\item Asymptotic behavior:
                \begin{equation}\nonumber
                M_\Lambda^{out}(z)=I+\mathcal{O}(z^{-1}),\quad z\to\infty;\end{equation}
\item Residue conditions:$M_{\Lambda}^{out}(z)$ has simple poles at each point $\zeta_n$ and $\bar{\zeta}_n$ for $n\in\Lambda$ with:
\begin{equation}\nonumber
\operatorname*{\mathrm{Res}}_{z=\zeta_n}M_\Lambda^{out}(z)=\lim_{z\to\zeta_n}M_\Lambda^{out}(z)
\begin{pmatrix}0&0\\C_ne^{-2it\theta_n}T^2(\zeta_n)&0\end{pmatrix},\end{equation}
\begin{equation}\nonumber
\left.\operatorname*{Res}_{z=\bar{\zeta}_n}M_\Lambda^{out}(z)=\lim_{z\to\bar{\zeta}_n}M_\Lambda^{out}(z)
\left(\begin{array}{cc}0&-\bar{C}_nT^{-2}(\bar{\zeta}_n)e^{2it\bar{\theta}_n}\\0&0\end{array}\right.\right).\end{equation}
    		\end{enumerate}
    	\end{RHP}
    \begin{equation}\nonumber
M^{out}(z)=M_{\Lambda}^{out}(z)(I+\mathcal{O}(e^{-ct})),~~t\to\infty.\end{equation}
Morever, denote the asymptotic expansion of $M_{\Lambda}^{out}(z)$ as $z\to i$:
\begin{equation}\label{4.21}
M_{\Lambda}^{out}(z)=M_\Lambda^{out}(i)+M_{\Lambda,1}^{out}(z-i)+\mathcal{O}((z-i)^{-2}).\end{equation}
\begin{proposition}\label{prop4.5}
The RH problem \ref{RHP4.4} possesses unique solution. This
fact could be guaranteed by the Liouville's theorem. Moreover, $M_{\Lambda}^{out}(z)$ possesses equivalent a reflectionless solution
to the original RH problem \ref{RHP2} with modified scattering data $\tilde{\mathcal{D}}_{\Lambda}=\left\{0,\left\{\zeta_n,C_nT^2(\zeta_n)\right\}_{n\in\Lambda}\right\}$
as follows:

i. If $\Lambda=\emptyset$, then
\begin{equation}M_{\Lambda}^{out}(z)=I.\end{equation}

ii. If $\Lambda\neq\emptyset$, without loss of generality, we assume that there exist s discrete spectral points
belonging to $\Lambda$, i.e., $\Lambda=\{j_{1},j_{2},\ldots,j_{s}\}$, then
\begin{equation}
M_{\Lambda}^{out}(z)=I+\sum_{k=1}^{s}\begin{pmatrix}\frac{\beta_k}{z-\zeta_{j_k}}&\frac{-\overline{\alpha_k}}{z-\zeta_{j_k}}\\
\frac{\alpha_k}{z-\zeta_{j_k}}&\frac{\overline{\beta_k}}{z-\overline{\zeta}_{j_k}}\end{pmatrix},\end{equation}
where $\alpha_k=\alpha_k(x,t)$ and $\beta_k=\beta_k(x,t)$ with linearly dependant equations:
\begin{equation}c_{j_k}^{-1}T(z_{j_k})^{-2}e^{-2i\theta(z_{j_k})t}\beta_k=\sum_{h=1}^{\mathcal{N}}\frac{-\overline{\varsigma_h}}{\zeta_{j_k}-\overline{\zeta}_{j_h}},\end{equation}
\begin{equation}c_{j_k}^{-1}T(z_{j_k})^{-2}e^{-2i\theta(z_{j_k})t}\varsigma_k=1+\sum_{h=1}^{\mathcal{N}}\frac{\overline{\beta_h}}{\zeta_{j_k}-\overline{\zeta}_{j_h}},\end{equation}
for $k=1,2,\ldots,s$.

\end{proposition}
Then, we could derive the soliton solutions for the $M_{\Lambda}^{out}(z)$. With reflection coefficients $r(s)\equiv0$, the scattering matrices $S(z)\equiv I$. Denote $u_{p}(y,t,\tilde{\mathcal{D}}_{\Lambda})$ is the $N(\Lambda)$-soliton with scattering data $\tilde{\mathcal{D}}_{\Lambda}=\left\{0,\left\{\zeta_{n},C_{n}T^{2}(\zeta_{n})\right\}_{n\in\Lambda}\right\}$. Consequently, the solution $u_{p}(y,t,\tilde{\mathcal{D}}_{\Lambda})$ is given by

\begin{equation}
u_{p}(y,t,\tilde{\mathcal{D}}_{\Lambda})=\lim_{z\to i}\frac{1}{z-i}\left(1-\frac{([M_{\Lambda}^{out}]_{11}(z)+[M_{\Lambda}^{out}]_{21}(z))([M_{\Lambda}^{out}]_{12}(z)
+[M_{\Lambda}^{out}]_{22}(z))}{([M_{\Lambda}^{out}]_{11}(i)+[M_{\Lambda}^{out}]_{21}(i))([M_{\Lambda}^{out}]_{12}(i)
+[M_{\Lambda}^{out}]_{22}(i))}\right),
\end{equation}

where
\begin{equation}\label{c+2}
x(y,t;\tilde{\mathcal{D}}_\Lambda)=y+c_+^{out}(y,t;\tilde{\mathcal{D}}_\Lambda)
=y-\ln\left(\frac{[M_\Lambda^{out}]_{12}(i)+[M_\Lambda^{out}]_{22}(i)}{[M_\Lambda^{out}]_{11}(i)+[M_\Lambda^{out}]_{21}(i)}\right).\end{equation}
When $\Lambda=\emptyset$,
\begin{equation}
u_{p}(y,t;\tilde{\mathcal{D}}_\Lambda)=c_+^{out}(y,t;\tilde{\mathcal{D}}_\Lambda)=0.
\end{equation}
When $\Lambda\neq\varnothing$ with $\Lambda=\{\zeta_{j_{k}}\}_{k=1}^{\mathcal{N}}$
\begin{equation}\label{up2}
\begin{aligned}
u_{p}(y,t;\tilde{\mathcal{D}}_\Lambda)=&\left[\sum_{k=1}^{\mathcal{N}}\left(\frac{-\overline{\alpha_{k}}}{(i-\bar{\zeta}_{j_{k}})^{2}}
+\frac{\overline{\beta_{k}}}{(i-\bar{\zeta}_{j_{k}})^{2}}\right)\right]/\left[1+\sum_{k=1}^{\mathcal{N}}
\left(\frac{-\overline{\alpha_{k}}}{i-\bar{\zeta}_{j_{k}}}+\frac{\overline{\beta_{k}}}{i-\bar{\zeta}_{j_{k}}}\right)\right]\\
&+\left[\sum_{k=1}^{\mathcal{N}}\frac{\beta_k}{(i-\zeta_{j_k})^2}+\frac{\alpha_k}{(i-\zeta_{j_k})^2}\right]/
\left[1+\sum_{k=1}^{\mathcal{N}}\left(\frac{\beta_k}{i-\zeta_{j_k}}+\frac{\alpha_k}{i-\zeta_{j_k}}\right)\right],
\end{aligned}
\end{equation}
\begin{equation}
x(y,t;\tilde{\mathcal{D}}_\Lambda)=y-\ln\left(\frac{1+\sum_{k=1}^{\mathcal{N}}\left(\frac{-\overline{\alpha_{k}}}{i-\zeta_{j_{k}}}
+\frac{\overline{\beta_{k}}}{i-\zeta_{j_{k}}}\right)}{1+\sum_{k=1}^{\mathcal{N}}\left(\frac{\beta_{k}}{i-\zeta_{j_{k}}}+\frac{\alpha_{k}}{i-\zeta_{j_{k}}}\right)}\right).
\end{equation}

Next, we study $M^{rhpj}(z)$. We analyze the local properties of the phase function $t\theta(z)$ near $z=\pm(2\pm\sqrt3)$. Their analysis is similar, let's take $z\to2+\sqrt3$ an examples, we find that
\begin{equation}\label{4.45}
    t\theta(z)=\frac{4}{3}\hat{k}^3+\tilde{s}\hat{k}+t\theta\left(2+\sqrt{3}\right)+\mathcal{O}\left(t^{-1/3}\hat{k}^4\right),
\end{equation}
where
\begin{equation}\label{4.46}
    \tilde{s}=-\left(\frac{8}{9}\right)^{1/3}\left(\frac{y}{t}+\frac{1}{4}\right)t^{2/3},~~~\hat{k}
    =\left[\frac{9}{8}\left(26-15\sqrt{3}\right)t\right]^{1/3}\left(z-\left(2+\sqrt{3}\right)\right).
\end{equation}
Significantly, the first two terms $\frac{4}{3}\hat{k}^3+\tilde{s}\hat{k}$ play a central and decisive role in the task of making the local region conform to the Painlev\'{e} model.

It is worth noting that in the transition region $\mathcal{P}_{II}$, as $t\to\infty$, according to the formula \eqref{4.19},
 $z_1$ and $z_2$ merge to $2+\sqrt3$, $z_3$ and $z_4$ merge to $2-\sqrt3$, $z_5$ and $z_6$ merge to $-2+\sqrt3$ and $z_7$ and $z_8$ merge to $-2-\sqrt3$ in the $z$-plane. Correspondingly, we demonstrate
that two scaled phase points
$\hat{k}_j=\left[\frac{9}{8}\left(26-15\sqrt{3}\right)t\right]^{1/3}\left(z_j-\left(2+\sqrt{3}\right)\right)$, $j=1,2$, $\hat{k}_j=\left[\frac{9}{8}\left(26+15\sqrt{3}\right)t\right]^{1/3}\left(z_j-\left(2-\sqrt{3}\right)\right)$, $j=3,4$, $\hat{k}_j=\left[\frac{9}{8}\left(26+15\sqrt{3}\right)t\right]^{1/3}\left(z_j-\left(-2+\sqrt{3}\right)\right)$, $j=5,6$, $\hat{k}_j=\left[\frac{9}{8}\left(26-15\sqrt{3}\right)t\right]^{1/3}\left(z_j-\left(-2-\sqrt{3}\right)\right)$, $j=7,8$, are always in a fixed interval
in the $k$-plane.
\begin{proposition}
In the transition region $\mathcal{P}_{II}$ and under scaling transformation \eqref{4.46}, then for large enough t, we have
\begin{equation}\nonumber
    |\hat{k}_j|\leqslant\left[\frac{9}{8}\left(26-15\sqrt{3}\right)\right]^{1/3}\sqrt{2C}, j=1,2,7,8,
\end{equation}
\begin{equation}\nonumber
    |\hat{k}_j|\leqslant\left[\frac{9}{8}\left(26+15\sqrt{3}\right)\right]^{1/3}\sqrt{2C}, j=3,4,5,6.
\end{equation}
\end{proposition}
\begin{proof}
We take $j=1,2$ as an example to prove. If $0\leqslant(\xi+\frac{1}{4})t^{\frac{2}{3}}\leqslant C$, it follows from \eqref{4.17} that
\begin{equation}\nonumber
    0\leqslant s_{+}\leqslant\frac{C}{2}t^{-2/3}, -\frac{C}{2}t^{-2/3}\leqslant s_-<0.
\end{equation}
Thus, by \eqref{4.15} and \eqref{4.16}
\begin{equation}\nonumber
    \left|z_j-\left(2+\sqrt{3}\right)\right|\leqslant\sqrt{2C}t^{-1/3},j=1,2.
\end{equation}
Finally, by \eqref{4.46} we get
\begin{equation}\nonumber
    |\hat{k}_j|\leqslant\left[\frac{9}{8}\left(26-15\sqrt{3}\right)\right]^{1/3}\sqrt{2C}, j=1,2.
\end{equation}
\end{proof}
For a fix constant
\begin{equation}\nonumber
    c_0:=\min\left\{1-\frac{\sqrt{3}}{2},2\left(z_1-2-\sqrt{3}\right)t^{\delta_2},2\left(z_3-2+\sqrt{3}\right)t^{\delta_2}\right\},\quad\delta_2\in(1/27,1/12),
\end{equation}
we further define four open disks
 \begin{equation}\nonumber
    U^{(1)}=\left\{z:|z-\left(2+\sqrt{3}\right)|\leqslant c_0\right\},\quad U^{(2)}=\left\{z:|z-\left(2-\sqrt{3}\right)|\leqslant c_0\right\},
\end{equation}
\begin{equation}\nonumber
    U^{(3)}=\left\{z:|z-(-2+\sqrt{3})|\leqslant c_0\right\},\quad U^{(4)}=\left\{z:|z-(-2-\sqrt{3})|\leqslant c_0\right\}.
\end{equation}
This particularly implies that $c_0\lesssim t^{\delta_2-1/3}\to0$ as $t\to+\infty$, hence, $U^{(1)}$, $U^{(2)}$, $U^{(3)}$ and $U^{(4)}$ are four shrinking
disks with respect to $t$. As previously stated, which appeal us to construct
the solution of $M^{rhpj}(z)$ as follows:
\begin{RHP}\label{RHP4.8}
    		Find a matrix valued function $M^{rhpj}(z)\triangleq M^{rhpj}(z;y,t)$ admits:
    		\begin{enumerate}[($i$)]
    			\item Analyticity:~$M^{rhpj}(z)$ is analytical in $\mathbb{C}\setminus\Sigma^{(j)}$, where
    \begin{equation}\nonumber
    \Sigma^{(rhpj)}:=U^{(j)}\cap\Sigma^{(3)}.
\end{equation}

    			\item Jump condition:
    			\begin{equation}\nonumber
    M^{rhpj}_+(z)=M^{rhpj}_-(z)V^{(3)}(z),
\end{equation}
	where $V^{(3)}(z)$ is given by \eqref{4.39}.
    			\item Asymptotic behavior:
                $M^{rhpj}(z)$ has the same asymptotics with $M^{(3)}(z)$.

    		\end{enumerate}
    	\end{RHP}
To solve the RH problem for $M^{rhp1}(z;y,t)$, we split the analysis into the following steps.\\
\textbf{Step I: Variable transformation.}
Under the change of variable \eqref{4.46}, the contour $\Sigma^{(rhp1)}$ is
changed into a contour $\hat{\Sigma}^{(rhp1)}$ in the $\hat{k}$-plane.
\begin{equation}\label{4.54}
    \hat{\Sigma}^{(rhp1)}:=\bigcup_{j=1,2}\left(\hat{\Sigma}_j^{(rhp1)}\cup\hat{\Sigma}_j^{(rhp1)*}\right)\cup(\hat{k}_2,\hat{k}_1).
\end{equation}
Here, $\hat{k}_{j}=\left[\frac{9}{8}\left(26-15\sqrt{3}\right)t\right]^{1/3}\left(z_{j}-\left(2+\sqrt{3}\right)\right)$ and
\begin{equation}\nonumber
    \hat{\Sigma}_1^{(rhp1)}=\left\{\hat{k}:\hat{k}-\hat{k}_1=le^{(\varphi_0)i},0\leqslant l\leqslant c_0\left[\frac{9}{8}\left(26-15\sqrt{3}\right)t\right]^{1/3}\right\},
\end{equation}
\begin{equation}\nonumber
    \hat{\Sigma}_2^{(rhp1)}=\left\{\hat{k}:\hat{k}-\hat{k}_2=le^{(\pi-\varphi_0)i},0\leqslant l\leqslant c_0\left[\frac{9}{8}\left(26-15\sqrt{3}\right)t\right]^{1/3}\right\}.
\end{equation}
Further, RH problem \ref{RHP4.8} becomes the following RH problem in the $\hat{k}$-plane.

\begin{RHP}
    		Find a matrix valued function $M^{rhp1}(\hat{k})\triangleq M^{rhp1}(\hat{k};y,t)$ admits:
    		\begin{enumerate}[($i$)]
    			\item Analyticity:~$M^{rhp1}(\hat{k})$ is analytical in $\mathbb{C}\setminus \hat{\Sigma}^{(rhp1)}$;
    			\item Jump condition: For $\hat{k}\in\hat{\Sigma}^{(rhp1)}$
    			\begin{equation}\nonumber
    M^{rhpj}_+(\hat{k})=M^{rhpj}_-(\hat{k})V^{rhp1}(\hat{k}),
\end{equation}
	where
\begin{equation}\nonumber
    \begin{aligned}V^{(rhp1)}(\hat{k})&=\begin{cases}e^{i\theta(z(\hat{k}))\hat{\sigma}_3}\begin{pmatrix}1&0\\\bar{d}(z(\hat{k}))&1\end{pmatrix}
    \begin{pmatrix}1&d(z(\hat{k}))\\0&1\end{pmatrix},&\hat{k}\in(\hat{k}_2,\hat{k}_1),\\
    \begin{pmatrix}1&d(z_j)e^{2it\theta(2+\sqrt{3}+\left(\frac{9}{8}(26-15\sqrt{3})t\right)^{-1/3}\hat{k})}\\0&1\end{pmatrix},&\hat{k}\in\hat{\Sigma}_j^{(rhp1)},j=1,2,\\
    \begin{pmatrix}1&0\\\bar{d}(z_j)e^{-2it\theta(2+\sqrt{3}+\left(\frac{9}{8}(26-15\sqrt{3})t\right)^{-1/3}\hat{k})}&1\end{pmatrix},&\hat{k}\in\hat{\Sigma}_j^{(rhp1)*}j=1,2.
    \end{cases}\end{aligned}
\end{equation}
    			\item Asymptotic behavior:
                $M^{rhp1}(\hat{k})$ has the same asymptotics with $M^{(3)}(z)$.

    		\end{enumerate}
    	\end{RHP}

\textbf{Step II: Matching the model.}
In order to use the standard model to approximate this RH problem for $M^{rhp1}(\hat{k})$, the following proposition is a prerequisite.
It is natural to expect that $M^{rhp1}(\hat{k})$ is
well-approximated by the following model RH problem for $\hat{M}^{rhp1}(\hat{k})$.
\begin{RHP}\label{RHP4.9}
    		Find a matrix valued function $\hat{M}^{rhp1}(\hat{k})$ admits:
    		\begin{enumerate}[($i$)]
    			\item Analyticity:~$\hat{M}^{rhp1}(\hat{k})$ is analytical in $\mathbb{C}\setminus\hat{\Sigma}^{(1)}$, where
    \begin{equation}\nonumber
    \Sigma^{(rhp1)}:=U^{(1)}\cap\Sigma^{(3)}.
\end{equation}

    			\item Jump condition:
    			\begin{equation}\nonumber
    \hat{M}^{rhp1}_+(\hat{k})=\hat{M}^{rhp1}_-(\hat{k})\hat{V}^{rhp1}(\hat{k}),
\end{equation}
	where
\begin{equation}\nonumber
    \hat{V}^{rhp1}(\hat{k})=\begin{cases}e^{i\left(\frac{4}{3}\hat{k}^{3}+\tilde{s}\hat{k}\right)\hat{\sigma}_{3}}
    \begin{pmatrix}1&0\\ \bar{R}_{a}&1\end{pmatrix}\begin{pmatrix}1&R_{a}\\0&1\end{pmatrix},&\hat{k}\in(\hat{k}_{2},\hat{k}_{1}),\\
    e^{i\left(\frac{4}{3}\hat{k}^{3}+\tilde{s}\hat{k}\right)\hat{\sigma}_{3}}\begin{pmatrix}1&R_{a}\\0&1\end{pmatrix},&\hat{k}\in\hat{\Sigma}_{j}^{(rhp1)},j=1,2,\\
    e^{i\left(\frac{4}{3}\hat{k}^{3}+\tilde{s}\hat{k}\right)\hat{\sigma}_{3}}\begin{pmatrix}1&0\\\bar{R}_{a}&1\end{pmatrix},&\hat{k}\in\hat{\Sigma}_{j}^{(rhp1)*},j=1,2,
    \end{cases}
\end{equation}
with
\begin{equation}\label{4.59}
    \begin{aligned}R_{a}&:=d\left(2+\sqrt{3}\right)e^{2i\theta\left(2+\sqrt{3}\right)}\\
    &=\bar{r}(2+\sqrt{3})\prod_{j\in\Delta}
    \left(\frac{2+\sqrt{3}-\zeta_n}{\bar{\zeta}_n^{-1}(2+\sqrt{3})-1}\right)^{-2}\exp\left\{\frac{1}{\pi i}\int_{\mathbb{R}}\frac{\log\left(1+|r(s)|^2\right)}{s-(2+\sqrt{3})}ds\right\}e^{2i\theta\left(2+\sqrt{3}\right)}.\end{aligned}
\end{equation}
    			\item Asymptotic behavior:
                \begin{equation}\nonumber
    \hat{M}^{rhp1}(\hat{k})=I+\mathcal{O}(\hat{k}^{-1}).
\end{equation}

    		\end{enumerate}
    	\end{RHP}
    By writing $R_{a}=|R_{a}|e^{i\phi_{a}}$, it is readily seen from \eqref{4.59} that
    \begin{equation}\nonumber
    |R_a|=|\bar{r}(2+\sqrt{3})|,
\end{equation}
\begin{equation}\label{4.62}
\phi_a=\arg d(2+\sqrt{3})+2t\theta(2+\sqrt{3}).\end{equation}

Subsequently, we will provide an explicit solution to the aforementioned RH problem by leveraging the Painlev\'{e} II parametrix. Consideration in the rest part is to establish the error between the
RH problems for $M^{rhp1}(\hat{k})$ and $\hat{M}^{rhp1}(\hat{k})$ for large $t$.

\begin{proposition}\label{prop4.10}
As $t\to+\infty$, we have
\begin{equation}\label{4.333}
M^{rhp1}(\hat{k})=\hat{M}^{rhp1}(\hat{k})+\mathcal{O}\left(t^{-1/3+4\delta_2}\right),\quad t\to+\infty.
\end{equation}
\end{proposition}

\begin{figure}[h]
	\centering
	\begin{tikzpicture}[node distance=2cm]
		\draw[dash pattern={on 0.84pt off 2.51pt}][->](-3.6,0)--(4,0)node[right]{ };
		\draw[dash pattern={on 0.84pt off 2.51pt}][->](0,-1.8)--(0,1.8)node[above]{ };
        \draw(-2.5,1)--(-1,0);
       \draw(-2.5,-1)--(-1,0);
\draw(2.5,1)--(1,0);
\draw(2.5,-1)--(1,0);
\draw(-1,0)--(1,0);
\draw[-latex](-1,0)--(-0.3,0);
\draw[-latex](-2.5,1)--(-1.725,0.5);
\draw[-latex](-2.5,-1)--(-1.725,-0.5);
\draw[color=blue] (-1.5,1.5)--(1.5,-1.5);
\draw[color=blue] (1.5,1.5)--(-1.5,-1.5);
\draw(1,0)--(2.5,1)node[above]{\footnotesize$\hat{\Sigma}_1^{(rhp1)}$};
\draw(1,0)--(2.5,-1)node[above]{\footnotesize$\hat{\Sigma}_1^{(rhp1)*}$};
\draw(-1,0)--(-2.5,1)node[above]{\footnotesize$\hat{\Sigma}_2^{(rhp1)}$};
\draw(-1,0)--(-2.5,-1)node[above]{\footnotesize$\hat{\Sigma}_2^{(rhp1)*}$};
        \coordinate (b) at (1,0);
			\fill (b) circle (1pt) node[below] {$\hat{k}_{1}$};
	\coordinate (f) at (-1,0);
			\fill (f) circle (1pt) node[below] {$\hat{k}_{2}$};
\draw[-latex](1,0)--(1.725,0.5);
\draw[-latex](1,0)--(1.725,-0.5);
\draw[color=blue][-latex](-1.5,-1.5)--(-0.4,-0.4);		
\draw[color=blue][-latex](0,0)--(0.6,0.6);		
\draw[color=blue][-latex](0,0)--(0.6,-0.6);			
\draw[color=blue][-latex](-1.5,1.5)--(-0.4,0.4);

        \node at (0.8,0.3) {\footnotesize $\hat{\Omega}$};
        \node at (0.6,-0.3) {\footnotesize $\hat{\Omega}^*$};
        \node at (-0.6,-0.3) {\footnotesize $\hat{\Omega}^*$};
        \node at (-0.8,0.3) {\footnotesize $ \hat{\Omega}$};

		\coordinate (I) at (0.2,0);
		\fill (I) circle (0pt) node[below] {$0$};

	\end{tikzpicture}
	\caption{ The jump contours of the RH problems for $\hat{M}^{rhp1}(\hat{k})$ (black) and $M^P$ (blue). }
	\label{fig12}
\end{figure}
We could solve the RH problem for $\hat{M}^{rhp1}(\hat{k})$ explicitly by using the Painlev\'{e} II parametrix $M^P(z;s,\kappa)$
introduced in the Appendix \ref{appendix}. More precisely, define
\begin{equation}\label{4.72}
    \hat{M}^{rhp1}(\hat{k})=e^{i\left(\frac{\phi_a}{2}-\frac{\pi}{4}\right)\sigma_3}
    \sigma_1M^P\left(\hat{k};\tilde{s},-|r\left(2+\sqrt{3}\right)|\right)\hat{H}(\hat{k})\sigma_1e^{-i\left(\frac{\phi_a}{2}-\frac{\pi}{4}\right)\sigma_3},
\end{equation}
where $\phi_a$ is given in \eqref{4.62},
\begin{equation}\nonumber
    \begin{gathered}H(\hat{k})=\begin{cases}e^{-i\left(\frac{4}{3}\hat{k}^{3}+\tilde{s}\hat{k}\right)\hat{\sigma}_3}
    \left(\begin{array}{cc}1&0\\i|r\left(2+\sqrt{3}\right)|&1\end{array}\right),&\hat{k}\in\hat{\Omega},\\\\
    e^{-i\left(\frac{4}{3}\hat{k}^{3}+\tilde{s}\hat{k}\right)\hat{\sigma}_3}\left(\begin{array}{cc}1&i|r\left(2+\sqrt{3}\right)|\\0&1\end{array}\right),
    &\hat{k}\in\hat{\Omega}^{*},\\
    I,&\text{elsewhere,}\end{cases}\end{gathered}
\end{equation}
and where the region $\hat{\Omega}$ is illustrated in Fig.\ref{fig12}.

In view of RH problem \ref{PII} it is readily seen that $\hat{M}^{rhp1}(\hat{k})$ in \eqref{4.72} indeed solves RH problem \ref{RHP4.9}.
Moreover, as a corollary of Proposition \ref{prop4.10}, we have the following result.
\begin{corollary}
As $\hat{k}\to\infty$,
\begin{equation}\label{4.444}
    M^{rhp1}(\hat{k})=I+\frac{M_{1}^{rhp1}(\tilde{s})}{\hat{k}}+\mathcal{O}(\hat{k}^{-2}),
\end{equation}
where
\begin{equation}
    M_1^{rhp1}=\frac{i}{2}\left(\begin{array}{cc}\int_{\tilde{s}}^{+\infty}v_{II}^{2}(\zeta)d\zeta&-v_{II}(\tilde{s})e^{i\phi_{a}}\\
    v_{II}(\tilde{s})e^{-i\phi_{a}}&-\int_{\tilde{s}}^{+\infty}v_{II}^{2}(\zeta)d\zeta\end{array}\right)+\mathcal{O}\left(t^{-1/3+4\delta_2}\right),
\end{equation}
where $v_{II}(\tilde{s})$ is the unique solution of Painlev\'{e} II equation \eqref{v}, fixed
by the boundary condition
\begin{equation}\label{4.75}
    v_{II}(\tilde{s})\to|r\left(2+\sqrt{3}\right)|\mathrm{Ai}(\tilde{s}),\quad\tilde{s}\to+\infty,
\end{equation}
with $|r(2+\sqrt3)|<1.$

\end{corollary}

Finally, the RH problem associated with the function $M^{rhp2}(z)$, $M^{rhp3}(z)$ and $M^{rhp4}(z)$
can be resolved through a procedure
that closely parallels the methods employed previously. In fact, when considering the situation where
the complex variable $z$ is an element of the set $z\in U^{(2)}$ and the parameter $t$ attains a sufficiently large
value, we have
\begin{equation}\nonumber
    t\theta(z)=t\theta\left(2-\sqrt{3}\right)+\frac{4}{3}\check{k}^3+\tilde{s}\check{k}+\mathcal{O}\left(t^{-1/3}\check{k}^4\right),
\end{equation}
where $\tilde{s}$ is defined in \eqref{4.46} and
\begin{equation}\nonumber
    \check{k}=\left[\frac{9}{8}\left(26+15\sqrt3\right)t\right]^{1/3}\left(z-\left(2-\sqrt3\right)\right).
\end{equation}
Following previous steps we have performed, we
could approximate $M^{rhp2}(\check{k})$
\begin{equation}\nonumber
    M^{rhp2}(\check{k})=I+\frac{M_1^{rhp2}}{\check{k}}+\mathcal{O}(\check{k}^{-2}),
\end{equation}
where, as $t\to+\infty,$

\begin{equation}\label{4.79}
   M^{rhp2}_1(\tilde{s})=\frac{i}{2}\begin{pmatrix}\int_{\tilde{s}}^{+\infty}v_{II}^{2}(\zeta)d\zeta&-v_{II}(\tilde{s})e^{i\phi_{b}}\\
    v_{II}(\tilde{s})e^{-i\phi_{b}}&-\int_{\tilde{s}}^{+\infty}v_{II}^{2}(\zeta)d\zeta\end{pmatrix}+\mathcal{O}\left(t^{-1/3+4\delta_2}\right),
\end{equation}
with the same $v_{II}$ in \eqref{4.75} and
\begin{equation}\label{4.81}
    \begin{aligned}\phi_{b}(\tilde{s},t)&=\arg d\left(2-\sqrt{3}\right)+2t\theta\left(2-\sqrt{3}\right).\end{aligned}
\end{equation}
Here, we also need to use the fact that $r(z)=\overline{r(\bar{z}^{-1})}=r(-z^{-1})=-\overline{r(-\bar{z})}$,
which particularly implies that $|r(2+\sqrt{3})|=|r\left(2-\sqrt{3}\right)|$.

For $z\in U^{(3)}$ and $t$ large enough, we have
\begin{equation}\nonumber
    t\theta(z)=t\theta\left(-2+\sqrt{3}\right)+\frac{4}{3}\tilde{k}^3+\tilde{s}\tilde{k}+\mathcal{O}\left(t^{-1/3}\tilde{k}^4\right),
\end{equation}
where $\tilde{s}$ is defined in \eqref{4.46} and
\begin{equation}\nonumber
    \tilde{k}=\left[\frac{9}{8}\left(26+15\sqrt{3}\right)t\right]^{1/3}\left(z-\left(-2+\sqrt{3}\right)\right).
\end{equation}
Follow the steps we just performed above, we
could approximate $M^{rhp3}(\tilde{k})$
\begin{equation}\nonumber
    M^{rhp3}(\tilde{k})=I+\frac{M^{rhp3}_1}{\tilde{k}}+\mathcal{O}(\tilde{k}^{-2}),
\end{equation}
where as $t\to+\infty,$

\begin{equation}\label{4.85}
    M_{1}^{rhp3}(\tilde{s})=\frac{i}{2}\begin{pmatrix}\int_{\tilde{s}}^{+\infty}v_{II}^{2}(\zeta)d\zeta&-v_{II}(\tilde{s})e^{-i\phi_{c}}\\
    v_{II}(\tilde{s})e^{i\phi_{c}}&-\int_{\tilde{s}}^{+\infty}v_{II}^{2}(\zeta)d\zeta\end{pmatrix}+\mathcal{O}\left(t^{-1/3+4\delta_2}\right),
\end{equation}
with the same $v_{II}$ in \eqref{4.75} and
\begin{equation}\label{phic}
    \begin{aligned}\phi_{c}(\tilde{s},t)&=\arg d\left(-2+\sqrt{3}\right)+2t\theta\left(-2+\sqrt{3}\right).\end{aligned}
\end{equation}
Here, we also need to use the fact that $r(z)=\overline{r(\bar{z}^{-1})}=r(-z^{-1})=-\overline{r(-\bar{z})}$,
which particularly implies that $|r(2+\sqrt{3})|=|r\left(-2+\sqrt{3}\right)|$.

For $z\in U^{(4)}$ and $t$ large enough, we have
\begin{equation}\nonumber
    t\theta(z)=t\theta\left(-2-\sqrt{3}\right)+\frac{4}{3}\breve{k}^3+\tilde{s}\breve{k}+\mathcal{O}\left(t^{-1/3}\breve{k}^4\right),
\end{equation}
where $\tilde{s}$ is defined in \eqref{4.46} and
\begin{equation}\nonumber
    \breve{k}=\left[\frac{9}{8}\left(26-15\sqrt{3}\right)t\right]^{1/3}\left(z-\left(-2-\sqrt{3}\right)\right).
\end{equation}
Following the steps we have just performed, we
could approximate $M^{rhp4}(\breve{k})$ as $\breve{k}\to\infty,$
\begin{equation}\nonumber
    M^{rhp4}(\breve{k})=I+\frac{M_{1}^{rhp4}}{\breve{k}}+\mathcal{O}(\breve{k}^{-2}),
\end{equation}
where as $t\to+\infty,$
\begin{equation}\label{4.90}
    M^{rhp4}_1(\tilde{s})=\frac{i}{2}\begin{pmatrix}\int_{\tilde{s}}^{+\infty}v_{II}^2(\zeta)d\zeta&-v_{II}
    (\tilde{s})e^{-\mathrm{i}\phi_d}\\v_{II}(\tilde{s})e^{i\phi_d}&-\int_{\tilde{s}}^{+\infty}v_{II}^2(\zeta)d\zeta\end{pmatrix}+\mathcal{O}\left(t^{-1/3+4\delta_2}\right),
\end{equation}
with the same $v_{II}$ in \eqref{4.75} and
\begin{equation}\label{phid}
    \begin{aligned}\phi_{d}(\tilde{s},t)&=\arg d\left(-2-\sqrt{3}\right)+2t\theta\left(-2-\sqrt{3}\right).\end{aligned}
\end{equation}
Here, we also need to use the fact that $r(z)=\overline{r(\bar{z}^{-1})}=r(-z^{-1})=-\overline{r(-\bar{z})}$,
which particularly implies that $|r(2+\sqrt{3})|=|r\left(-2-\sqrt{3}\right)|$.

\subsubsection{Small-norm RH problem}\label{s:4.3.2}
On the basis of \eqref{4.51} and the boundedness of $M^{out}(z)$ and $M^{rhpj}(z)$, the unknown error function
$E(z)$ can be shown as
\begin{equation}\label{4.92}
    E(z)=\begin{cases}
    M^{rhp}(M^{out}(z))^{-1},&z\notin\cup_{j\in\{a,b,c,d\}}U^{(j)}(z),\\
    M^{rhp}M^{rhp1}(z)^{-1}(M^{out}(z))^{-1},&z\in U^{(1)},\\
    M^{rhp}M^{rhp2}(z)^{-1}(M^{out}(z))^{-1},&z\in U^{(2)},\\
    M^{rhp}M^{rhp3}(z)^{-1}(M^{out}(z))^{-1},&z\in U^{(3)},\\
    M^{rhp}M^{rhp4}(z)^{-1}(M^{out}(z))^{-1},&z\in U^{(4)}.\end{cases}
\end{equation}
It is then readily seen that $E(z)$ satisfies the following RH problem.
\begin{RHP}\label{RHP4.12}
    		Find a matrix valued function $E(z)$ admits:
    		\begin{enumerate}[($i$)]
    			\item Analyticity:~$E(z)$ is analytical in $z\in\mathbb{C}\setminus\Sigma^{(E)}$, where
    \begin{equation}\nonumber
    \Sigma^{(E)}:=\cup_{j\in\{1,2,3,4\}}\partial U^{(j)}\cup(\Sigma^{(3)}\setminus\cup_{j\in\{1,2,3,4\}}U^{(j)}), ~~~~see~~~Fig.\ref{fig13};
\end{equation}

    			\item Jump condition:
    			\begin{equation}\nonumber
    E_+(z)=E_-(z)V^{(E)}(z),
\end{equation}
	where
\begin{equation}\nonumber
    \begin{aligned}V^{(E)}(z)&=\begin{cases}
    M^{out}(z)V^{(3)}(z)(M^{out}(z))^{-1},&z\in\Sigma^{(4)}\setminus(\cup_{j\in\{1,2,3,4\}}U^{(j)}),\\
    M^{out}M^{rhp1}(z)^{-1}(M^{out}(z))^{-1},&z\in\partial U^{(1)},\\
    M^{out}M^{rhp2}(z)^{-1}(M^{out}(z))^{-1},&z\in\partial U^{(2)},\\
    M^{out}M^{rhp3}(z)^{-1}(M^{out}(z))^{-1},&z\in\partial U^{(3)},\\
    M^{out}M^{rhp4}(z)^{-1}(M^{out}(z))^{-1},&z\in\partial U^{(4)},\end{cases}\end{aligned}
\end{equation}
where $V^{(3)}(z)$ is given by \eqref{4.39}.
    			\item Asymptotic behavior:
                \begin{equation}\nonumber
    E(z)=I+\mathcal{O}(z^{-1}).
\end{equation}

    		\end{enumerate}
    	\end{RHP}
\begin{figure}[h]
\centering
		\begin{tikzpicture}
		\draw[-latex,dashed](-6.5,0)--(6.5,0)node[right]{ Re$z$};

		\coordinate (I) at (0,0);
		\fill (I) circle (1pt) node[below] {$0$};

		\draw (-0.75,-0.25)--(0,-1);
		\draw [-latex](-0.75,-0.25)--(-0.2,-0.8);

		\draw (-0.75,0.25)--(0,1);
        \draw [-latex](-0.75,0.25)--(-0.2,0.8);
	
		\draw (0,1)--(0.75,0.25);
       \draw [-latex](0,1)--(0.5,0.5);
		
		\draw (0,-1)--(0.75,-0.25);
       \draw [-latex](0,-1)--(0.5,-0.5);
\draw [-latex](0,0)--(0,-1);
\draw [-latex](0,0)--(0,1);

\draw (2.05,0.2)--(3.1,1);
       \draw [-latex](2.05,0.2)--(2.45,0.5);
\draw (2.05,-0.2)--(3.1,-1);
       \draw [-latex](2.05,-0.2)--(2.45,-0.5);

       \draw (3.1,1)--(4.03,0.2);
       \draw [-latex](3.1,1)--(3.575,0.6);

       \draw [-latex](3.1,0)--(3.1,-1);
\draw [-latex](3.1,0)--(3.1,1);

           \draw (3.1,-1)--(4.03,-0.2);
       \draw [-latex](3.1,-1)--(3.575,-0.6);

        \draw (5.35,0.21)--(6.5,0.9);
       \draw [-latex](5.35,0.21)--(6,0.6);

               \draw (5.35,-0.21)--(6.5,-0.9);
       \draw [-latex](5.35,-0.21)--(6,-0.6);

\draw (-2.05,0.2)--(-3.1,1);
       \draw [-latex](-3.1,1)--(-2.45,0.5);
\draw (-2.05,-0.2)--(-3.1,-1);
       \draw [-latex](-3.1,-1)--(-2.45,-0.5);

       \draw (-3.1,1)--(-4.03,0.2);
       \draw [-latex](-4.03,0.2)--(-3.575,0.6);

       \draw [-latex](-3.1,0)--(-3.1,-1);
\draw [-latex](-3.1,0)--(-3.1,1);

           \draw (-3.1,-1)--(-4.03,-0.2);
       \draw [-latex](-4.03,-0.2)--(-3.575,-0.6);

        \draw (-5.35,0.21)--(-6.5,0.9);
       \draw [-latex](-6.5,0.9)--(-6,0.6);

               \draw (-5.35,-0.21)--(-6.5,-0.9);
       \draw [-latex](-6.5,-0.9)--(-6,-0.6);

		\coordinate (A) at (-5,0);
		\fill (A) circle (1pt) node[below] {$z_8$};
		\coordinate (b) at (-4.4,0);
		\fill (b) circle (1pt) node[below] {$z_7$};
		\coordinate (C) at (-1,0);
		\fill (C) circle (1pt) node[below] {$z_5$};
		\coordinate (d) at (-1.8,0);
		\fill (d) circle (1pt) node[below] {$z_6$};
		\coordinate (E) at (5,0);
		\fill (E) circle (1pt) node[below] {$z_1$};
		\coordinate (R) at (4.4,0);
		\fill (R) circle (1pt) node[below] {$z_2$};
		\coordinate (T) at (1,0);
		\fill (T) circle (1pt) node[below] {$z_4$};
		\coordinate (Y) at (1.8,0);
		\fill (Y) circle (1pt) node[below] {$z_3$};

        \draw[thick,blue] (1.4,0) circle (0.7);
        \draw[thick,blue] (-1.4,0) circle (0.7);
        \draw[thick,blue] (4.7,0) circle (0.7);
        \draw[thick,blue] (-4.7,0) circle (0.7);
        \node at (5,1) {\footnotesize $U^{(1)}$};
        \node at (1.5,1) {\footnotesize $U^{(2)}$};
        \node at (-1.5,1) {\footnotesize $U^{(3)}$};
        \node at (-5,1) {\footnotesize $U^{(4)}$};

		\end{tikzpicture}
	\caption{The jump contour $\Sigma^{(E)}$.}
	\label{fig13}
\end{figure}
A simple calculation shows that
    \begin{equation}\nonumber
    \parallel V^{(E)}(z)-I\parallel_p=\begin{cases}\mathcal{O}(\exp\left\{-ct^{3\delta_2}\right\}),&z\in\Sigma^{(E)}\setminus(U^{(1)}\cup U^{(2)}\cup U^{(3)}\cup U^{(4)}),\\ \mathcal{O}(t^{-\kappa_p}),&z\in\partial U^{(1)}\cup\partial U^{(2)}\cup\partial U^{(3)}\cup\partial\partial U^{(4)},\end{cases}
\end{equation}
for some positive $c$ and $\kappa_{p}=\frac{p-1}{p}\delta_{2}+\frac{1}{3p}$.
Subsequently, by virtue of the small norm
RH problem theory \cite{boo-49}, it is evident that for significantly large positive $t$, there exists a unique solution to RH problem \ref{RHP4.12}. Moreover, according to \cite{boo-50}, we have
\begin{equation}\label{4.98}
    E(z)=I+\frac{1}{2\pi i}\int_{\Sigma^{(E)}}\frac{\mu_{E}(\zeta)(V^{(E)}(\zeta)-I)}{\zeta-z}d\zeta,
\end{equation}
where $\mu_{E}\in I+L^2(\Sigma^{(E)})$ is the unique solution of the Fredholm-type equation
\begin{equation}\label{4.99}
    \mu_{E}=I+C_E\mu_{E}.
\end{equation}
Here, $C_E{:}L^2(\Sigma^{(E)})\to L^2(\Sigma^{(E)})$ is an integral operator defined by $C_E(f)(z)=C_-\left(f(V^{(E)}(z)-I)\right)$
with $C_{-}$ being the Cauchy projection operator on $\Sigma^{(E)}$. Thus,
\begin{equation}\nonumber
    \parallel C_E\parallel_{L^2(\Sigma^E)}\leqslant\parallel C_-\parallel_{L^2(\Sigma^E)}\parallel V^{(E)}(z)-I\parallel_{L^\infty(\Sigma^E)}\lesssim t^{-\delta_2},
\end{equation}
which guarantees the existence of the resolvent operator $(1-C_E)^{-1}$. So, $\mu_{E}$ exists uniquely with
\begin{equation}\nonumber
    \mu_{E}=I+(1-C_E)^{-1}(C_EI).
\end{equation}
On the other hand, \eqref{4.99} can be rewritten as
\begin{equation}\nonumber
    \mu_{E}=I+\sum_{j=1}^4C_E^jI+(1-C_E)^{-1}(C_E^5I),
\end{equation}
where for $j=1,\cdots,4,$ we have the estimates
\begin{equation}\nonumber
    \|C_E^jI\|_{L^2(\Sigma^E)}\lesssim\|C_E^{j-1}I\|_{L^2(\Sigma^E)}\parallel C_E\parallel_{L^2(\Sigma^E)}\lesssim t^{-(j-1)\delta_2}t^{-(1/6+\delta_2/2)}  \lesssim t^{-(1/6+j\delta_2-\delta_2/2)},
\end{equation}

\begin{equation}\label{4.104}
    \|\mu_{E}-I-\sum_{j=1}^4C_E^jI\parallel_{L^2(\Sigma^E)}\lesssim\frac{\|C_E^5I\parallel_{L^2(\Sigma^E)}}{1-\|C_E\|}\lesssim t^{-(1/6+9\delta_2/2)}.
\end{equation}
In order to retrieve the potential using the reconstruction formula \eqref{2.35}, a thorough investigation of the properties of the function $E(z)$ at the points of $i$ is of utmost importance. As a subsequent step, we will conduct an expansion of the function $E(z)$ at $z=i$. By \eqref{4.98}, it is readily seen that
\begin{equation}\label{4.105}
    E(z)=E(i)+E_1(z-i)+O((z-i)^2),\quad z\to i,
\end{equation}
where
\begin{equation}\label{4.106}
    E(i)=I+\frac{1}{2\pi i}\int_{\Sigma^{(E)}}\frac{\mu_{E}(\zeta)(V^{(E)}(\zeta)-I)}{\zeta-i}d\zeta,
\end{equation}

\begin{equation}\label{4.107}
    E_1=\frac{1}{2\pi i}\int_{\Sigma^{(E)}}\frac{\mu_{E}(\zeta)(V^{(E)}(\zeta)-I)}{(\zeta-i)^2}d\zeta.
\end{equation}

\begin{proposition}\label{prop4.13}
With $E_{1}$ and $E(i)$ defined in \eqref{4.106}-\eqref{4.107}, we have, as $t\to\infty$,
\begin{equation}\label{4.108}
    E(i)=I+t^{-1/3}N^{(0)}
    +\mathcal{O}\left(t^{-2/3+4\delta_2}\right),
\end{equation}

\begin{equation}\label{4.109}
    E_{1}=t^{-1/3}N^{(1)}+\mathcal{O}\left(t^{-2/3+4\delta_2}\right),
\end{equation}
where
\begin{align}
 N^{(0)}&=-\left(\frac{9}{8}(26-15\sqrt{3})\right)^{-1/3}\left(\frac{M^{out}(2+\sqrt{3})\hat{M}^{rhp1}_{1}(\tilde{s})(M^{out}(2+\sqrt{3}))^{-1}}{2+\sqrt{3}-i}\right)\nonumber \\
    &-\left(\frac{9}{8}(26-15\sqrt{3})\right)^{-1/3}\left(\frac{M^{out}(-2-\sqrt{3})\breve{M}^{rhp4}_{1}(\tilde{s})(M^{out}(-2-\sqrt{3}))^{-1}}{-2-\sqrt{3}-i}\right) \nonumber\\
    &-\left(\frac{9}{8}(26+15\sqrt{3})\right)^{-1/3}\left(\frac{M^{out}(2-\sqrt{3})\check{M}^{rhp2}_{1}(\tilde{s})(M^{out}(2-\sqrt{3}))^{-1}}{2-\sqrt{3}-i}\right)\label{N0}\\
    &-\left(\frac{9}{8}(26+15\sqrt{3})\right)^{-1/3}\left(\frac{M^{out}(-2+\sqrt{3})\tilde{M}^{rhp3}_{1}(\tilde{s})(M^{out}(-2+\sqrt{3}))^{-1}}{-2+\sqrt{3}-i}\right)\nonumber,
\end{align}
\begin{align}
N^{(1)}&=-\left(\frac{9}{8}(26-15\sqrt{3})\right)^{-1/3}\left(\frac{M^{out}(2+\sqrt{3})\hat{M}^{rhp1}_{1}(\tilde{s})(M^{out}(2+\sqrt{3}))^{-1}}{(2+\sqrt{3}-i)^2}\right)\nonumber\\
    &-\left(\frac{9}{8}(26-15\sqrt{3})\right)^{-1/3}\left(\frac{M^{out}(-2-\sqrt{3})\breve{M}^{rhp4}_{1}(\tilde{s})(M^{out}(-2-\sqrt{3}))^{-1}}{(-2-\sqrt{3}-i)^2}\right) \nonumber\\
    &-\left(\frac{9}{8}(26+15\sqrt{3})\right)^{-1/3}\left(\frac{M^{out}(2-\sqrt{3})\check{M}^{rhp2}_{1}(\tilde{s})(M^{out}(2-\sqrt{3}))^{-1}}{(2-\sqrt{3}-i)^2}\right)\label{N1}\\
    &-\left(\frac{9}{8}(26+15\sqrt{3})\right)^{-1/3}\left(\frac{M^{out}(-2+\sqrt{3})\tilde{M}^{rhp3}_{1}(\tilde{s})(M^{out}(-2+\sqrt{3}))^{-1}}{(-2+\sqrt{3}-i)^2}\right)\nonumber,
\end{align}

with
\begin{equation}\nonumber
    \hat{M}_1^{rhp1}=\frac{i}{2}\left(\begin{array}{cc}\int_{\tilde{s}}^{+\infty}v_{II}^{2}(\zeta)d\zeta&-v_{II}(\tilde{s})e^{i\phi_{a}}\\
    v_{II}(\tilde{s})e^{-i\phi_{a}}&-\int_{\tilde{s}}^{+\infty}v_{II}^{2}(\zeta)d\zeta\end{array}\right), \breve{M}^{rhp2}_1(\tilde{s})=\frac{i}{2}\begin{pmatrix}\int_{\tilde{s}}^{+\infty}v_{II}^{2}(\zeta)d\zeta&-v_{II}(\tilde{s})e^{i\phi_{b}}\\
    v_{II}(\tilde{s})e^{-i\phi_{b}}&-\int_{\tilde{s}}^{+\infty}v_{II}^{2}(\zeta)d\zeta\end{pmatrix},
\end{equation}
\begin{equation}\nonumber
    \tilde{M}_1^{rhp3}=\frac{i}{2}\left(\begin{array}{cc}\int_{\tilde{s}}^{+\infty}v_{II}^{2}(\zeta)d\zeta&-v_{II}(\tilde{s})e^{i\phi_{c}}\\
    v_{II}(\tilde{s})e^{-i\phi_{c}}&-\int_{\tilde{s}}^{+\infty}v_{II}^{2}(\zeta)d\zeta\end{array}\right), \breve{M}^{rhp4}_1(\tilde{s})=\frac{i}{2}\begin{pmatrix}\int_{\tilde{s}}^{+\infty}v_{II}^{2}(\zeta)d\zeta&-v_{II}(\tilde{s})e^{i\phi_{d}}\\
    v_{II}(\tilde{s})e^{-i\phi_{d}}&-\int_{\tilde{s}}^{+\infty}v_{II}^{2}(\zeta)d\zeta\end{pmatrix},
\end{equation}
and $M^{out}$, $\phi_{j}, j=a,b,c,d$ are given in Proposition \ref{prop4.5}, \eqref{4.62}, \eqref{4.81}, \eqref{phic} and \eqref{phid}.
\end{proposition}
\begin{proof}
From \eqref{4.79}, \eqref{4.85}, \eqref{4.90}, \eqref{4.104}, and \eqref{4.106}, it follows that
 \begin{align*}
 E(i)=&I+\sum_{j\in\{1,2,3,4\}}\oint_{\partial U^{(j)}}\frac{(M^{out}(\zeta)M^{rhpj}(\zeta)(M^{out}(\zeta))^{-1})^{-1}-I}{\zeta-i}+\mathcal{O}(t^{-1/3-5\delta_{2}})\nonumber\\
    =&I-\frac{1}{2\pi i}\oint_{\partial U^{(1)}}\frac{\left(\frac{9}{8}(26-15\sqrt{3})t\right)^{-1/3}M^{out}(\zeta)\hat{M}_1^{rhp1}(\tilde{s})(M^{out}(\zeta))^{-1}}{(\zeta-i)\left(\zeta-\left(2+\sqrt{3}\right)\right)}d\zeta\nonumber\\
    &-\frac{1}{2\pi i}\oint_{\partial U^{(2)}}\frac{\left(\frac{9}{8}(26+15\sqrt{3})t\right)^{-1/3}M^{out}(\zeta)\check{M}_1^{rhp2}(\tilde{s})(M^{out}(\zeta))^{-1}}{(\zeta-i)\left(\zeta-\left(2-\sqrt{3}\right)\right)}d\zeta \nonumber \\
    &~~~-\frac{1}{2\pi i}\oint_{\partial U^{(3)}}\frac{\left(\frac{9}{8}(26+15\sqrt3)t\right)^{-1/3}M^{out}(\zeta)\tilde{M}_1^{rhp3}(\tilde{s})(M^{out}(\zeta))^{-1}}{(\zeta-i)\left(\zeta-(-2+\sqrt3)\right)}d\zeta \nonumber\\
    &~~~-\frac{1}{2\pi i}\oint_{\partial U^{(4)}}\frac{\left(\frac{9}{8}(26-15\sqrt3)t\right)^{-1/3}M^{out}(\zeta)\breve{M}_1^{rhp4}(\tilde{s})(M^{out}(\zeta))^{-1}}{(\zeta-i)\left(\zeta-(-2-\sqrt3)\right)}d\zeta
    +\mathcal{O}(t^{-2/3+4\delta_2})+\mathcal{O}(t^{-1/3-5\delta_2}) \nonumber\\
    =&I-\left(\frac{9}{8}(26-15\sqrt{3})\right)^{-1/3}t^{-1/3}\left(\frac{M^{out}(2+\sqrt{3})\hat{M}^{rhp1}_{1}(\tilde{s})(M^{out}(2+\sqrt{3}))^{-1}}{2+\sqrt{3}-i}\right) \nonumber\\
    &-\left(\frac{9}{8}(26-15\sqrt{3})\right)^{-1/3}t^{-1/3}\left(\frac{M^{out}(-2-\sqrt{3})\breve{M}^{rhp4}_{1}(\tilde{s})(M^{out}(-2-\sqrt{3}))^{-1}}{-2-\sqrt{3}-i}\right) \nonumber\\
    &-\left(\frac{9}{8}(26+15\sqrt{3})\right)^{-1/3}t^{-1/3}\left(\frac{M^{out}(2-\sqrt{3})\check{M}^{rhp2}_{1}(\tilde{s})(M^{out}(2-\sqrt{3}))^{-1}}{2-\sqrt{3}-i}\right)\nonumber\\
    &-\left(\frac{9}{8}(26+15\sqrt{3})\right)^{-1/3}t^{-1/3}\left(\frac{M^{out}(-2+\sqrt{3})\tilde{M}^{rhp3}_{1}(\tilde{s})(M^{out}(-2+\sqrt{3}))^{-1}}{-2+\sqrt{3}-i}\right)
    +\mathcal{O}\left(t^{-2/3+4\delta_2}\right) \nonumber\\
    =&I+t^{-1/3}N^{(0)}
    +\mathcal{O}\left(t^{-2/3+4\delta_2}\right) \nonumber\\
    =&\eqref{4.108}.
    \end{align*}

For $E_{1}$, \eqref{4.79}, \eqref{4.85}, \eqref{4.90}, \eqref{4.104}, and \eqref{4.107}, it follows that
    \begin{align*}
 E_{1}=&\sum_{j\in\{1,2,3,4\}}\oint_{\partial U^{(j)}}\frac{(M^{out}(\zeta)M^{rhpj}(\zeta)(M^{out}(\zeta))^{-1})^{-1}-I}{(\zeta-i)^2}+\mathcal{O}(t^{-1/3-5\delta_{2}})\nonumber\\
    =&-\frac{1}{2\pi i}\oint_{\partial U^{(1)}}\frac{\left(\frac{9}{8}(26-15\sqrt{3})t\right)^{-1/3}M^{out}(\zeta)\hat{M}_1^{rhp1}(\tilde{s})(M^{out}(\zeta))^{-1}}{(\zeta-i)^{2}\left(\zeta-\left(2+\sqrt{3}\right)\right)}d\zeta\nonumber\\
    &-\frac{1}{2\pi i}\oint_{\partial U^{(2)}}\frac{\left(\frac{9}{8}(26+15\sqrt{3})t\right)^{-1/3}M^{out}(\zeta)\check{M}_1^{rhp2}(\tilde{s})(M^{out}(\zeta))^{-1}}{(\zeta-i)^{2}\left(\zeta-\left(2-\sqrt{3}\right)\right)}d\zeta \nonumber \\
    &~~~-\frac{1}{2\pi i}\oint_{\partial U^{(3)}}\frac{\left(\frac{9}{8}(26+15\sqrt3)t\right)^{-1/3}M^{out}(\zeta)\tilde{M}_1^{rhp3}(\tilde{s})(M^{out}(\zeta))^{-1}}{(\zeta-i)^{2}\left(\zeta-(-2+\sqrt3)\right)}d\zeta \nonumber\\
    &~~~-\frac{1}{2\pi i}\oint_{\partial U^{(4)}}\frac{\left(\frac{9}{8}(26-15\sqrt3)t\right)^{-1/3}M^{out}(\zeta)\breve{M}_1^{rhp4}(\tilde{s})(M^{out}(\zeta))^{-1}}{(\zeta-i)^{2}\left(\zeta-(-2-\sqrt3)\right)}d\zeta
    +\mathcal{O}(t^{-2/3+4\delta_2})+\mathcal{O}(t^{-1/3-5\delta_2}) \nonumber\\
    =&-\left(\frac{9}{8}(26-15\sqrt{3})\right)^{-1/3}t^{-1/3}\left(\frac{M^{out}(2+\sqrt{3})\hat{M}^{rhp1}_{1}(\tilde{s})(M^{out}(2+\sqrt{3}))^{-1}}{(2+\sqrt{3}-i)^2}\right) \nonumber\\
    &-\left(\frac{9}{8}(26-15\sqrt{3})\right)^{-1/3}t^{-1/3}\left(\frac{M^{out}(-2-\sqrt{3})\breve{M}^{rhp4}_{1}(\tilde{s})(M^{out}(-2-\sqrt{3}))^{-1}}{(-2-\sqrt{3}-i)^2}\right) \nonumber\\
    &-\left(\frac{9}{8}(26+15\sqrt{3})\right)^{-1/3}t^{-1/3}\left(\frac{M^{out}(2-\sqrt{3})\check{M}^{rhp2}_{1}(\tilde{s})(M^{out}(2-\sqrt{3}))^{-1}}{(2-\sqrt{3}-i)^2}\right)\nonumber\\
    &-\left(\frac{9}{8}(26+15\sqrt{3})\right)^{-1/3}t^{-1/3}\left(\frac{M^{out}(-2+\sqrt{3})\tilde{M}^{rhp3}_{1}(\tilde{s})(M^{out}(-2+\sqrt{3}))^{-1}}{(-2+\sqrt{3}-i)^2}\right)
    +\mathcal{O}\left(t^{-2/3+4\delta_2}\right) \nonumber\\
    =&t^{-1/3}N^{(1)}
    +\mathcal{O}\left(t^{-2/3+4\delta_2}\right) \nonumber\\
    =&\eqref{4.108}.
    \end{align*}
\end{proof}
\subsection{Asymptotic analysis on a pure $\bar{\partial}$-problem}\label{s:4.3}
Since we have successfully demonstrated the existence of the solution $M^{rhp}(z)$, we are now in a position to utilize $M^{rhp}(z)$ to transform $M^{(3)}(z)$ into a pure $\bar{\partial}$-problem which will be analyzed in this section. Define the function
\begin{equation}\label{4.111}
    M^{(4)}(z)=M^{(3)}(z)(M^{rhp}(z))^{-1}.
\end{equation}
By the properties of $M^{(3)}(z)$ and $M^{rhp}(z)$, we find $M^{(4)}(z)$ satisfies the following $\bar{\partial}$-problem.
\begin{Dbarproblem}\label{RHP4.14}
    		Find a matrix valued function $M^{(4)}(z)$ admits:
    		\begin{enumerate}[($i$)]
    			\item Continuity:~$M^{(4)}(z)$ is continuous and has sectionally continuous first partial derivatives in $\mathbb{C}$;
          \item Asymptotic behavior:
                \begin{equation}\nonumber
   M^{(4)}(z)=I+\mathcal{O}(z^{-1}).
\end{equation}
    			
    			 \item $\bar{\partial}$-Derivative:
    \begin{equation}\nonumber
    \bar{\partial}M^{(4)}(z)=M^{(4)}(z)W^{(3)}(z),z\in\mathbb{C},
\end{equation}
where
\begin{equation}\nonumber
    W^{(3)}(z)=M^{rhp}\bar{\partial}R^{(2)}(z)(M^{rhp})^{-1},
\end{equation}
and $\bar{\partial}R^{(2)}(z)$ has been given in \eqref{4.43}.
    		\end{enumerate}
    	\end{Dbarproblem}
$\bar{\partial}$-RH problem \ref{RHP4.14} is equivalent to the integral equation
    \begin{equation}\label{4.115}
    M^{(4)}(z)=I+\frac{1}{\pi}\iint_\mathbb{C}\frac{M^{(4)}(\zeta)W^{(3)}(\zeta)}{\zeta-z}d\mu(\zeta),
\end{equation}
which can be written as an operator equation
\begin{equation}\label{4.116}
    fS(z)=\frac{1}{\pi}\iint_\mathbb{C}\frac{f(\zeta)W^{(3)}(\zeta)}{\zeta-z}d\mu(\zeta).
\end{equation}
We could rewrite \eqref{4.115} in an operator form
\begin{equation}\nonumber
    M^{(4)}(z)=I\cdot(I-S)^{-1}.
\end{equation}
\begin{proposition}\label{prop4.15}
Let $S$ be the operator defined in \eqref{4.116}, we have
\begin{equation}\label{4.118}
    \parallel S\parallel_{L^\infty\to L^\infty}\lesssim t^{-1/3},\quad t\to+\infty,
\end{equation}
which implies that $\|(I-S)^{-1}\|$ is uniformly bounded for large positive $t$.
\end{proposition}
An immediate consequence of Proposition \ref{prop4.15} is that the pure $\bar{\partial}$-problem \ref{RHP4.14} for $M^{(4)}(z)$ admits
a unique solution for large positive $t$. For later use, we need the local behaviors of $M^{(4)}(z)$ at $z=i$ as $t\to+\infty$. By \eqref{4.115}, it is readily seen that
\begin{equation}\nonumber
    M^{(4)}(z)=M^{(4)}(i)+M_1^{(4)}(z-i)+O((z-i)^2),\quad z\to i,
\end{equation}
where
\begin{equation}\label{4.129}
    M^{(4)}(i)=I+\frac{1}{\pi}\iint_\mathbb{C}\frac{M^{(4)}(\zeta)W^{(3)}(\zeta)}{\zeta-i}d\mu(\zeta),
\end{equation}
\begin{equation}\label{4.130}
    M_1^{(4)}=\frac{1}{\pi}\iint_\mathbb{C}\frac{M^{(4)}(\zeta)W^{(3)}(\zeta)}{(\zeta-i)^2}d\mu(\zeta).
\end{equation}
\begin{proposition}\label{prop4.16}
With $M^{(4)}(i)$, $M_1^{(4)}$ defined in \eqref{4.129}-\eqref{4.130}, we have, as $t\to+\infty,$
\begin{equation}\nonumber
    |M^{(4)}(i)-I|,|M_1^{(4)}|\lesssim t^{-2/3}.
\end{equation}
\end{proposition}
\begin{proof}
The proof process is similar to Proposition \ref{prop3.18}, except that the defined values of $R_{j}$ and $R^{(2)}(z,\xi)$ in \eqref{4.29} and \eqref{4.35} are different.
\end{proof}

\subsection{Painlev\'{e} Asymptotics in the Transition Region $\mathcal{P}_{II}$}\label{s:4.4}

For transition region $\mathcal{P}_{II}$, inverting the sequence of transformations \eqref{3.7}, \eqref{4.37}, \eqref{4.92}, and \eqref{4.111}, we conclude that, as $t\to+\infty$,
\begin{equation}\nonumber
    M(z)=M^{(4)}(z)E(z)M^{out}(z)R^{(2)}(z)^{-1}T(z)^{-\sigma_3}G(z)^{-1}+\mathcal{O}(e^{-ct}),
\end{equation}
where $c$ is a constant, and $E(z)$, $R^{(2)}(z)$, $T(z)$ and $G(z)$ are defined in \eqref{4.92}, \eqref{4.35}, \eqref{2.40}, and \eqref{3.8}, respectively.
As $G(z)=R^{(2)}(z)=I$ on a neighborhood of $i$, and $T(z)=T(i)\left(1-I_0(z-i)\right)+\mathcal{O}((z-i)^2)$ with
\begin{equation}\nonumber
I_0=\frac{1}{2\pi i}\int_{\mathbb{R}}\frac{\log(1+|r(s)|^2)}{(s-i)^2}ds.\end{equation}
We can obtain that as $z\to i$
\begin{equation}\nonumber
\begin{aligned}M(z)=&\left(I+N^{(0)}t^{-1/2}+E_1(z-i)\right)\left(M_\Lambda^{out}(i)+M_{\Lambda,1}^{out}(z-i)\right)\\
&\times T(i)^{-\sigma_3}\left(I-I_0\sigma_3(z-i)\right)+O((z-i)^2)+\mathcal{O}(t^{-2/3+4\delta_2})+\mathcal{O}(t^{-2/3}),
    \end{aligned}
\end{equation}
\begin{equation}\nonumber
    M(i)=\left(I+N^{(0)}t^{-1/2}\right)\left(M_\Lambda^{out}(i)\right)T(i)^{-\sigma_3}+\mathcal{O}(t^{-2/3+4\delta_2}),
\end{equation}
where $E(i)$ and $E_1$ are given in \eqref{4.108} and \eqref{4.109}, respectively. Here, the error term $\mathcal{O}(t^{-2/3})$ comes
from the pure $\bar{\partial}$-problem. From the reconstruction formula \eqref{2.35}, we then obtain the Painlev\'{e} asymptotics of the mCH equation in the transition region $\mathcal{P}_{II}$
\begin{equation}\nonumber
u(y,t)=u_{p}(y,t;\tilde{\mathcal{D}}_\Lambda)+h_{11}t^{-1/3}+\mathcal{O}(t^{-2/3+4\delta_2})\end{equation}
\begin{equation}\nonumber
x(y,t)=y-2\ln\left(T(i)\right)+c_+^{out}(y,t;\tilde{\mathcal{D}}_\Lambda)+h_{12}t^{-1/3}+\mathcal{O}(t^{-2/3+4\delta_2}),\end{equation}
where $T(i)$, $u_{p}(y,t;\tilde{\mathcal{D}}_\Lambda)$ and $c_+^{out}(y,t;\tilde{\mathcal{D}}_\Lambda)$ are show in \eqref{4.5}, \eqref{up2} and \eqref{c+2},
\begin{align}
h_{11}(y,t)=&-u_{p}(y,t;\tilde{\mathcal{D}}_{\Lambda})\left(\sum_{j,k=1,2}[N^{(0)}]_{jk}\right)+\sum_{j,k=1,2}[N^{(1)}]_{jk}\nonumber\\
&+\left([N^{(0)}]_{11}+[N^{(0)}]_{21}\right)\frac{[M_{\Lambda,1}^{out}]_{11}+T(i)^{2}e^{c_{+}^{out}(y,t;\tilde{\mathcal{D}}_{\Lambda})}[M_{\Lambda,1}^{out}]_{12}}
{[M_{\Lambda}^{out}]_{11}(i)+[M_{\Lambda}^{out}]_{21}(i)}\nonumber\\
&+\left([N^{(0)}]_{11}+[N^{(0)}]_{21}\right)I_0(\xi)\frac{[M_{\Lambda}^{out}]_{11}(i)-T(i)^2e^{c_{+}^{out}(y,t;\tilde{\mathcal{D}}_{\Lambda})}
[M_{\Lambda}^{out}]_{21}(i)}{T(i)e^{1/2c_{+}^{out}(y,t;\tilde{\mathcal{D}}_{\Lambda})}([M_{\Lambda}^{out}]_{11}(i)+[M_{\Lambda}^{out}]_{21}(i))}\nonumber\\
&+\left([N^{(0)}]_{12}+[N^{(0)}]_{22}\right)\frac{[M_{\Lambda,1}^{out}]_{21}+T(i)^{2}e^{c_{+}^{out}(y,t;\tilde{\mathcal{D}}_{\Lambda})}
[M_{\Lambda,1}^{out}]_{22}}{[M_{\Lambda}^{out}]_{11}(i)+[M_{\Lambda}^{out}]_{21}(i)}\label{h11}\\
&+\left([N^{(0)}]_{12}+[N^{(0)}]_{22}\right)I_0(\xi)\frac{[M_{\Lambda}^{out}]_{21}(i)-T(i)^2e^{c_{+}^{out}(y,t;\tilde{\mathcal{D}}_{\Lambda})}
[M_{\Lambda}^{out}]_{11}(i)}{T(i)e^{1/2c_{+}^{out}(y,t;\tilde{\mathcal{D}}_{\Lambda})}([M_{\Lambda}^{out}]_{11}(i)+[M_{\Lambda}^{out}]_{21}(i))}\nonumber,
\end{align}
\begin{equation}\label{h12}
h_{12}=\left([N^{(0)}]_{11}+[N^{(0)}]_{21}-[N^{(0)}]_{12}-[N^{(0)}]_{22}\right)([M_\Lambda^{out}]_{21}(i)-[M_\Lambda^{out}]_{11}(i)),\end{equation}
and $N^{(0)}$, $N^{(1)}$ and $M_{\Lambda}^{out}$ are show in \eqref{N0}, \eqref{N1} and \eqref{4.21}.
\section{Painlev\'{e} Asymptotics of the solution for the mCH equtaion}\label{s:5}
In this section, we will revert the asymptotic results of the mCH equtaion \eqref{def:mCHeq} from $u(y,t)$ back to the $(x,t)$-plane.

For the transition region $\mathcal{P}_{I}$,
\begin{equation}\nonumber
u(y,t)=u_{p}(y,t;\tilde{\mathcal{D}}_\Lambda)+f_{11}(y,t)t^{-1/3}+\mathcal{O}(t^{-2/3+4\delta_1}),\end{equation}
\begin{equation}\nonumber
x(y,t)=y-2\ln\left(T(i)\right)+c_+^{out}(y,t;\tilde{\mathcal{D}}_\Lambda)+f_{12}t^{-1/3}+\mathcal{O}(t^{-2/3+4\delta_1}),\end{equation}
where $u_{p}(y,t;\tilde{\mathcal{D}}_\Lambda)$ and $f_{11}(y,t)$ are show in \eqref{ur} and \eqref{f11}. To express the asymptotic of $u(y,t)$ in the $(x,t)$ variables, we apply the strategy in \cite{boo-38} to obtain
\begin{equation}\nonumber
x=y-2\ln\left(T(i)\right)+c_+^{out}(y,t;\tilde{\mathcal{D}}_\Lambda)+f_{12}t^{-1/3}+\mathcal{O}(t^{-2/3+4\delta_1}),
\end{equation}
where $c_+^{out}(y,t;\tilde{\mathcal{D}}_\Lambda)$ is represented by soliton solutions show in \eqref{c+}, therefore
\begin{equation}\label{5.3}
\frac{y}{t}=\frac{x}{t}+\frac{2\ln\left(T(i)\right)-c_+^{out}(x,t;\tilde{\mathcal{D}}_\Lambda)-f_{12}t^{-1/3}}{t}+\mathcal{O}(t^{-5/3+4\delta_1}).
\end{equation}
Since $T(i)$ and soliton solutions is bounded, we have $\theta(z,y,t)=\theta(z,x,t)+\mathcal{O}(t^{-5/3+4\delta_1})$, where
\begin{equation}\nonumber
\theta(z;x,t)=-\frac{1}{4}(z-z^{-1})\left[\frac{x}{t}+\frac{2\ln\left(T(i)\right)-c_+^{out}(x,t;\tilde{\mathcal{D}}_\Lambda)
-f_{12}t^{-1/3}}{t}-\frac{8}{(z+z^{-1})^2}\right].\end{equation}
Combining with Corollary \ref{prop3.6} and \eqref{342}, this will further lead to the following relationship
\begin{equation}\nonumber
u_{p}(y,t;\tilde{\mathcal{D}}_\Lambda)=u_{p}(x,t;\tilde{\mathcal{D}}_\Lambda)+\mathcal{O}(t^{-2/3+4\delta_1}),
\end{equation}
with
\begin{equation}\label{5.4}
\begin{gathered}
u_{p}(x,t;\tilde{\mathcal{D}}_\Lambda)=\left[\sum_{k=1}^{\mathcal{N}}\left(\frac{-\overline{\alpha_{k}(x,t)}}{(i-\bar{\zeta}_{j_{k}})^{2}}
+\frac{\overline{\beta_{k}(x,t)}}{(i-\bar{\zeta}_{j_{k}})^{2}}\right)\right]/\left[1+\sum_{k=1}^{\mathcal{N}}
\left(\frac{-\overline{\alpha_{k}(x,t)}}{i-\bar{\zeta}_{j_{k}}}+\frac{\overline{\beta_{k}}(x,t)}{i-\bar{\zeta}_{j_{k}}}\right)\right]\\
+\left[\sum_{k=1}^{\mathcal{N}}\frac{\beta_k(x,t)}{(i-\zeta_{j_k})^2}+\frac{\alpha_k(x,t)}{(i-\zeta_{j_k})^2}\right]/
\left[1+\sum_{k=1}^{\mathcal{N}}\left(\frac{\beta_k(x,t)}{i-\zeta_{j_k}}+\frac{\alpha_k(x,t)}{i-\zeta_{j_k}}\right)\right].
\end{gathered}
\end{equation}
Since
\begin{align}\label{1.2b}
		s=6^{-2/3}\left(\frac{x}{t}-2\right)t^{2/3},
		\end{align}
\begin{align}\label{1.2b}
		\tilde{s}=6^{-2/3}\left(\frac{y}{t}-2\right)t^{2/3},
		\end{align}
therefore, $\tilde{s}-s=\mathcal{O}(t^{-1/3})$. Replacing $\tilde{s}$ by $s$ in \eqref{H0} and \eqref{H1}, taking the term $\int_{\tilde{s}}^{+\infty}v^2(\zeta)d\zeta$ as an example, we have
\begin{equation}\nonumber
\int_{\tilde{s}}^{+\infty}v^{2}(\zeta)d\zeta-\int_{s}^{+\infty}v^{2}(\zeta)d\zeta=\int_{\tilde{s}}^{s}v^{2}(\zeta)d\zeta\lesssim (s-\tilde{s})\|v^{2}\|_{L^{\infty}}\lesssim t^{-1/3},
\end{equation}
thus
\begin{equation}\nonumber
H^{(0)}(y)=H^{(0)}(x)+\mathcal{O}(t^{-1/3});H^{(1)}(y)=H^{(1)}(x)+\mathcal{O}(t^{-1/3}).
\end{equation}
This means
\begin{equation}\nonumber
f_{11}(y)=f_{11}(x)+\mathcal{O}(t^{-1/3}),
\end{equation}
where
\begin{equation}\label{5.7}
\begin{aligned}
f_{11}(x)=&-u_{p}(x,t;\tilde{\mathcal{D}}_{\Lambda})\left(\sum_{j,k=1,2}[H^{(0)}]_{jk}(x)\right)+\sum_{j,k=1,2}[H^{(1)}]_{jk}(x)\\
&+\left([H^{(0)}]_{11}(x)+[H^{(0)}]_{21}(x)\right)\frac{[M_{\Lambda,1}^{out}]_{11}+T(i)^{2}e^{c_{+}^{out}(x,t;\tilde{\mathcal{D}}_{\Lambda})}[M_{\Lambda,1}^{out}]_{12}}
{[M_{\Lambda}^{out}]_{11}(i)+[M_{\Lambda}^{out}]_{21}(i)}\\
&+\left([H^{(0)}]_{12}(x)+[H^{(0)}]_{22}(x)\right)\frac{[M_{\Lambda,1}^{out}]_{21}+T(i)^{2}e^{c_{+}^{out}(x,t;\tilde{\mathcal{D}}_{\Lambda})}[M_{\Lambda,1}^{out}]_{22}}
{[M_{\Lambda}^{out}]_{11}(i)+[M_{\Lambda}^{out}]_{21}(i)}.\end{aligned}\end{equation}
Therefore, we get
\begin{equation}\nonumber
u(x,t)=u_{p}(x,t;\tilde{\mathcal{D}}_\Lambda)+f_{11}(x,t)t^{-1/3}+\mathcal{O}(t^{-2/3+4\delta_1}).
\end{equation}
Similarly, the same treatment can be done for the regions transition region $\mathcal{P}_{II}$. For the transition region $\mathcal{P}_{II}$,
\begin{equation}\nonumber
u(y,t)=u_{p}(y,t;\tilde{\mathcal{D}}_\Lambda)+h_{11}(y,t)t^{-1/3}+\mathcal{O}(t^{-2/3+4\delta_2}),\end{equation}
\begin{equation}\nonumber
x(y,t)=y-2\ln\left(T(i)\right)+c_+^{out}(y,t;\tilde{\mathcal{D}}_\Lambda)+h_{12}t^{-1/3}+\mathcal{O}(t^{-2/3+4\delta_2}),\end{equation}
where $u_{p}(y,t;\tilde{\mathcal{D}}_\Lambda)$ and $h_{11}(y,t)$ are show in \eqref{up2} and \eqref{h11}. To express the asymptotic of $u(y,t)$ in the $(x,t)$ variables, we apply the strategy in \cite{boo-38} to obtain
\begin{equation}\nonumber
x=y-2\ln\left(T(i)\right)+c_+^{out}(y,t;\tilde{\mathcal{D}}_\Lambda)+h_{12}t^{-1/3}+\mathcal{O}(t^{-2/3+4\delta_2}),
\end{equation}
where $c_+^{out}(y,t;\tilde{\mathcal{D}}_\Lambda)$ is represented by soliton solutions show in \eqref{c+2}, therefore
\begin{equation}\nonumber
\frac{y}{t}=\frac{x}{t}+\frac{2\ln\left(T(i)\right)-c_+^{out}(x,t;\tilde{\mathcal{D}}_\Lambda)-h_{12}t^{-1/3}}{t}+\mathcal{O}(t^{-5/3+4\delta_2}).
\end{equation}
Since $T(i)$ and soliton solutions is bounded, we have $\theta(z,y,t)=\theta(z,x,t)+\mathcal{O}(t^{-5/3+4\delta_2})$, where
\begin{equation}\nonumber
\theta(z,x,t)=-\frac{1}{4}(z-z^{-1})\left[\frac{x}{t}+\frac{2\ln\left(T(i)\right)-c_+^{out}(x,t;\tilde{\mathcal{D}}_\Lambda)-f_{12}t^{-1/3}}{t}-\frac{8}{(z+z^{-1})^2}\right].
\end{equation}
Combining with Corollary \ref{prop4.5} and \eqref{up2}, this will further lead to the following relationship
\begin{equation}\nonumber
u_{p}(y,t;\tilde{\mathcal{D}}_\Lambda)=u_{p}(x,t;\tilde{\mathcal{D}}_\Lambda)+\mathcal{O}(t^{-2/3+4\delta_2}),
\end{equation}
where
\begin{equation}\label{5.12}
\begin{gathered}
u_{p}(x,t;\tilde{\mathcal{D}}_\Lambda)=\left[\sum_{k=1}^{\mathcal{N}}\left(\frac{-\overline{\alpha_{k}(x,t)}}{(i-\bar{\zeta}_{j_{k}})^{2}}
+\frac{\overline{\beta_{k}(x,t)}}{(i-\bar{\zeta}_{j_{k}})^{2}}\right)\right]/\left[1+\sum_{k=1}^{\mathcal{N}}
\left(\frac{-\overline{\alpha_{k}(x,t)}}{i-\bar{\zeta}_{j_{k}}}+\frac{\overline{\beta_{k}}(x,t)}{i-\bar{\zeta}_{j_{k}}}\right)\right]\\
+\left[\sum_{k=1}^{\mathcal{N}}\frac{\beta_k(x,t)}{(i-\zeta_{j_k})^2}+\frac{\alpha_k(x,t)}{(i-\zeta_{j_k})^2}\right]/
\left[1+\sum_{k=1}^{\mathcal{N}}\left(\frac{\beta_k(x,t)}{i-\zeta_{j_k}}+\frac{\alpha_k(x,t)}{i-\zeta_{j_k}}\right)\right].
\end{gathered}
\end{equation}
Since
\begin{align}\label{1.2b}
		s=-\left(\frac{8}{9}\right)^{1/3}\left(\frac{x}{t}+\frac{1}{4}\right)t^{2/3},
		\end{align}
\begin{align}\label{1.2b}
		\tilde{s}=-\left(\frac{8}{9}\right)^{1/3}\left(\frac{y}{t}+\frac{1}{4}\right)t^{2/3},
		\end{align}
therefore, $\tilde{s}-s=\mathcal{O}(t^{-1/3})$. Replacing $\tilde{s}$ by $s$ in \eqref{N0} and \eqref{N1}, taking the term $\int_{\tilde{s}}^{+\infty}v^2(\zeta)d\zeta$ as an example, we have
\begin{equation}\nonumber
\int_{\tilde{s}}^{+\infty}v^{2}(\zeta)d\zeta-\int_{s}^{+\infty}v^{2}(\zeta)d\zeta=\int_{\tilde{s}}^{s}v^{2}(\zeta)d\zeta\lesssim (s-\tilde{s})\|v^{2}\|_{L^{\infty}}\lesssim t^{-1/3},
\end{equation}
thus
\begin{equation}\nonumber
N^{(0)}(y)=N^{(0)}(x)+\mathcal{O}(t^{-1/3});N^{(1)}(y)=N^{(1)}(x)+\mathcal{O}(t^{-1/3}).
\end{equation}
This means
\begin{equation}\nonumber
h_{11}(y)=h_{11}(x)+\mathcal{O}(t^{-1/3}),
\end{equation}
where
\begin{equation}\label{h11x}
\begin{aligned}
h_{11}(x,t)=&-u_{p}(x,t;\tilde{\mathcal{D}}_{\Lambda})\left(\sum_{j,k=1,2}[N^{(0)}]_{jk}(x)\right)+\sum_{j,k=1,2}[N^{(1)}]_{jk}(x)\\
&+\left([N^{(0)}]_{11}(x)+[N^{(0)}]_{21}(x)\right)\frac{[M_{\Lambda,1}^{out}]_{11}+T(i)^{2}e^{c_{+}^{out}(x,t;\tilde{\mathcal{D}}_{\Lambda})}[M_{\Lambda,1}^{out}]_{12}}
{[M_{\Lambda}^{out}]_{11}(i)+[M_{\Lambda}^{out}]_{21}(i)}\\
&+\left([N^{(0)}]_{11}(x)+[N^{(0)}]_{21}(x)\right)I_0\frac{[M_{\Lambda}^{out}]_{11}(i)-T(i)^2e^{c_{+}^{out}(x,t;\tilde{\mathcal{D}}_{\Lambda})}
[M_{\Lambda}^{out}]_{21}(i)}{T(i)e^{1/2c_{+}^{out}(x,t;\tilde{\mathcal{D}}_{\Lambda})}([M_{\Lambda}^{out}]_{11}(i)+[M_{\Lambda}^{out}]_{21}(i))}\\
&+\left([N^{(0)}]_{12}(x)+[N^{(0)}]_{22}(x)\right)\frac{[M_{\Lambda,1}^{out}]_{21}+T(i)^{2}e^{c_{+}^{out}(x,t;\tilde{\mathcal{D}}_{\Lambda})}
[M_{\Lambda,1}^{out}]_{22}}{[M_{\Lambda}^{out}]_{11}(i)+[M_{\Lambda}^{out}]_{21}(i)}\\
&+\left([N^{(0)}]_{12}(x)+[N^{(0)}]_{22}(x)\right)I_0\frac{[M_{\Lambda}^{out}]_{21}(i)-T(i)^2e^{c_{+}^{out}(x,t;\tilde{\mathcal{D}}_{\Lambda})}
[M_{\Lambda}^{out}]_{11}(i)}{T(i)e^{1/2c_{+}^{out}(x,t;\tilde{\mathcal{D}}_{\Lambda})}([M_{\Lambda}^{out}]_{11}(i)+[M_{\Lambda}^{out}]_{21}(i))}.
\end{aligned}\end{equation}
Therefore, we get
\begin{equation}\nonumber
u(x,t)=u_{p}(x,t;\tilde{\mathcal{D}}_\Lambda)+h_{11}(x,t)t^{-1/3}+\mathcal{O}(t^{-2/3+4\delta_2}).
\end{equation}
\begin{theorem}\label{the1.3}
Let initial data $u(x,0)\in H^{4,2}(\mathbb{R})$ and $\left\{r(z),\{\zeta_n,C_n\}_{n=1}^{4N_1+2N_2}\right\}$ be the scattering data generated by the initial data $u(x,0)$. Then as $t \to +\infty$, the solution of the mCH equation \eqref{def:mCHeq} in transition regions $\mathcal{P}_{I}$ and $\mathcal{P}_{II}$ given in \eqref{P} can be described as follow:
	\begin{enumerate}[($i$)]
       \item if $\tau\triangleq x/t\in\mathcal{P}_{I}$, then
		\begin{equation}
u(x,t)=u_{p}(x,t;\tilde{\mathcal{D}}_\Lambda)+f_{11}(x,t)t^{-1/3}+\mathcal{O}(t^{-2/3+4\delta_1}),
\end{equation}
where $\frac{1}{27}<\delta_1<\frac{1}{12}$, $u_{p}(x,t;\tilde{\mathcal{D}}_\Lambda)$ is given by \eqref{5.4}, and $f_{11}(x,t)$ is given by \eqref{5.7} which is expressed in terms of the unique solution $v(s)$ of the Painlev\'{e} II equation
\begin{equation}\label{v}
	     v''(s)=sv(s)+2v^3(s),
        \end{equation}
where $s=6^{-2/3}\left(\frac{x}{t}-2\right)t^{2/3}$,
		\begin{align}
 v(s)\to -r(1)\textnormal{Ai}(s) \qquad s \to +\infty,
		\end{align}
		with $\textnormal{Ai}(s)$ being the classical Airy function and $r(1)\in[-1,1]$.
       \item if $\tau\in\mathcal{P}_{II}$, then
		\begin{equation}
u(x,t)=u_{p}(x,t;\tilde{\mathcal{D}}_\Lambda)+h_{11}(x,t)t^{-1/3}+\mathcal{O}(t^{-2/3+4\delta_2}),
\end{equation}
where $\frac{1}{27}<\delta_2<\frac{1}{12}$, $u_{p}(x,t;\tilde{\mathcal{D}}_\Lambda)$ is given by \eqref{5.12}, and $f_{11}(x,t)$ is given by \eqref{h11x} which is expressed in terms of the unique solution $v_{II}(s)$ of the Painlev\'{e} II equation
\begin{equation}\label{vII}
	     v_{II}''(s)=sv_{II}(s)+2v_{II}^3(s),
        \end{equation}
where $s=-\left(\frac{8}{9}\right)^{1/3}\left(\frac{x}{t}+\frac{1}{4}\right)t^{2/3}$,
		\begin{align}
 v_{II}(s)\to -\vert r(2+\sqrt{3})\vert \textnormal{Ai}(s), \qquad s \to +\infty,
		\end{align}
		with $\textnormal{Ai}(s)$ being the classical Airy function and $|r(2+\sqrt{3})|<1$.
	\end{enumerate}
\end{theorem}
The above results improve the asymptotic results of the mCH equation in the transition region.

\section*{Acknowledgments}\label{s:6} 	
    This work was supported by the National Natural Science Foundation of China under Grant No. 12371255, Xuzhou Basic Research Program Project under Grant No. KC23048, the Fundamental Research Funds for the Central Universities of CUMT under Grant No. 2024ZDPYJQ1003, and the Postgraduate Research \& Practice Program of Education \& Teaching Reform of CUMT under Grant No. 2025YJSJG031, the Graduate Innovation Program of the China University of Mining and Technology under Grant No. 2025WLKXJ142, and the Postgraduate Research \& Practice Innovation Program of Jiangsu Province under Grant No. KYCX25\_3010.\\

  \textbf{Data availibility}: The data which supports the findings of this study is available within the article.\\

  \textbf{Declarations}\\

\textbf{Conflict of interest}: The authors declare no conflict of interest.

\appendix{}

\section{The Painlev\'e II parametrix}\label{appendix}
A result due to Hastings and McLeod \cite{boo-48} asserts that, for any $\kappa \in \mathbb{R}$, there exists a unique solution to the homogeneous Painlev\'e II equation \eqref{v} which behaves like $\kappa \textnormal{Ai}(s)$ for large positive $s$. This one-parameter family of solutions are characterized by the following RH problem; cf. \cite{boo-41}.

\begin{RHP}\label{PII}
\hfill
\begin{itemize}
  \item $M^{P}(z; s,\kappa)$ is holomorphic for $z\in \mathbb{C}\setminus \Gamma$, where $\Gamma:=\cup_{i=1,3,4,6}\Gamma_i$ with
  \begin{equation}\nonumber
		\Gamma_i:=\left\{z\in\mathbb{C}: \arg z=\frac{\pi}{6}+\frac{\pi}{3}(i-1)\right\};
	\end{equation}
see Fig.\ref{fig14} for an illustration.
  \item $M^{P}$ satisfies the jump condition
  $$
  M^{P}_{+}(z; s,\kappa)=M^{P}_{-}(z; s,\kappa)e^{-i\left(\frac{4}{3}z^3+sz\right)\hat{\sigma}_3}S_i, \quad z \in \Gamma_i\backslash\{0\}, \quad i=1,3,4,6,
  $$
where the matrix $S_i$ depending on $\kappa$ for each ray $\Gamma_i$ is shown in Fig.\ref{fig14}.
  \item As $z\to \infty$ in $\mathbb{C} \setminus \Gamma$, we have $M^{P}(z;s,\kappa)=I+\mathcal{O}(z^{-1})$.
   \item As $z\to 0 $, we have $M^{P}(z;s,\kappa)=\mathcal{O}(1)$.
\end{itemize}

\end{RHP}
\begin{figure}[h]
	\centering
	\begin{tikzpicture}[node distance=2cm]
		\draw[dash pattern={on 0.84pt off 2.51pt}][->](-3.6,0)--(4,0);
		\draw[dash pattern={on 0.84pt off 2.51pt}][->](0,-1.8)--(0,1.8);

        \draw (-1.5,-1.5) -- (1.5,1.5);
      \draw (-1.5,1.5) -- (1.5,-1.5);
		
\draw [-latex](0,0) -- (1,1);
\draw [-latex](0,0) -- (-1,-1);

\draw [-latex](0,0) -- (-1,1);

\draw [-latex](0,0) -- (1,-1);

\draw (1.5,1.3) node [anchor=north west][inner sep=0.75pt]  [font=\scriptsize]  {$\Gamma_{1}$};
\draw (-1.8,1.3) node [anchor=north west][inner sep=0.75pt]  [font=\scriptsize]  {$\Gamma_{3}$};
\draw (1.5,-1.3) node [anchor=north west][inner sep=0.75pt]  [font=\scriptsize]  {$\Gamma_{4}$};
\draw (-1.8,-1.2) node [anchor=north west][inner sep=0.75pt]  [font=\scriptsize]  {$\Gamma_{6}$};

\draw (1.5,2.5) node [anchor=north west][inner sep=0.75pt]  {$S_{1} =\begin{pmatrix}
1 & 0\\
-\kappa i & 1
\end{pmatrix}$};

\draw (-4,-1.5) node [anchor=north west][inner sep=0.75pt]    {$S_{4} =\begin{pmatrix}
1 & \kappa i\\
0 & 1
\end{pmatrix}$};

\draw (-4,2.5) node [anchor=north west][inner sep=0.75pt]    {$S_{3} =\begin{pmatrix}
1 & 0\\
\kappa i & 1
\end{pmatrix}$};

\draw (1.5,-1.5) node [anchor=north west][inner sep=0.75pt]    {$S_{6} =\begin{pmatrix}
1 & -\kappa i \\
0 & 1
\end{pmatrix}$};
		
		\coordinate (I) at (0.2,0);
		\fill (I) circle (0pt) node[below] {$0$};

	\end{tikzpicture}
	\caption{ The jump contours $\Gamma_i$. }
	\label{fig14}
\end{figure}

The above RH problem admits a unique solution. Moreover, there exist smooth functions $\{M^{p}_j(s)\}_{j=1}^{\infty}$ such that, for each $j \geqslant 1$, we have
\begin{equation}\nonumber
	M^{P}(s,z)=I+\sum_{j=1}^{N}\frac{M^{P}_j(s)}{z^j}+\mathcal{O}(z^{-N-1}),\quad z\rightarrow\infty,
\end{equation}
where
\begin{equation}\label{A3}
    M_1^P(s)=\frac{1}{2}\begin{pmatrix}-i\int_s^{+\infty}v^2(\zeta)d\zeta&v(s)\\v(s)&i\int_s^{+\infty}v^2(\zeta)d\zeta\end{pmatrix}.
\end{equation}

The function $v$ in \eqref{A3} then solves the Painlev\'e II equation \eqref{v} with the boundary condition
\begin{equation}\label{A5}
	v(s)\sim \kappa\textnormal{Ai}(s), \qquad s\to +\infty.
\end{equation}
We call RH problem \ref{PII} the Painlev\'e II parametrix.

	\bibliographystyle{plain}

\end{document}